\documentclass{article}[11pt]
\usepackage{latexsym}
\usepackage{amsmath}
\usepackage{amsfonts}
\usepackage{amssymb}
\usepackage{mathrsfs}
\usepackage{graphicx}	
\usepackage{amsthm}
\usepackage[top=2.5cm, bottom=2.5cm, left=2.5cm, right=2.5cm]{geometry}
\usepackage{hyperref}
\usepackage{slashed,mathtools,color}
\usepackage{upgreek}
\usepackage{bbm}
\usepackage{comment}
\usepackage{xcolor}

\usepackage[normalem]{ulem}

\newcommand{\beq}{\begin{equation}}
\newcommand{\eeq}{\end{equation}}



\theoremstyle{plain}
\newtheorem{theorem}{Theorem}[section]
\newtheorem{proposition}[theorem]{Proposition}
\newtheorem{lemma}[theorem]{Lemma}

\theoremstyle{definition}

\newtheorem{remark}[theorem]{Remark}

\numberwithin{equation}{section}

\allowdisplaybreaks[4] 

\swapnumbers
\begin{document}

\title{Future stability of perfect fluids with extreme tilt and linear equation of state $p=c_s^2\rho$ for the Einstein-Euler system with positive cosmological constant: The range $\frac{1}{3}<c_s^2<\frac{3}{7}$}

\author{Grigorios Fournodavlos,\footnote{Department of Mathematics \& Applied Mathematics, University of Crete, Voutes Campus, 70013 Heraklion, Greece, gfournodavlos@uoc.gr.} \footnote{Institute of Applied and Computational Mathematics,
FORTH, 70013 Heraklion, Greece.}\;\; Elliot Marshall,\footnote{School of Mathematics,
9 Rainforest Walk,
Monash University, VIC 3800, Australia, elliot.marshall@monash.edu.}\;\; Todd A. Oliynyk\footnote{School of Mathematics,
9 Rainforest Walk,
Monash University, VIC 3800, Australia, todd.oliynyk@monash.edu.}}

\date{}

\maketitle

\begin{abstract}
We study the future stability of cosmological fluids, in spacetimes with an accelerated expansion, which exhibit extreme tilt behavior, ie. their fluid velocity becoming asymptotically null at timelike infinity. It has been predicted in the article \cite{LEUW} that the latter behavior is dominant for sound speeds beyond radiation $c_s=1/\sqrt{3}$, hence, bifurcating off of the stable orthogonal fluid behavior modeled by the classical FLRW family of solutions, for $c_s^2\in[0,\frac{1}{3}]$. First, we construct homogeneous solutions to the Einstein-Euler system with the latter behavior, in $\mathbb{S}^3$ spatial topology, for sound speeds $c_s^2\in(\frac{1}{3},1)$. Then, we study their future dynamics and prove a global stability result in the restricted range $c_s^2\in(\frac{1}{3},\frac{3}{7})$. In particular, we show that extreme tilt behavior persists to sufficiently small perturbations of the homogeneous backgrounds, without any symmetry assumptions or analyticity. Our method is based on a bootstrap argument, in weighted Sobolev spaces, capturing the exponential decay of suitable renormalized variables. Extreme tilt behavior is associated with a degeneracy in the top order energy estimates that we derive, which allows us to complete our bootstrap argument only in the aforementioned restricted range of sound speeds. Interestingly, this is a degeneracy that does not appear in the study of formal series expansions. Moreover, for the Euler equations on a fixed FLRW background, our estimates can be improved to treat the entire beyond radiation interval $c_s^2\in(\frac{1}{3},1)$, a result already obtained in \cite{MO}. The latter indicates that the former issue is related to the general inhomogeneous geometry of the perturbed metric in the coupled to Einstein case.
\end{abstract}

\tableofcontents

\parskip = 0 pt


%
\section{Introduction}

The standard model in cosmology is the Einstein-Euler system with a positive cosmological constant:
\begin{align}
\label{EE}
{\bf Ric}_{\mu\nu}-\frac{1}{2}{\bf g}_{\mu\nu}{\bf R}+\Lambda {\bf g}_{\mu\nu}=&\,T_{\mu\nu},\\
\label{eq.motion}{\bf D}^\mu T_{\mu\nu}=&\,0,
\end{align}
where $T_{\mu\nu}$ is the energy momentum tensor of a perfect fluid,
\begin{align}
\label{Tmunu}T_{\mu\nu}=(\rho+p)u_\mu u_\nu +p{\bf g}_{\mu\nu},
\end{align}
with linear equation of state $p=c_s^2\rho$ and sound speed in the interval $c_s\in[0,1]$. Different values of $c_s$ are used to model different states of matter in the universe, e.g. dust $c_s=0$, radiation $c_s=\frac{1}{\sqrt{3}}$, stiff $c_s=1$. Here, $p,\rho$ are the pressure and density respectively, $u$ is the unit fluid velocity, ${\bf g}(u,u)=-1$, and we have chosen geometric units so that in particular the speed of light is 1.

It is convenient to consider the renormalized fluid speed 
\begin{align}\label{v.rs}
v_\mu = \rho^{r_s}u_\mu,\qquad r_s=\frac{c_s^2}{1+c^2_s}\in[0,\frac{1}{2}]. 
\end{align}
Then, the equations of motion \eqref{eq.motion} read (Lemma \ref{lem:eq.motion}): 
\begin{align}
\label{divT.v}({\bf D}_v v)_\nu+\frac{1}{2}{\bf D}_\nu(\rho^{2r_s})=&\,0,\\
\label{divT.rho}(1-2r_s){\bf D}_v\log\rho+{\bf D}^\mu v_\mu=&\,0.
\end{align}

The initial data for \eqref{EE}, \eqref{divT.v}, \eqref{divT.rho} on a Cauchy hypersurface $\Sigma$ consist of ($\mathring{g},\mathring{k},\mathring{\rho},\mathring{\underline{v}},\mathring{v}_0$), satisfying the constraint equations: 
\begin{align}
\label{Hamconst.intro}\mathring{R}-|\mathring{k}|^2+(\text{tr}\mathring{k})^2=&\,2\Lambda+2(1+c_s^2)\mathring{\rho}^{1-2r_s}\mathring{v}_0^2-c_s^2\mathring{\rho},\\
\label{momconst.intro}\mathring{\text{div}}\mathring{k}-\text{d}\text{tr}\mathring{k}=&-(1+c_s^2)\mathring{\rho}^{1-2r_s}\mathring{v}_0\mathring{\underline{v}},
\end{align}
and the identity $\mathring{v}_0^2=\mathring{g}(\mathring{\underline{v}},\mathring{\underline{v}})+\mathring{\rho}^{2r_s}$.
It is well-known that the above system of equations is locally well-posed (see \cite[Section 3.2]{Sp1} for a detailed discussion), giving rise to a development $(\mathcal{M},{\bf g},\rho,v)$.

\subsection{The space of homogeneous solutions: Orthogonal vs tilted fluids, radiation bifurcation}\label{SS:hom.class}

There are numerous papers in the physics literature \cite{CW,GN,GE,MMBE,ME,RM,SR,W} that study the space of homogeneous solutions to the Einstein-Euler system \eqref{EE}-\eqref{momconst.intro}. Most relevant to the present paper is the work \cite{GE} (see also \cite{GN,RM}), where the authors, using dynamical systems techniques, notice the following classification of homogeneous solutions in various Bianchi symmetry classes (for the whole picture below to be valid, a certain number of non-trivial structure coefficients are required): 
\begin{enumerate}
\item The metric approaches a de Sitter-like state\label{class:1}
\begin{align}\label{gen.hom.metric.intro}
{\bf g}=-dt^2+g_{ij}dx^idx^j,\qquad\qquad (t,x)\in [T,+\infty)\times\Sigma\cong\mathcal{M}.
\end{align}
That is to say, the rate of expansion of the universe, $g_{ij}\sim e^{2Ht}c_{ij}$, is uniform in all spatial directions and is dictated by the Hubble constant, equal to leading order to $H=\sqrt{\Lambda/3}$.

\item \label{class:2} Asymptotically orthogonal fluids, $0\leq c_s< 1/\sqrt{3}$: For sound speeds below radiation, the fluid velocity becomes asymptotically orthogonal to the level sets of $t$, ie. $u=-u_0\partial_t+u_Ie_I$ with $u_0\rightarrow-1$, $u_I\rightarrow0$, as $t\rightarrow+\infty$, where $g(e_I,e_J)=\delta_{IJ}$.

\item Tilted fluids, $c_s=1/\sqrt{3}$: The fluid velocity can pick up a `tilt' at infinity, relative to the level sets of $t$, while all components of the fluid remain bounded, ie. there exist constants $u^\infty_0,u^\infty_I$, such that $u_0\rightarrow u^\infty_0$, $u_I\rightarrow u^\infty_I$, as $t\rightarrow+\infty$, where $u=-u_0\partial_t+u_Ie_I$ and $g(e_I,e_J)=\delta_{IJ}$.

\item  
\label{class:4} Fluids with extreme tilt, $1/\sqrt{3}<c_s<1$: The components of the fluid velocity become unbounded, while the fluid vector field $u\in T\mathcal{M}$ approaches the null cone at infinity. More precisely, there exist constants $u_0^\infty,u_I^\infty$ such that 
\begin{align}\label{gen.hom.tilt.intro}
e^{-A_sHt}u_0\overset{t\rightarrow+\infty}{\longrightarrow} u^\infty_0,\qquad e^{-A_sHt}u_I\overset{t\rightarrow+\infty}{\longrightarrow} u^\infty_I,\qquad A_s=\frac{3c_s^2-1}{1-c_s^2},\qquad(u^\infty_0)^2=u^\infty_Iu^\infty_I,
\end{align}
where $u=-u_0\partial_t+u_Ie_I$ and $g(e_I,e_J)=\delta_{IJ}$.
Note that $A_s>0$, for $c_s^2\in(\frac{1}{3},1)$. In other words, the leading order behavior of $u$ is a null vector field. 

\item \label{class:5} The density $\rho$ decays exponentially in all above cases, as $t\rightarrow+\infty$, at a rate which depends on $c_s$, compensating especially in case \ref{class:4} for the exponential growth of the fluid components $u_0,u_I$. 
\end{enumerate}
An immediate observation from the above classification is that there is a bifurcation phenomenon, regarding the behavior of the fluid at infinity, relative to the different sound speeds at exactly the radiation case $c_s=1/\sqrt{3}$. Of course, this is a behavior exhibited by homogeneous solutions and the natural question that arises is whether the latter phenomenon persists to inhomogeneous solutions.

\subsection{Analytic solutions without symmetries: Stability/instability predictions}

In \cite{LEUW}, the authors carried out a heuristic analysis of general, inhomogeneous, solutions to the Einstein-Euler system \eqref{EE}-\eqref{momconst.intro} based on formal series expansions from infinity. They concluded that the behavior of the fluid for the different sound speeds is consistent with the picture \ref{class:1}-\ref{class:5} in the homogeneous case, the difference being that the asymptotic data at infinity are now functions instead of constants, e.g. $c_{ij}:=c_{ij}(x)$, $u^\infty_0=u^\infty_0(x)$ etc. In particular, extreme tilt behavior \ref{class:4}, for sound speeds in the range $1/\sqrt{3}<c_s<1$, persists to inhomogeneous solutions enjoying all degrees of freedom, in a function counting sense. The latter is an indication that the obtained qualitative picture is generic. 

A rigorous construction of analytic solutions, without symmetries, was subsequently achieved by Rendall \cite{Ren} using Fuchsian techniques. Once more, the behavior of solutions is determined in detail by expansions from infinity to any order, in the spirit of \cite{LEUW}, containing all functional degrees of freedom. The only deviation from the heuristic picture becomes apparent in the range $1/\sqrt{3}<c_s<1$, for solutions having extreme tilt. In the latter regime, Rendall's result \cite{Ren} only applies to asymptotic data having non-vanishing spatial velocity, ie. $u^\infty_I(x)u^\infty_I(x)\neq0$. In the presence of a vanishing point, he predicted the blowup of the density contrast $\nabla\log\rho$, and hence a form of breakdown in the fluid, which in his words ``is reminiscent of spikes'' that are observed in Big Bang singularities. Numerical support for Rendall's breakdown prediction was recently given in \cite{BMO}, and rigorously proven in \cite{Ol3}, for a restricted range of sound speeds beyond radiation, under symmetry assumptions, albeit for the relativistic Euler equations on a fixed FLRW background and not coupled to Einstein. An interesting feature of the inhomogenous solutions of the relativistic Euler equations from \cite{Ol3} is that they behave like orthogonal fluids at points where the spatial velocity vanishes at timelike infinity while at points with non-vanishing spatial velocity they exhibit extreme tilt.

\subsection{Future stability of FLRW: The range up to radiation $0\leq c_s\leq 1/\sqrt{3}$}

The current work concerns the future stability of certain homogeneous background solutions in finite regularity spaces. To that extent, the simplest background solution one can consider is a member of the FLRW family of solutions. This is an orthogonal fluid and according to the predictions from the study of formal series expansions, it can only be expected to be stable for sound speeds in the range $c_s^2\in[0,\frac{1}{3}]$. 

Indeed, future stability of FLRW was first established for $c_s^2\in(0,\frac{1}{3})$ in the irrotational case \cite{RS} and subsequently in the presence of vorticity \cite{Sp1}, showing in particular that the perturbed solution becomes asymptotically orthogonal as well. Remarkably, the latter works demonstrated that an accelerated expansion can silence fluid degeneracies, like shocks, which are in general expected to occur otherwise.\footnote{See for example Christodoulou's breakthrough work \cite{Christ} on the formation of shocks for the relativistic Euler equations in Minkowski spacetime.}\footnote{Fluid stabilisation, as an effect of an accelerated expansion, was first rigorously established by Brauer-Rendall-Reula \cite{BRR} for Newtonian cosmological models.} The dust case $c_s=0$ was treated in \cite{HS}, where the fluid is asymptotically orthogonal as well, and the radiation case $c_s^2=\frac{1}{3}$ was treated via conformal methods \cite{LK}, where as in the heuristic picture, the fluid can pick up a slight tilt at infinity. An alternative approach that provides a uniform treatment of future stability in the parameter range $0<c_s^2\leq \frac{1}{3}$ can be found in \cite{Ol1}. For fluid stabilisation in higher dimensions, see the more recent \cite{M}.

\subsection{Main results: Going beyond the radiation threshold $c_s>1/\sqrt{3}$}

For sound speeds above radiation, ie. $c_s>1/\sqrt{3}$, FLRW is not the right background to consider, since the heuristic analysis \cite{LEUW} suggests that extreme tilt should be the dominant behavior, a feature qualitatively different from FLRW. In order to study the future stability of extreme tilt behavior in fluids, it is appropriate to consider homogeneous backgrounds that have this property.

\subsubsection{The tilted homogeneous backgrounds}

To allow for extreme tilt behavior in homogeneous solutions, a certain number of non-trivial structure coefficients and anisotropy is needed to satisfy the constraints. We find it convenient to look at homogeneous background metrics of the form 
\begin{align}\label{metric.hom}
\widetilde{\bf g}=-dt^2+\sum_{i=1}^3G^2_i(t)\psi^i\otimes \psi^i+G^2(t)(\psi^2\otimes \psi^3+\psi^3\otimes \psi^2),
\end{align}
defined over $[T,+\infty)\times\mathbb{S}^3$,
where $\psi^i$ are the 1-forms dual to the non-holonomic basis of spatial vector fields $Y_i$ defined in \eqref{Yi}.

Also, we assume that the homogeneous fluid speeds $\widetilde{u}$ are tilted only relative to the first spatial coordinate: 
\begin{align}\label{fluid.hom}
\widetilde{u}=\cosh(\theta(t))\partial_t+\sinh(\theta(t))G_1^{-1}(t)Y_1.
\end{align}
The existence of such solutions, in the range $c_s^2\in(\frac{1}{3},1)$, is obtained by an ODE analysis at infinity (see Appendix \ref{app:ODE}). In other words, a backwards construction, assuming that the metric \eqref{metric.hom} asymptotically isotropizes:
\begin{align}\label{metric.lead}
G_i^2(t)= (G_i^\infty)^2e^{2Ht}+\text{l.o.t.}\qquad G^2(t)=\text{l.o.t.},
\end{align}
for some constants $G_i^\infty>0$. $T>0$ is some sufficiently large time given by the interval of existence of the solution to the homogeneous Einstein-Euler system.

An important feature of the fluids that we consider is that both components of the fluid speed \eqref{fluid.hom} grow exponentially, while the fluid speed itself becomes asymptotically null, as $t\rightarrow+\infty$, in the sense that
\begin{align}\label{fluid.lead.null}
\cosh(\theta(t))= a_0 e^{\frac{3c_s^2-1}{1-c_s^2}Ht}+\text{l.o.t.},\qquad\cosh(\theta(t))= a_1 e^{\frac{3c_s^2-1}{1-c_s^2}Ht}+\text{l.o.t.},\qquad a_0^2=a_1^2\neq0,
\end{align}
for $\frac{1}{\sqrt{3}}<c_s<1$. The growth of the fluid components is compensated by the exponential decay of the density:
\begin{align}\label{rho.hom}
\widetilde{\rho}(t)= \widetilde{\rho}^\infty e^{-2\frac{1+c_s^2}{1-c_s^2}Ht}+\text{l.o.t.}
\end{align}
The precise estimates satisfied by the above variables and their derivatives are stated in Lemma \ref{lem:hom.est}, and derived in Appendix \ref{app:ODE}.

\subsubsection{Brief framework}\label{sss:framework}

We express all spacetime metrics in this article 
relative to a coframe, propagated along the gradient of a time function $t\in[T,+\infty)$: 
\begin{align}\label{metric}
{\bf g}=-n^2dt^2+g_{ij}\psi^i\psi^j.
\end{align}
Here $n$ is the lapse of the $t$-foliation, $\Sigma_t\cong\mathbb{S}^3$, the shift vector field is set to zero, and $g$ is the induced metric on $\Sigma_t$. When we compare variables associated to different metrics, e.g. a homogeneous solution and a perturbation of the latter, we use the natural identification of points in the different manifolds relative to the above splitting. In this regard, the time function of the perturbed solution is fixed by requiring that the lapse of its level sets equals
\begin{align}\label{parab.gauge}
n-1=\text{tr}\widetilde{k}-\text{tr}k,
\end{align}
where $\text{tr}\widetilde{k}$ is the mean curvature of $\Sigma_t$ in the homogeneous background, see \eqref{red.var.hom}.
As it turns out, the condition \eqref{parab.gauge} leads to a parabolic equation for the lapse $n$, see \eqref{n.eq}, which is well-posed in the future direction. 

The 1-forms $\psi^1,\psi^2,\psi^3$ are dual to the basis of $\Sigma_t$-tangent vector fields $Y_1,Y_2,Y_3$ that are defined as follows. Consider the basis of standard $d\mathbb{S}^3$-Killing vector fields
\begin{align}\label{Yi}
\notag Y_1=&\,(x^1\frac{\partial}{\partial x^2}-x^2\frac{\partial}{\partial x^1})\big|_{\mathbb{S}^3}+(x^3\frac{\partial}{\partial x^4}-x^4\frac{\partial}{\partial x^3})\big|_{\mathbb{S}^3},\\ 
Y_2=&\,(x^1\frac{\partial}{\partial x^4}-x^4\frac{\partial}{\partial x^1})\big|_{\mathbb{S}^3}
+(x^2\frac{\partial}{\partial x^3}-x^3\frac{\partial}{\partial x^2})\big|_{\mathbb{S}^3},\\
\notag Y_3=&\,(x^1\frac{\partial}{\partial x^3}-x^3\frac{\partial}{\partial x^1})\big|_{\mathbb{S}^3}+(x^4\frac{\partial}{\partial x^2}-x^2\frac{\partial}{\partial x^4})\big|_{\mathbb{S}^3},
\end{align}
where $\frac{\partial}{\partial x^j}$ here denote the Cartesian coordinate vector fields in $\mathbb{R}^4$. Since $\Sigma_T$ is diffeomorphic to $\mathbb{S}^3$, we pull back $Y_i$ and then Lie propagate them along $\partial_t$,
\begin{align}\label{[Yi,dt]}
[\partial_t,Y_i]=0.
\end{align}
Note that $Y_i$ are non-holonomic, since initially on $\Sigma_T$
\begin{align}\label{[Yi,Yj]}
[Y_i,Y_j]=2\epsilon_{ij}{}^lY_l.
\end{align}
Also, note that the relation \eqref{[Yi,Yj]} is propagated in evolution by virtue of \eqref{[Yi,dt]}. The benefit of working with the vector fields $Y_i$ is that they are non-degenerate. We will also use them as commutation vector fields to derive higher order energy estimates in Section \ref{sec:FSE}.

Moreover, we consider an orthonormal frame 
\begin{align}\label{ei}
e_I=e^i_IY_i,\qquad e_0=n^{-1}\partial_t,
\end{align}
which is Fermi propagated along $e_0$ according to 
\begin{align}\label{De0eI}
{\bf D}_{e_0}e_I=n^{-1}(e_In)e_0,
\end{align}
where ${\bf D}$ is the Levi-Civita connection of ${\bf g}$. Note that $e_I$ is $\Sigma_t$-tangent, provided it is initially the case. It also holds that
\begin{align}\label{De0e0}
{\bf D}_{e_0}e_0=n^{-1}(e_In)e_I.
\end{align}
The connection coefficients associated to $e_\mu$ are
\begin{align}\label{kIJgammaIJB}
k_{IJ}:=-{\bf g}({\bf D}_{e_I}e_0,e_J)=k_{IJ},\qquad\gamma_{IJB}={\bf g}({\bf D}_{e_I},e_J,e_B)=g(\nabla_{e_I}e_J,e_B)=-\gamma_{IBJ},
\end{align}
where $\nabla$ is the Levi-Civita connection of $g$. 

Lastly, the fluid vector field $u$ is projected onto $e_\mu$:
\begin{align}\label{u0uI}
u_0={\bf g}(u,e_0)<0,\qquad u_I={\bf g}(u,e_I),\qquad -u_0^2+u_Cu_C=-1.
\end{align}
Similarly, the renormalized fluid components equal
\begin{align}\label{v0vI}
v_0=\rho^{r_s}u_0,\qquad v_I=\rho^{r_s}u_I,\qquad -v_0^2+v_Cv_C=-\rho^{2r_s}.
\end{align}

\subsubsection{Main theorem}

We are now ready to state our overall stability theorem.
\begin{theorem}\label{thm:A}
Let $\widetilde{\bf g},\widetilde{\rho},\widetilde{u}$ be a homogeneous solution \eqref{metric.hom}-\eqref{rho.hom} with extreme tilt, for $c_s^2\in(\frac{1}{3},\frac{3}{7})$. Let $(\Sigma_T\cong\mathbb{S}^3,\mathring{g},\mathring{k},\mathring{\rho},\mathring{u})$ denote a perturbed initial data set around the latter homogeneous solution, satisfying the constraints \eqref{Hamconst.intro}-\eqref{momconst.intro}, and let $(\mathcal{M},{\bf g},\rho,u)$ be the corresponding maximal solution to \eqref{EE}, \eqref{divT.v}, \eqref{divT.rho}. Consider the geometric and fluid variables $k_{IJ},\gamma_{IJB},e_I^i,n,\rho^{2r_s},v_\mu$ introduced above and let $\widehat{k}_{IJ},\widehat{\gamma}_{IJB},\widehat{e}^i_I,\widehat{n},\widehat{\rho^{2r_s}},\widehat{v}_\mu$ denote the corresponding differences, after subtracting their background values, see \eqref{frame.hom}, \eqref{red.var.hom}, \eqref{diff.var}. Assume that their initial data, induced by $(\mathring{g},\mathring{k},\mathring{\rho},\mathring{u})$ and our choice of $e_I$ on $\Sigma_T$ (see Section \ref{subsec:red.ID}), are sufficiently small in $H^N(\Sigma_T)$, for some $N\ge7$ (see Section \ref{SS:SOBOLEVNORMS} for the precise definition of the norms): 
\begin{align}\label{thm.main.init}
\begin{split}
\mathring{\varepsilon}^2=&\,e^{2HT}\big\{\|\widehat{k}\|_{H^N(\Sigma_T)}^2
+\|\widehat{\gamma}\|_{H^N(\Sigma_T)}^2+\|\widehat{e}\|_{H^N(\Sigma_T)}^2+e^{HT}\|\widehat{n}\|_{H^N(\Sigma_T)}^2\big\}\\
&+e^{2HT}\|\widehat{v}\|_{H^N(\Sigma_T)}^2+e^{\frac{8r_s}{1-2r_s}HT}\|\widehat{\rho^{2r_s}}\|_{H^N(\Sigma_T)}^2.
\end{split}
\end{align}
Then, the perturbed solution is globally defined in the future of $\Sigma_T$, relative to the time function $t$ normalized by \eqref{parab.gauge}, and the following estimates hold true:
\begin{align}\label{thm.main.est}
\notag e^{2Ht}\big\{\|\widehat{k}\|_{H^N(\Sigma_t)}^2
+\|\widehat{\gamma}\|_{H^N(\Sigma_t)}^2+\|\widehat{e}\|_{H^N(\Sigma_t)}^2
+e^{Ht}\|\widehat{n}\|_{H^N(\Sigma_t)}^2\big\}\leq &\,C\mathring{\varepsilon}^2,\\
e^{2Ht}\|\widehat{v}\|_{H^{N-1}(\Sigma_t)}^2+e^{\frac{8r_s}{1-2r_s}Ht}\|\widehat{\rho^{2r_s}}\|^2_{H^{N-2}(\Sigma_t)}\leq &\,C\mathring{\varepsilon}^2,\\
\notag e^{-2A_sHt}\big\{e^{2Ht}\|\widehat{v}\|_{\dot{H}^N(\Sigma_t)}^2+e^{\frac{8r_s}{1-2r_s}Ht}\big(\|\widehat{\rho^{2r_s}}\|^2_{\dot{H}^{N-1}(\Sigma_t)}+
\|\widehat{\rho^{2r_s}}\|^2_{\dot{H}^{N}(\Sigma_t)}\big)\big\}\leq &\,C\mathring{\varepsilon}^2,
\end{align}
for all $t\in[T,+\infty)$, where the constant $C>0$ depends on $N$ and the background homogeneous solution, while 
\begin{align}\label{As}
r_s=\frac{c_s^2}{1+c_s^2},\qquad A_s=\frac{3c_s^2-1}{1-c_s^2}.
\end{align} 
Moreover, the following improved estimates are valid:
\begin{align}\label{thm.k.n.hat.imp.est}
\|\widehat{k}_{IJ}\|_{W^{N-4,\infty}(\Sigma_t)}+\|\widehat{n}\|_{W^{N-4,\infty}(\Sigma_t)}\leq C\mathring{\varepsilon} e^{-2Ht}
\end{align}
and there exist functions $(\widehat{e}^\infty)_I^i\in W^{N-4,\infty}(\mathbb{S}^3)$, $\widehat{\rho^{2r_s}}^\infty,\widehat{v}_0^\infty,\widehat{v}_I^\infty\in W^{N-5,\infty}(\mathbb{S}^3)$ such that 
\begin{align}
\notag\|\widehat{e}^i_I-(\widehat{e}^i_I)^\infty(\omega)e^{-Ht}\|_{W^{N-4,\infty}(\Sigma_t)}\leq &\, C\mathring{\varepsilon} e^{-3Ht},\\
\label{thm.e.v.rho.hat.infty}\|\widehat{v}_\mu-\widehat{v}_\mu^\infty(\omega)e^{-Ht}\|_{W^{N-5,\infty}(\Sigma_t)}\leq &\,C\mathring{\varepsilon} e^{-2Ht},\\
\notag\|\widehat{\rho^{2r_s}}-\widehat{\rho^{2r_s}}^\infty(\omega)e^{-\frac{4r_s}{1-2r_s}Ht}\|_{W^{N-5,\infty}(\Sigma_t)}\leq &\,C\mathring{\varepsilon}(e^{-Ht}+e^{2Ht-\frac{4r_s}{1-2r_s}Ht})e^{-\frac{4r_s}{1-2r_s}Ht},
\end{align}
for all $t\in[T,+\infty)$, where $\frac{4r_s}{1-2r_s}>2$. The leading order coefficients satisfy as well: 
\begin{align}\label{e.rho.v.infty.est}
\|(\widehat{e}^i_I)^\infty(\omega)\|_{W^{N-4,\infty}(\mathbb{S}^3)},\|\widehat{v}_\mu^\infty(\omega)\|_{W^{N-5,\infty}(\mathbb{S}^3)},\|\widehat{\rho^{2r_s}}^\infty(\omega)\|_{W^{N-5,\infty}(\mathbb{S}^3)}\leq C\mathring{\varepsilon}.
\end{align}
\end{theorem}
\begin{proof}
The proof of the main estimates \eqref{thm.main.est} consists of a bootstrap argument \eqref{Boots} that is split into Propositions \ref{prop:geom.est}, \ref{prop:fluid.est}, \ref{prop:cont.arg}. The improved asymptotic behaviors \eqref{thm.k.n.hat.imp.est}-\eqref{thm.e.v.rho.hat.infty} are a consequence of \eqref{thm.main.est} and are derived in Proposition \ref{prop:asym.beh}.
\end{proof}
\begin{remark}
For $c_s^2\in(\frac{1}{3},\frac{3}{7})$, the other parameters range in $r_s\in(\frac{1}{4},\frac{3}{10})$, $A_s\in(0,\frac{1}{2})$. If we were to treat the whole beyond radiation range of sound speeds $c_s^2\in(\frac{1}{3},1)$, then we would have $r_s\in(\frac{1}{4},\frac{1}{2})$, $A_s\in(0,+\infty)$. The reason for the restrictive assumption $c_s^2\in(\frac{1}{3},\frac{3}{7})$ is due to a degenerate top order estimate for the fluid variables, see \eqref{fluid.top.est} in Proposition \ref{prop:fluid.est}, which is only useful when $A_s<1/2$, ie. $c_s<\sqrt{3/7}$, see also Remark \ref{rem:top.fluid.est}. We elaborate more on this issue in Section \ref{subsec:method}.
\end{remark}
\begin{remark}\label{rem:asym.beh}
The asymptotic behavior of the original variables ${\bf g},\rho,u$ is easily deduced from Theorem \ref{thm:A}: 
\begin{align}\label{g.asym}
\|n-1\|_{W^{N-4,\infty}(\Sigma_t)}\leq C e^{-2Ht},\qquad \|g_{ij}-g_{ij}^\infty(\omega)e^{2Ht}\|_{W^{N-4,\infty}(\Sigma_t)}\leq C,
\end{align}
\begin{align}\label{rho.u.asym}
\begin{split}
\|e^{-A_sHt}u_\mu -u_\mu^\infty(\omega)\|_{W^{N-5,\infty}(\Sigma_t)}\leq&\, C(e^{-Ht}+e^{2Ht-\frac{4r_s}{1-2r_s}Ht}),\\
\|e^{\frac{2r_s}{1-2r_s}Ht}\rho-\rho^\infty(\omega)\|_{W^{N-5,\infty}(\Sigma_t)}
\leq&\, C(e^{-Ht}+e^{2Ht-\frac{4r_s}{1-2r_s}Ht}),
\end{split}
\end{align}
for all $[T,+\infty)$, where 
\begin{align}\label{g.rho.u.infty}
g_{ij}^\infty=(\Omega^C_i)^\infty(\Omega^C_j)^\infty,\qquad (\Omega^C_i)^\infty(e_C^b)^\infty=\delta_i^b,\qquad (e^b_C)^\infty=(\widehat{e}^b_C)^\infty+(\widetilde{e}^b_C)^\infty,\\
\rho^\infty=\big[\widehat{\rho^{2r_s}}^\infty+(\widetilde{\rho}^\infty)^{2r_s}\big]^\frac{1}{2r_s},\qquad u_\mu^\infty=(\rho^\infty)^{r_s}v_\mu^\infty,\qquad v_\mu^\infty=\widehat{v}_\mu^\infty+\widetilde{v}^\infty_\mu.
\end{align}
Here, $(\widetilde{e}_C^b)^\infty,\widetilde{\rho}^\infty,\widetilde{v}^\infty_\mu$ are the leading order coefficients of the corresponding background variables, see \eqref{hom.exp}.
\end{remark}
\begin{remark}
The asymptotic behavior of the metric \eqref{g.asym} is de Sitter-like, as anticipated. On the other hand, the asymptotic behavior of the fluid variables \eqref{rho.u.asym}-\eqref{g.rho.u.infty} confirms that extreme tilt behavior persists to general perturbations of the homogeneous backgrounds. Indeed, from \eqref{thm.e.v.rho.hat.infty} and \eqref{hom.exp} we have 
\begin{align}\label{extreme.tilt}
-(v_0^\infty)^2+v_I^\infty v_I^\infty=\lim_{t\rightarrow+\infty}e^{2Ht}(-v_0^2+v_Iv_I)
=-\lim_{t\rightarrow+\infty}e^{2Ht}\rho^{2r_s}=0.
\end{align}
Hence, 
\begin{align}\label{extreme.tilt.2}
(u_0^\infty)^2=(\rho^\infty)^{2r_s}(v_0^\infty)^2=[(\rho^\infty)^{r_s}v_I^\infty] (\rho^\infty)^{r_s}v_I^\infty=u_I^\infty u_I^\infty,
\end{align}
that is, the leading order behavior of $u_\mu$ is null. 
\end{remark}
%


%

%
\subsection{Previous work on the relativistic Euler equations: The entire beyond radiation interval $1/\sqrt{3}<c_s<1$}

Our result is preceded by works on the relativistic Euler equations on a fixed (flat) FLRW background, proving the future stability of homogeneous perfect fluids with extreme tilt, first for the restricted range $c_s^2\in(\frac{1}{3},\frac{1}{2})$ \cite{Ol2} and then for all sound speeds beyond radiation $c_s^2\in(\frac{1}{3},1)$ \cite{MO}. This raises the question as to whether our method can be improved to treat all $1/\sqrt{3}<c_s<1$. In Appendix \ref{app:Euler}, we sketch a proof of the analogue of the future stability result in \cite{MO} for the Euler equations on fixed $\mathbb{S}^3$-FLRW metrics, incorporating the entire interval $c_s^2\in(\frac{1}{3},1)$. The argument works independently of the value of $c_s$, thanks to the absence of certain error terms in the case of the simplified isotropic FLRW metric. Hence, we pinpoint in a sense where the problem lies when studying the Euler equations on general inhomogeneous metrics, as in the coupled to Einstein case. At the moment, we cannot see an improvement in our estimates that would lead to an extension of Theorem \ref{thm:A}, for $c_s^2\ge3/7$. We comment more on the degenerate top order estimates that we are able to obtain in the next subsection. We should also emphasize that this seems to be a problem related to energy estimates in finite regularity Sobolev spaces, which is not encountered in the study of formal series expansions, making it even more peculiar should a real instability be present for $c_s^2\ge3/7$.
\subsection{Method of proof: Degenerate top order estimates}\label{subsec:method}
Our proof is based on bootstrapping the smallness of a weighted $H^N$-type of energy for the variables $\widehat{k},\widehat{\gamma},\widehat{e},\widehat{n},\widehat{v},\widehat{\rho^{2r_s}}$, precisely, \eqref{geom.en}-\eqref{Boots}. Improving the bootstrap assumptions yields a global solution satisfying the main stability estimate \eqref{thm.main.est} by a standard continuation argument, see Proposition \ref{prop:cont.arg}. 

Assuming that the fluid variables are `well-behaved', the stability analysis for the geometric variables is similar to the stability of de Sitter, which by now has been well-understood in various different settings \cite{And,Fr,FK,F,MV,Rin}. Hence, the crux of the matter lies in controlling the part of the energy corresponding to the fluid variables $\widehat{v},\widehat{\rho^{2r_s}}$. As it turns out, the lower than top order energy in \eqref{fluid.en} can be controlled by treating the resulting Euler equations \eqref{vI.hat.eq}, \eqref{rho.hat.eq} as transport equations, using the bootstrap assumptions to view all other terms in the equations as error (here written at zeroth order):
\begin{align}\label{ODE.error.intro}
\partial_t\widehat{v}_I+H\widehat{v}_I=\text{error},\qquad\qquad
\partial_t\widehat{\rho^{2r_s}}+\frac{4r_s}{1-2r_s}\widehat{\rho^{2r_s}}=\text{error}.
\end{align}
The argument actually resembles the analysis of formal series expansions \cite{LEUW,Ren} and it works for all sound speeds beyond radiation $c_s^2\in(\frac{1}{3},1)$. 

The stability analysis becomes more intricate when deriving top order energy estimates, due to the extreme tilt behavior that we wish to bootstrap in the first place. Firstly, it requires a formulation of the Euler equations which is useful for reading off the asymptotic behavior of the renormalized fluid variables, like \eqref{ODE.error.intro}, and in addition, one where its symmetric hyperbolic structure can be exploited. At top order, this is achieved by complementing the top order analogue of \eqref{ODE.error.intro} with an equation for $v_IY^\iota\widehat{v}_I$, see \eqref{vI.hat.eq.diff}, \eqref{rho.hat.eq.diff}, \eqref{vI.hat.top.eq}. Regardless of the weights one chooses in the definition of the total energy $\mathcal{E}_{tot}(t)$ of the system, a top order energy estimate for the Euler equations as we formulated them will have the form 
\begin{align}\label{top.en.est.intro}
\mathcal{E}_{tot}(t)\leq \text{initial data}-\int^t_T C_{error}(\tau)\cdot\frac{v_0^2}{{\bf g}(v,v)}\mathcal{E}_{tot}(\tau)d\tau,
\end{align}
where the coefficient of the energy in the last integral contains the ${\bf g}$-norm of the fluid vector field. On the other hand, extreme tilt behavior implies that $|{\bf g}(v,v)|$ decays faster than $v_0^2$. Indeed, based on the behavior of the fluid that we are trying to derive, see \eqref{thm.e.v.rho.hat.infty}, we have
\begin{align}\label{top.en.coeff.intro}
-\frac{v_0^2}{{\bf g}(v,v)}=\frac{v_0^2}{\rho^{2r_s}}\sim e^{-2Ht+\frac{4r_s}{1-2r_s}Ht}=e^{2A_sHt}.
\end{align}
Since $A_s\rightarrow +\infty$ as $c_s\rightarrow1$, we observe that the energy estimate \eqref{top.en.est.intro} becomes more and more degenerate the larger the sound speed is. 

It is clear that in order to obtain a useful estimate for $\mathcal{E}_{tot}(t)$ we need 
\begin{align}\label{Cerror.est}
C_{error}(t)e^{2A_sHt}\lesssim e^{-\delta Ht}, \qquad\text{for some $\delta>0$.}
\end{align}
However, for general inhomogeneous metrics, $C_{error}(t)\lesssim e^{-Ht}$ is the best bound we have, see Remarks \ref{rem:top.fluid.error}. The latter allows us to complete our bootstrap argument only when
\begin{align}\label{range.intro}
2A_s<1\qquad\Leftrightarrow\qquad c_s^2<\frac{3}{7},
\end{align} 
which leads to the restricted range of sound speeds beyond radiation, $c_s^2\in(\frac{1}{3},\frac{3}{7})$, see also Remark \ref{rem:top.fluid.est}.

At the moment, we have no improved estimate for $C_{error}(t)$, apart from when the metric is exactly FLRW, see Appendix \ref{app:Euler}. In that case, we can show that actually $C_{error}(t)\lesssim e^{-\delta Ht}e^{-2A_sHt}$, which in turn yields a uniform energy estimate in $c_s^2\in(\frac{1}{3},1)$. This is consistent with the results in \cite{MO}. We do not know whether for general metrics the above degeneracy is actually present or merely a defect of our energy argument. We should emphasize that it is only encountered at the level of energy estimates and not at the level of a formal series expansion argument using \eqref{ODE.error.intro}, which indicates that if it is a real phenomenon, it cannot be seen by studying analytic solutions.

\subsection{Outline of the paper}

In Section \ref{sec:red.eq}, we express the Einstein-Euler system in the framework introduced in Section \ref{sss:framework}. We set up the equations for the perturbed variables minus their homogeneous counterparts, whose precise asymptotic behavior is derived in Appendix \ref{app:ODE}. In Section \ref{sec:Boots}, we introduce the weighted norms and bootstrap assumptions that we will use to prove the future stability of the homogeneous background. In Section \ref{sec:FSE}, we derive the main energy estimates that complete the bootstrap argument, for the restricted range of sound speeds $c_s^2\in(\frac{1}{3},\frac{3}{7})$. The bootstrap estimates are reiterated in Section \ref{sec:refined.est} to obtain the precise asymptotic behavior of all variables at infinity. Finally, in Appendix \ref{app:Euler}, we consider the Euler equations on a fixed $\mathbb{S}^3$-FLRW background and sketch how our main energy estimates can improved to encompass the entire beyond radiation interval $c_s^2\in(\frac{1}{3},\frac{3}{7})$.

\subsection{Notation}

\begin{itemize}
\item We will use the symbol $\mathcal{O}(e^{Bt})$, $B\in\mathbb{R}$, to denote smooth homogeneous functions $\mathbb{R}\to\mathbb{R}$ that satisfy $|\partial_t^N\mathcal{O}(e^{Bt})|\leq C_{N,B}e^{Bt}$, for any $N\in\mathbb{N}$, where the constant $C_{N,B}$ depends on $N,B,\Lambda,c_s$, as well as the background tilted solutions. 

\item $H=\sqrt{\frac{\Lambda}{3}}$.

\item $C>0$ denotes a generic constant that depends on the number $N$ of derivatives in our norms and the homogeneous background we are perturbing about. Also, it will be allowed to change from one line to the next. 

\item Certain key parameters related to the speed of sound appear in the definition of our norms and renormalized quantities:  
\begin{align}\label{key.param}
\notag
\text{(unrestricted range)}&\qquad c_s^2\in(\frac{1}{3},1),\qquad r_s=\frac{c_s^2}{1+c_s^2}\in(\frac{1}{4},\frac{1}{2}),\qquad A_s=\frac{3c_s^2-1}{1-c_s^2}\in(0,+\infty)\\
\text{(restricted range)}&\qquad c_s^2\in(\frac{1}{3},\frac{3}{7}),\qquad r_s=\frac{c_s^2}{1+c_s^2}\in(\frac{1}{4},\frac{3}{10}),\qquad A_s=\frac{3c_s^2-1}{1-c_s^2}\in(0,\frac{1}{2})\\
\text{(algebraic relations)}&\qquad\frac{4r_s}{1-2r_s}>2,\qquad A_s=-1+\frac{2r_s}{1-2r_s}
\notag
\end{align}

\item We use Einstein summation for repeated indices. Whenever a sum is computed relative to the spatial orthonormal frame, we do not raise indices, e.g. $(e_In)e_I=\sum_{I=1}^3(e_In)e_I$.

\item Latin indices $a,b,i,j,A,B,I,J$  
range over $\{1,2,3\}$. Small letters correspond to the vector fields $Y_1,Y_2,Y_3$, while capital letters are reserved for the spatial orthonormal frame $e_I$. In few instances we use Greek letters $\alpha,\beta,\mu,\nu$, ranging over $\{0,1,2,3\}$. For example, $v^\alpha e_\alpha=-v_0e_0+v_Ie_I$.

\item The Riemann curvature ${\bf Riem}$, Ricci curvature ${\bf Ric}$, and scalar curvature ${\bf R}$ of ${\bf g}$ are defined as follows:
\begin{align} \label{curv.bf.g}
\notag	{\bf Riem}(e_{\alpha},e_{\beta},e_{\mu},e_{\nu}):=&\,{\bf g}
	({\bf D}^2_{e_{\alpha} e_{\beta}} e_\nu 
		- 
		{\bf D}^2_{e_{\beta} e_{\alpha}} e_\nu,
		e_\mu ), \\
	{\bf Ric}(e_{\alpha},e_{\beta})
	 :=&-{\bf Riem}(e_{\alpha},e_0,e_{\beta},e_0)+{\bf Riem}(e_{\alpha},e_I,e_{\beta},e_I),\\
\notag	{\bf R }
	 :=&-{\bf Ric}(e_0,e_0)+{\bf Ric}(e_I,e_I),
\end{align}
where ${\bf D}^2_{e_{\alpha} e_{\beta}}e_\nu:={\bf D}_{e_\alpha}({\bf D}_{e_\beta}e_\nu)-{\bf D}_{{\bf D}_{e_\alpha}e_{\beta}}e_\nu$.
The corresponding curvature tensors of $g$, denoted by $Riem$, $Ric$, $R$, are defined analogously.

\end{itemize}

\subsection{Acknowledgments}

G.F. would like to thank Hans Ringstr\"om, Jared Speck, and Claes Uggla for useful discussions. 
Also, G.F. would like to thank everyone at MFO for their hospitality during the workshop ``Mathematical Aspects of General Relativity'' 29 August - 4 September, 2021, which was beneficial for the current project.
G.F. gratefully acknowledges the support of the ERC starting grant 101078061 SINGinGR, under the European Union's Horizon Europe program for research and innovation.

\section{Setting up the stability problem}\label{sec:red.eq}

In this section, we express the Einstein-Euler system \eqref{EE}, \eqref{divT.v}, \eqref{divT.rho} and the background homogeneous variables in the gauge \eqref{metric}-\eqref{v.rs}. Moreover, we derive the resulting equations for the variables that measure the closeness of the perturbed solution to the homogeneous background.

\subsection{The reduced Einstein-Euler equations}
\begin{lemma}\label{lem:red.eq}
The geometric variables $k_{IJ},\gamma_{IJB},e^i_I,n$ satisfy the evolution equations: 
\begin{align}
\label{k.eq}e_0k_{IJ}+(n-1-\mathrm{tr}\widetilde{k})k_{IJ}=&-n^{-1} e_I e_J n+e_C \gamma_{IJC}-e_I\gamma_{CJC}\\
\notag&+n^{-1}\gamma_{IJC} e_C n-\gamma_{CID}\gamma_{DJC}-\gamma_{IJD}\gamma_{CCD}\\
\notag&-\Lambda\delta_{IJ}-(1+c_s^2)\rho^{1-2r_s} v_I v_J-\frac{1}{2}\delta_{IJ}(1-c_s^2)\rho,\\
\label{gamma.eq}e_0\gamma_{IJB}-k_{IC} \gamma_{CJB}=&\,e_B k_{IJ}-e_J k_{BI}-k_{IC}\gamma_{BJC}-k_{CJ} \gamma_{BIC}+k_{IC} \gamma_{JBC}
+k_{BC} \gamma_{JIC}\\
&+n^{-1}(e_B n)k_{JI}-n^{-1}(e_J n)k_{BI},\notag\\
\label{eIi.eq}e_0e_I^i=&\,k_{IC} e_C^i,\\
\label{n.eq}\partial_tn-e_C e_C n=&-\gamma_{CCD}e_Dn-nk_{CD}k_{CD}+n\Lambda+\partial_t\mathrm{tr}\widetilde{k}\\
\notag&+(1+c_s^2)n\rho^{1-2r_s} v_Cv_C+\frac{3}{2}(1-c_s^2)n\rho.
\end{align}
The renormalized fluid variables $\rho,v_0,v_I$ satisfy the evolution equations:
\begin{align}
\label{v0.eq}v^\alpha e_\alpha v_0+k_{CD}v_Cv_D=&-\frac{1}{2}e_0(\rho^{2r_s})-n^{-1}(e_Cn)v_0v_C,\\
\label{vI.eq}
v^\alpha e_\alpha v_I+k_{CI}v_0v_C=&-\frac{1}{2}e_I(\rho^{2r_s})-n^{-1}(e_In)v_0^2-\gamma_{CDI}v_Cv_D,\\
\label{rho.eq}(1-2r_s)v^\alpha e_\alpha\log\rho+v_0\mathrm{tr}k=&\,e_0v_0-e_Cv_C-n^{-1}(e_Cn)v_C-\gamma_{CDC}v_D,
\end{align}
where $v^\alpha e_\alpha=-v_0e_0+v_Ce_C$.
Also, the following constraint equations hold:
\begin{align}
\label{Hamconst}2e_C \gamma_{DDC}
-
\gamma_{CDE}\gamma_{EDC}
-
\gamma_{CCD}\gamma_{EED}
=&\,k_{CD}k_{CD}-(n-1-\mathrm{tr}\widetilde{k})^2+2\Lambda+2(1+c_s^2)\rho^{1-2r_s}v_0^2-2c_s^2\rho,\\
e_C k_{CI}+e_In+ 
k_{ID} \gamma_{CDC}
-
k_{CD} \gamma_{CID}=&-(1+c_s^2)\rho^{1-2r_s} v_0v_I.
\label{momconst}
 \end{align}
\end{lemma}
\begin{proof}
First, note that the Einstein equations \eqref{EE}, the formula \eqref{Tmunu} of $T_{\mu\nu}$, for $p=c_s^2\rho$, and \eqref{v.rs} imply the formula:
\begin{align}\label{EE2}
{\bf Ric}_{\mu\nu}=&\,\Lambda{\bf g}_{\mu\nu}+(1+c_s^2)\rho^{1-2r_s} v_\mu v_\nu+\frac{1}{2}(1-c_s^2)\rho{\bf g}_{\mu\nu}.
\end{align}
To prove \eqref{k.eq}, we use the formula \eqref{kIJgammaIJB} for $k_{IJ}$, the propagation relations \eqref{De0eI}, \eqref{De0e0}, and \eqref{EE2} to compute its $e_0$ derivative:
\begin{align}\label{R0I0J}
\notag e_0k_{IJ}=&-{\bf g}({\bf D}_{e_0}({\bf D}_{e_I}e_0),e_J)-n^{-1}(e_Jn){\bf g}({\bf D}_{e_I}e_0,e_0)\\
\notag=&-{\bf g}({\bf D}_{e_0e_I}^2e_0,e_J)-{\bf g}({\bf D}_{{\bf D}_{e_0}e_I}e_0,e_J)\\
\notag=&\,{\bf Riem}(e_I,e_0,e_J,e_0)-{\bf g}({\bf D}_{e_Ie_0}^2e_0,e_J)-n^{-2}(e_In)(e_Jn)\\
=&-{\bf Ric}(e_I,e_J)+{\bf Riem}(e_I,e_C,e_J,e_C)-{\bf g}({\bf D}_{e_I}({\bf D}_{e_0}e_0),e_J)+{\bf g}({\bf D}_{{\bf D}_{e_I}e_0}e_0,e_J)\\
\notag&-n^{-2}(e_In)(e_Jn)\\
\notag=&-\Lambda\delta_{IJ}-(1+c_s^2)\rho^{1-2r_s} v_Iv_J-\frac{1}{2}\delta_{IJ}(1-c_s^2)\rho+{\bf Riem}(e_I,e_C,e_J,e_C)\\
\notag&-n^{-1}e_Ie_Jn+n^{-1}\gamma_{IJC}e_Cn+k_{IC}k_{CJ}.
\end{align}
Next, we use the Gauss equations
\begin{align}\label{Gauss}
{\bf Riem}(e_I,e_C,e_J,e_C)=&\,Ric(e_I,e_J)-k_{IC}k_{CJ}+\text{tr}kk_{IJ}
\end{align}
and expand the Ricci term in the last RHS:
\begin{align}\label{RIJ}
\notag Ric(e_I,e_J)=&\, Riem(e_I,e_C,e_J,e_C)=g(\nabla_{e_Ie_C}^2e_C-\nabla_{e_Ce_I}^2e_C,e_J)\\
\notag=&\,g(\nabla_{e_I}(\nabla_{e_C}e_C),e_J)-g(\nabla_{e_C}(\nabla_{e_I}e_C),e_J)-g(\nabla_{\nabla_{e_I}e_C}e_C,e_J)+g(\nabla_{\nabla_{e_C}e_I}e_C,e_J)\\
=&\,g(\nabla_{e_I}(\gamma_{CCD}e_D),e_J)-g(\nabla_{e_C}(\gamma_{ICD}e_D),e_J)-\gamma_{ICD}\gamma_{DCJ}+\gamma_{CID}\gamma_{DCJ}\\
\notag=&\,e_I\gamma_{CCJ}+\gamma_{CCD}\gamma_{IDJ}-e_C\gamma_{ICJ}-\gamma_{ICD}\gamma_{CDJ}-\gamma_{ICD}\gamma_{DCJ}+\gamma_{CID}\gamma_{DCJ}\\
\notag=&-e_I\gamma_{CJC}+e_C\gamma_{IJC}-\gamma_{CID}\gamma_{DJC}-\gamma_{CCD}\gamma_{IJD},
\end{align}
where in the last equality we used the anti-symmetry $\gamma_{IJB}=-\gamma_{IBJ}$.
Plugging \eqref{Gauss}-\eqref{RIJ} into \eqref{R0I0J} and using the gauge condition \eqref{parab.gauge} yields \eqref{k.eq}.

Also, contracting \eqref{Gauss}, \eqref{RIJ} gives
\begin{align}\label{Gauss2}
{\bf R}+2{\bf Ric}(e_0,e_0)=&\,2e_C\gamma_{DDC}-\gamma_{CED}\gamma_{DEC}-\gamma_{CCD}\gamma_{EED}-k_{IC}k_{CJ}+(\text{tr}k)^2,
\end{align}
which after using the Einstein equations \eqref{EE}, the gauge condition \eqref{parab.gauge}, and re-arranging the terms, results in the Hamiltonian constraint equation \eqref{Hamconst}. The momentum constraint equation is a rewriting of \eqref{momconst.intro} (its version along $\Sigma_t$ instead of the initial hypersurface) using \eqref{parab.gauge} and expanding the divergence of $k$ relative to the frame $e_I$.

For the lapse equation \eqref{n.eq}, we take the trace of \eqref{k.eq}, use the condition \eqref{parab.gauge}, multiply both sides with $n$, and plug in \eqref{Hamconst}. 

The evolution equation \eqref{eIi.eq} for the frame coefficients is a direct consequence of the identity \eqref{ei} and the propagation condition \eqref{De0eI}:
\begin{align*}
[e_0,e_I]=&\,{\bf D}_{e_0}e_I-{\bf D}_{e_I}e_0=n^{-1}(e_In)e_0+k_{IC}e_C=n^{-1}(e_In)e_0+k_{IC}e_C^i\partial_i,\\
[e_0,e_I]=&\,[n^{-1}\partial_t,e_I]=n^{-1}[\partial_t,e_I]+n^{-1}(e_In)e_0=n^{-1}(\partial_t e^i_I)\partial_i+n^{-1}(e_In)e_0.
\end{align*}

For \eqref{gamma.eq}, we use the formula \eqref{kIJgammaIJB} for $\gamma_{IJB}$, the propagation condition \eqref{De0eI}, and the Codazzi equations to compute its $e_0$ derivative:
\begin{align*}
e_0\gamma_{IJB}=&\,{\bf g}({\bf D}_{e_0}({\bf D}_{e_I}e_J),e_B)+{\bf g}({\bf D}_{e_I}e_J,{\bf D}_{e_0}e_B)\\
=&\,{\bf g}({\bf D}^2_{e_0e_I}e_J,e_B)+{\bf g}({\bf D}_{{\bf D}_{e_0}e_I}e_J,e_B)+n^{-1}(e_Bn)k_{IJ}\\
=&\,{\bf Riem}(e_0,e_I,e_B,e_J)+{\bf g}({\bf D}^2_{e_Ie_0}e_J,e_B)+n^{-2}(e_In)(e_Jn){\bf g}(e_0,e_B)+n^{-1}(e_Bn)k_{IJ}\\
=&\,{\bf Riem}(e_B,e_J,e_0,e_I)+{\bf g}({\bf D}_{e_I}({\bf D}_{e_0}e_J),e_B)-{\bf g}({\bf D}_{{\bf D}_{e_I}e_0}e_J,e_B)+n^{-1}(e_Bn)k_{IJ}\\
=&\,\nabla_{e_B}k_{JI}-\nabla_{e_J}k_{BI}+{\bf g}({\bf D}_{e_I}[n^{-1}(e_Jn)e_0],e_B)+k_{IC}\gamma_{CJB}+n^{-1}(e_Bn)k_{IJ}\\
=&\,\nabla_{e_B}k_{JI}-\nabla_{e_J}k_{BI}-n^{-1}(e_Jn)k_{BI}+k_{IC}\gamma_{CJB}+n^{-1}(e_Bn)k_{IJ}
\end{align*}
We arrive at the desired equation after expanding the covariant derivative terms in the last RHS.

To derive the fluid equations \eqref{v0.eq}-\eqref{rho.eq}, we expand the covariant operations:
\begin{align}
\label{divv}
{\bf D}^\mu v_\mu=&\,{\bf m}^{\mu\nu}{\bf g}({\bf D}_{e_\mu}(v^\alpha e_\alpha),e_\nu)=-e_0v_0+e_Cv_C+v_0\text{tr}k+v_D\gamma_{CDC}+n^{-1}(e_Cn)v_C,\\
\label{Dvv0}({\bf D}_vv)_0=&\,{\bf g}({\bf D}_{v^\mu e_\mu}(v^\nu e_\nu),e_0)
=-v_0 (e_0 v_0)+v_C(e_Cv_0)+v_C v_Dk_{CD}+n^{-1}(e_Cn)v_0v_C,\\
\label{DvvI}({\bf D}_vv)_I=&\,{\bf g}({\bf D}_{v^\mu e_\mu}(v^\nu e_\nu),e_I)=-v_0(e_0 v_I)+v_C(e_Cv_I)+n^{-1}(e_In)v_0^2
+v_C v_0 k_{CI}+v_C v_D \gamma_{CD I},
\end{align}
and plug them in \eqref{divT.v}-\eqref{divT.rho}.
\end{proof}
Although the system \eqref{v0.eq}-\eqref{rho.eq} is symmetric hyperbolic in $v_0,v_I,\rho$, we find it more suitable in the derivations of the main estimates for the latter variables to eliminate $v_0$ in favor of $v_I,\rho$ using the identity $-v_0^2+v_Iv_I=-\rho^{2r_s}$. 
\begin{lemma}\label{lem:red.eq2}
The renormalized density $\rho^{2r_s}$ satisfies the equation:
\begin{align}\label{rho.eq2}
\notag&\big(\frac{1-2r_s}{2r_s}+\frac{1}{2}\frac{\rho^{2r_s}}{v_0^2}\big)e_0( \rho^{2r_s})
+\big(\frac{v_Cv_I}{v_0^2}k_{CI}-\mathrm{tr}k\big)\rho^{2r_s}\\
=&-\frac{\rho^{2r_s}}{v_0^2}\frac{v_Iv_D}{v_0}e_Dv_I
+\frac{\rho^{2r_s}}{v_0}e_Cv_C+\big(\frac{1-2r_s}{2r_s}-\frac{1}{2}\frac{\rho^{2r_s}}{v_0^2}\big)\frac{v_C}{v_0}e_C(\rho^{2r_s})\\
\notag&-\frac{\rho^{2r_s}}{v_0^2}\frac{v_I}{v_0}
\big\{n^{-1}(e_In)v_0^2+\gamma_{CDI}v_Cv_D\big\}
+\frac{\rho^{2r_s}}{v_0}\big\{n^{-1}(e_Cn)v_C+\gamma_{CDC}v_D\big\}.
\end{align}
\end{lemma}
\begin{proof}
Rewrite the equation \eqref{rho.eq} as
\begin{align*}
\frac{1-2r_s}{2r_s}(-v_0 e_0+v_Ce_C)\rho^{2r_s}+v_0\mathrm{tr}k\rho^{2r_s}=&\,\rho^{2r_s}e_0v_0-\rho^{2r_s}e_Cv_C-\rho^{2r_s}\big\{n^{-1}(e_Cn)v_C+\gamma_{CDC}v_D\big\}
\end{align*}
and replace $e_0v_0$ by 
\begin{align}\label{e0v0}
e_0v_0=\frac{v_I}{v_0}e_0v_I+\frac{1}{2}e_0(\rho^{2r_s}).
\end{align}
Using the evolution equation \eqref{vI.eq} to replace $e_0v_I$ with
\begin{align}\label{e0vI}
e_0 v_I=\frac{1}{v_0}\big\{k_{CI}v_0v_C+\frac{1}{2}e_I(\rho^{2r_s})+n^{-1}(e_In)v_0^2+\gamma_{CDI}v_Cv_D\big\}
\end{align}
and rearranging the terms in the resulting equation gives \eqref{rho.eq2}.
\end{proof}

\subsection{The background tilted homogeneous solutions and the associated reduced variables}

In this subsection, we provide the precise asymptotic behavior and estimates satisfied by the homogeneous solutions.
\begin{lemma}\label{lem:hom.est}
There exist homogeneous solutions to the Einstein-Euler equations, of the form \eqref{metric.hom}-\eqref{fluid.hom}, such that the functions $G_i(t),G(t),\theta(t),\widetilde{\rho}(t):[T,+\infty)\to\mathbb{R}$ satisfy the estimates:
\begin{align}\label{hom.est}
\begin{split}
|\partial_t^NG_i(t)-G_i^\infty H^Ne^{Ht}|\leq C_N e^{-Ht},
\qquad|\partial_t^NG(t)|\leq C_N e^{-\frac{1}{2}Ht},\\
\bigg|\partial_t^N\bigg[\theta(t)-\frac{3c_s^2-1}{1-c_s^2}Ht+\theta_0\bigg]\bigg|\leq C_Ne^{-2Ht},\qquad
\bigg|\partial_t^N\bigg[\widetilde{\rho}(t)-\widetilde{\rho}^\infty e^{-2\frac{1+c_s^2}{1-c_s^2}Ht}\bigg]\bigg|\leq C_Ne^{-4\frac{1+c_s^2}{1-c_s^2}Ht},
\end{split}
\end{align}
for every $N\in\mathbb{N}$, $\frac{1}{3}<c_s^2<1$, where $G_i^\infty,G^\infty,\widetilde{\rho}^\infty>0$ are positive constants and $C_N$ also depends on $c_s$. 
\end{lemma}
\begin{proof}
It is contained in Lemma \ref{lem:app.hom}.
\end{proof}
Given a homogeneous solution as above, we make the following choice of background orthonormal frame: 
\begin{align}\label{frame.hom}
\widetilde{e}_1=G^{-1}_1(t)Y_1,\qquad \widetilde{e}_2=G_2^{-1}(t)Y_2,\qquad\widetilde{e}_3=\frac{Y_3-G^2(t)G^{-2}_2(t)Y_2}{\sqrt{G_3^2(t)-G^4(t)G_2^{-2}(t)}},\qquad\widetilde{e}_0=\partial_t,
\end{align}
and compute its associated reduced variables.
\begin{lemma}\label{lem:red.var.hom}
The reduced variables of the fixed background solution behave like
\begin{align}\label{red.var.hom}
\notag
\widetilde{k}_{IJ}=-\delta_{IJ}H+\mathcal{O}(e^{-2Ht}),\qquad \widetilde{\gamma}_{IJB}=\mathcal{O}(e^{-Ht}),\qquad\widetilde{e}^i_I=\mathcal{O}(e^{-Ht}),\\
 \widetilde{n}=1,\qquad\widetilde{\rho}^{1-2r_s}=(\widetilde{\rho}^{1-2r_s})^{\infty,2}e^{-2Ht}+\mathcal{O}(e^{-4Ht}),\qquad 1-2r_s=\frac{1-c_s^2}{1+c_s^2},\\
\notag\widetilde{v}_0=\widetilde{\rho}^{r_s}\widetilde{u}_0=\widetilde{v}^{\infty,1}_0e^{-Ht}+\mathcal{O}(e^{-3Ht}),\qquad
\widetilde{v}_I=\widetilde{\rho}^{r_s}\widetilde{u}_I=\delta_{1I}[\widetilde{v}^{\infty,1}_1e^{-Ht}+\mathcal{O}(e^{-3Ht})],
\end{align}
where $(\widetilde{\rho}^{1-2r_s})^{\infty,2}>0$, $|\widetilde{v}_0^{\infty,1}|=|\widetilde{v}_0^{\infty,1}|>0$ are constant coefficients of the above leading order terms.
\end{lemma}
\begin{proof}
See \eqref{hom.exp} in Appendix \ref{app:ODE} and Lemma \ref{lem:app.hom}.
\end{proof}
\subsection{Initial data for the perturbed variables on $\Sigma_T$}\label{subsec:red.ID}
The initial data set ($\mathring{g},\mathring{k},\mathring{\rho},\mathring{\underline{v}},\mathring{v}_0$) on $\Sigma$ induces initial data for the reduced variables $k_{IJ},\gamma_{IJB},e_I^i,n,v_0,v_I,\rho^{2r_s}$ on $\Sigma_T$, after choosing an initial orthonormal frame $e_I$. Indeed, we have the freedom to choose $t=T$, ie. $\Sigma_T$, to be the image of the embedding of $\Sigma$ in the development $(\mathcal{M},{\bf g},\rho,v)$. Identifying the underlying manifolds of the homogeneous and perturbed solution, according to the 1+3 splitting \eqref{metric}, we can then compare the two mean curvatures $\text{tr}\widetilde{k},\text{tr}k$ and obtain the initial datum of $n$ via \eqref{parab.gauge}. The initial data for $e_I^i$ can be any smooth perturbation of the homogeneous frame components $\widetilde{e}_I^i$ \eqref{frame.hom}. For example, we may consider the frame $e_I$ obtained by applying the Gram-Schmidt process to $\widetilde{e}_I$ relative to $g$. Since $g_{ij}$ is assumed to be sufficiently close to $\widetilde{g}_{ij}$ initially, so will the corresponding frame components $e_I^i,\widetilde{e}_I^i$. Having defined the frame $e_I$, the initial data for the rest of the reduced variables are readily induced by contracting the former with the original initial data set. In order to satisfy the initial closeness assumption \eqref{thm.main.init}, sufficient smallness in a norm of comparable order for $g_{ij}-\widetilde{g}_{ij}$ is needed.

\subsection{The resulting equations for the perturbed variables minus the homogeneous background variables}

First, notice that combining the gauge condition \eqref{parab.gauge} with \eqref{red.var.hom} gives the following relation:
\begin{align}\label{trk}
\text{tr}k=-\big[3H+\mathcal{O}(e^{-2Ht})+n-1\big].
\end{align}
In this subsection, we derive the equations satisfied by the differences:
\begin{align}\label{diff.var}
\begin{split}
\widehat{k}_{IJ}=k_{IJ}-\widetilde{k}_{IJ},\qquad
\widehat{\gamma}_{IJB}=\gamma_{IJB}-\widetilde{\gamma}_{IJB},\qquad
\widehat{n}=n-\widetilde{n}=n-1,\qquad
\widehat{e}^i_I=e^i_I-\widetilde{e}^i_I,\\
\widehat{v}_0=\rho^{r_s}u_0-\widetilde{\rho}^{r_s}\widetilde{u}_0,\qquad\widehat{v}_I=\rho^{r_s}u_I-\widetilde{\rho}^{r_s}\widetilde{u}_I,\qquad
\widehat{\rho^{2r_s}}=\rho^{2r_s}-\widetilde{\rho}^{2r_s},
\end{split}
\end{align}
using Lemma \ref{lem:red.eq}. More generally, we will use the notation $\widehat{\rho^\alpha}=\rho^\alpha-\widetilde{\rho}^\alpha$, for positive powers $\alpha>0$.
Note that since the background variables \eqref{red.var.hom} are homogeneous, it holds $e_I(\rho^{2r_s})=e_I(\widehat{\rho^{2r_s}})$ etc.
\begin{lemma}\label{lem:red.hat.eq}
The geometric variables $\widehat{k}_{IJ},\widehat{\gamma}_{IJB},\widehat{e}^i_I,\widehat{n}$ satisfy the evolution equations: 
\begin{align}
\label{k.hat.eq}\partial_t\widehat{k}_{IJ}+3H\widehat{k}_{IJ}=&-e_I e_J \widehat{n}+ne_C \widehat{\gamma}_{IJC}-ne_I\widehat{\gamma}_{CJC}\\
\notag&+\gamma_{IJC} e_C \widehat{n}-n\widehat{\gamma}_{CID}\gamma_{DJC}-n\widehat{\gamma}_{IJD}\gamma_{CCD}
-n\widetilde{\gamma}_{CID}\widehat{\gamma}_{DJC}
-n\widetilde{\gamma}_{IJD}\widehat{\gamma}_{CCD}\\
\notag&-(1+c_s^2)n(\widehat{\rho^{1-2r_s}} v_I v_J+\widetilde{\rho}^{1-2r_s} \widehat{v}_I v_J+\widetilde{\rho}^{1-2r_s} \widetilde{v}_I \widehat{v}_J)
-\frac{1}{2}\delta_{IJ}(1-c_s^2)\widehat{n}\rho\\
\notag&-\frac{1}{2}\delta_{IJ}(1-c_s^2)\widehat{\rho}
+\delta_{IJ}H(\widehat{n}+\widehat{n}^2)+\mathcal{O}(e^{-2Ht})\widehat{n}+\delta_{IJ}\mathcal{O}(e^{-2Ht})\widehat{n}^2\\
\notag&-\big[3H+\widehat{n}+\mathcal{O}(e^{-2Ht})\big]\widehat{n}\widehat{k}_{IJ}
-\big[\widehat{n}+\mathcal{O}(e^{-2Ht})\big]\widehat{k}_{IJ},\\
\label{gamma.hat.eq}\partial_t\widehat{\gamma}_{IJB}+H\widehat{\gamma}_{IJB}=&\,ne_B \widehat{k}_{IJ}-ne_J \widehat{k}_{BI}\\
\notag&-n\widehat{k}_{IC}\gamma_{BJC}-n\widehat{k}_{CJ} \gamma_{BIC}
+n\widehat{k}_{IC} \gamma_{JBC}
+n\widehat{k}_{BC} \gamma_{JIC}+n\widehat{k}_{IC} \gamma_{CJB}\\
\notag&-H\widehat{n}\gamma_{IJB}+\widehat{n}\mathcal{O}(e^{-2Ht})\gamma_{BJI}
+\widehat{n}\mathcal{O}(e^{-2Ht})\gamma_{BIJ}
+\widehat{n}\mathcal{O}(e^{-2Ht})\gamma_{JBI}\\
\notag&+\widehat{n}\mathcal{O}(e^{-2Ht})\gamma_{JIB}
+\widehat{n}\mathcal{O}(e^{-2Ht})\gamma_{IJB}
+\mathcal{O}(e^{-2Ht})\widehat{\gamma}_{BJI}
+\mathcal{O}(e^{-2Ht})\widehat{\gamma}_{BIJ}\\
\notag&+\mathcal{O}(e^{-2Ht})\widehat{\gamma}_{JBI}
+\mathcal{O}(e^{-2Ht})\widehat{\gamma}_{JIB}
+\mathcal{O}(e^{-2Ht})\widehat{\gamma}_{IJB}
+(e_B \widehat{n})k_{JI}-(e_J \widehat{n})k_{BI},\\
\label{eIi.hat.eq}\partial_t\widehat{e}_I^i+H\widehat{e}_I^i=&\,\widehat{n}k_{IC}e_C^i+\widehat{k}_{IC} \widehat{e}_C^i+\mathcal{O}(e^{-2Ht})\widehat{e}_I^i,\\
\label{n.hat.eq}\partial_t\widehat{n}+2H\widehat{n}=&\,e_C e_C \widehat{n}
-\gamma_{CCD}e_D\widehat{n}-n\widehat{k}_{CD}\widehat{k}_{CD}
+\mathcal{O}(e^{-2Ht})n\widehat{k}_{CC}\\
\notag&+(1+c_s^2)\big\{\widehat{n}\rho^{1-2r_s}v_Cv_C+\widehat{\rho^{1-2r_s}}v_Cv_C+\widetilde{\rho}^{1-2r_s}\widehat{v}_Cv_C
+\widetilde{\rho}^{1-2r_s}\widetilde{v}_1\widehat{v}_1\big\}\\
\notag&+\frac{3}{2}(1-c_s^2)\widehat{n}\rho+\frac{3}{2}(1-c_s^2)\widehat{\rho}
-2H\widehat{n}^2+\mathcal{O}(e^{-2Ht})\widehat{n}.
\end{align}
The fluid variables $\widehat{v}_I,\widehat{\rho^{2r_s}}$ satisfy the evolution equations:
\begin{align}
\label{vI.hat.eq}
\notag\partial_t\widehat{v}_I+H\widehat{v}_I=&\,\frac{1}{2}\frac{n}{v_0}e_I(\widehat{\rho^{2r_s}})+n\frac{v_C}{v_0}e_C\widehat{v}_I\\
&+(e_I\widehat{n})v_0+n\gamma_{CDI}\frac{v_D}{v_0}\widehat{v}_C
+n\gamma_{1DI}\frac{\widetilde{v}_1}{v_0}\widehat{v}_D
+\frac{n}{v_0}\widetilde{v}_1^2\widehat{\gamma}_{11I}\\
\notag&+\widehat{n}k_{CI}v_C+\widehat{k}_{CI}v_C+(\widetilde{k}_{CI}+\delta_{CI}H)\widehat{v}_C
\end{align}
and
\begin{align}
\label{rho.hat.eq}
\notag&\big(\frac{1-2r_s}{2r_s}+\frac{1}{2}\frac{\rho^{2r_s}}{v_0^2}\big)\partial_t(\widehat{\rho^{2r_s}})+2H\widehat{\rho^{2r_s}}\\
=&-n\frac{\rho^{2r_s}}{v_0^2}\frac{v_Iv_D}{v_0}e_D\widehat{v}_I
+n\frac{\rho^{2r_s}}{v_0}e_C\widehat{v}_C+n\big(\frac{1-2r_s}{2r_s}-\frac{1}{2}\frac{\rho^{2r_s}}{v_0^2}\big)\frac{v_C}{v_0}e_C (\widehat{\rho^{2r_s}})\\
\notag&-\frac{\rho^{2r_s}}{v_0^2}\frac{v_I}{v_0}
\big\{(e_I\widehat{n})v_0^2+n\gamma_{CDI}\widehat{v}_Cv_D
+n\gamma_{1DI}\widetilde{v}_1\widehat{v}_D+n\widehat{\gamma}_{11I}\widetilde{v}_1^2\big\}
+\frac{\rho^{2r_s}}{v_0}\big\{(e_C\widehat{n})v_C
+n\widehat{\gamma}_{CDC}v_D\big\}\\
\notag&-\frac{1}{2}\frac{\widehat{\rho^{2r_s}}}{v_0^2}\partial_t(\widetilde{\rho}^{2r_s})
+\frac{1}{2}\frac{\widetilde{\rho}^{2r_s}}{v_0^2\widetilde{v}^2_0}\partial_t(\widetilde{\rho}^{2r_s})(v_0+\widetilde{v}_0)\widehat{v}_0
-\widehat{n}\big(\frac{v_Cv_I}{v_0^2}k_{CI}-\mathrm{tr}k\big)\rho^{2r_s}
-\rho^{2r_s}\widehat{n}\\
\notag&+\mathcal{O}(e^{-2Ht})\widehat{\rho^{2r_s}}
-\frac{v_Cv_I}{v_0^2}\widehat{k}_{CI}\rho^{2r_s}
-H\frac{\widehat{\rho^{2r_s}}}{v_0^2}(\rho^{2r_s}+\widetilde{\rho}^{2r_s})
+H\frac{\widetilde{\rho}^{4r_s}}{v_0^2\widetilde{v}_0^2}(v_0+\widetilde{v}_0)\widehat{v}_0\\
\notag&-(\widetilde{k}_{CI}+\delta_{CI}H)
\big(\frac{v_Cv_I}{v_0^2}\widehat{\rho^{2r_s}}+\frac{\widehat{v}_Cv_I}{v_0^2}\widetilde{\rho}^{2r_s}+\frac{\widetilde{v}_C\widehat{v}_I}{v_0^2}\widetilde{\rho}^{2r_s}+\frac{\widetilde{v}_C\widetilde{v}_I}{v_0^2\widetilde{v}_0^2}(v_0+\widetilde{v}_0)\widehat{v}_0\widetilde{\rho}^{2r_s}\big)
\end{align}
where $v_0^2=v_Cv_C+\rho^{2r_s}$.

Also, the following constraint equation holds:
\begin{align}\label{mom.hat.const}
e_C \widehat{k}_{CI}
=&-e_I\widehat{n}-  
\widehat{k}_{ID} \gamma_{CDC}
+\widehat{k}_{CD} \gamma_{CID}-(\widetilde{k}_{ID}+\delta_{ID}H)\widehat{\gamma}_{CDC}+(\widetilde{k}_{CD}+\delta_{CD}H)\widehat{\gamma}_{CID}\\
\notag&-(1+c_s^2)\big\{\widehat{\rho^{1-2r_s}}v_0v_I+\widetilde{\rho}^{1-2r_s}\widehat{v}_0v_I+\widetilde{\rho}^{1-2r_s}\widetilde{v}_0\widehat{v}_I\big\}.
\end{align}
\end{lemma}
\begin{remark}
Notice that the equations \eqref{vI.hat.eq}-\eqref{rho.hat.eq} are not symmetric hyperbolic in $\widehat{v}_I,\widehat{\rho^{2r_s}}$. Nevertheless, they are symmetrizable, that is to say, energy estimates can be derived by considering suitable combinations of $\widehat{v}_I,\widehat{\rho^{2r_s}}$ at top order, see Section \ref{S:top.est}. 
\end{remark}
\begin{proof}
The equations for the differences follow from the ones in Lemmas \ref{lem:red.eq}, \ref{lem:red.eq2} and the fact that the background variables are a solution to the reduced system of equations and that they are homogeneous. For example, using \eqref{k.eq} we have
\begin{align*}
\partial_t\widehat{k}_{IJ}=&-n(n-1-\mathrm{tr}\widetilde{k})k_{IJ}- e_I e_J n+ne_C \gamma_{IJC}-ne_I\gamma_{CJC}
+\gamma_{IJC} e_C n-n\gamma_{CID}\gamma_{DJC}\\
\notag&-n\gamma_{IJD}\gamma_{CCD}
-\Lambda\delta_{IJ}n-(1+c_s^2)n\rho^{1-2r_s} v_I v_J-\frac{1}{2}\delta_{IJ}(1-c_s^2)n\rho-\partial_t\widetilde{k}_{IJ}\\
=&-\big[3H+\widehat{n}+\mathcal{O}(e^{-2Ht})\big]\widehat{k}_{IJ}
-\big[3H+\widehat{n}+\mathcal{O}(e^{-2Ht})\big]\widehat{n}\widehat{k}_{IJ}
+\delta_{IJ}\big[H+\mathcal{O}(e^{-2Ht}\big](\widehat{n}+\widehat{n}^2)\\
\notag&+\delta_{IJ}\big[3H^2+\mathcal{O}(e^{-2Ht})\big]\widehat{n}
+\mathrm{tr}\widetilde{k}\widetilde{k}_{IJ} 
- e_I e_J \widehat{n}+ne_C \widehat{\gamma}_{IJC}-ne_I\widehat{\gamma}_{CJC}
+\gamma_{IJC} e_C \widehat{n}\\
\notag&-n\widehat{\gamma}_{CID}\gamma_{DJC}-n\widehat{\gamma}_{IJD}\gamma_{CCD}
-n\widetilde{\gamma}_{CID}\widehat{\gamma}_{DJC}-n\widetilde{\gamma}_{IJD}\widehat{\gamma}_{CCD}
-\widetilde{\gamma}_{CID}\widetilde{\gamma}_{DJC}-\widetilde{\gamma}_{IJD}\widetilde{\gamma}_{CCD}\\
\notag&+\mathcal{O}(e^{-2Ht})\widehat{n}
-\Lambda\delta_{IJ}\widehat{n}-\Lambda\delta_{IJ}
-(1+c_s^2)n(\widehat{\rho^{1-2r_s}} v_I v_J+\widetilde{\rho}^{1-2r_s} \widehat{v}_I v_J+\widetilde{\rho}^{1-2r_s} \widetilde{v}_I \widehat{v}_J)\\
\notag&-(1+c_s^2)\widetilde{\rho}^{1-2r_s}\widetilde{v}_I\widetilde{v}_J
+\mathcal{O}(e^{-4Ht})\widehat{n}
-\frac{1}{2}\delta_{IJ}(1-c_s^2)\widehat{n}\rho-\frac{1}{2}\delta_{IJ}(1-c_s^2)\widehat{\rho}-\frac{1}{2}\delta_{IJ}(1-c_s^2)\widetilde{\rho}-\partial_t\widetilde{k}_{IJ}.
\end{align*}
Now \eqref{k.hat.eq} follows by noticing that the purely homogeneous terms cancel out, as well as $3\delta_{IJ}H^2\widehat{n}$ with $-\Lambda\delta_{IJ}\widehat{n}$, since $3H^2=\Lambda$.

We continue with the computations for \eqref{gamma.hat.eq}, using \eqref{gamma.eq}, \eqref{red.var.hom}, \eqref{diff.var}:
\begin{align*}
\partial_t\widehat{\gamma}_{IJB}=&\,ne_B k_{IJ}-ne_Jk_{BI}-nk_{IC}\gamma_{BJC}-nk_{CJ} \gamma_{BIC}+nk_{IC} \gamma_{JBC}
+nk_{BC} \gamma_{JIC}+nk_{IC} \gamma_{CJB}\\
&+(e_Bn)k_{JI}-(e_J n)k_{BI}-\partial_t\widetilde{\gamma}_{IJB}\\
=&\,ne_B \widehat{k}_{IJ}-ne_J \widehat{k}_{BI}
-n\widehat{k}_{IC}\gamma_{BJC}-n\widehat{k}_{CJ} \gamma_{BIC}
+n\widehat{k}_{IC} \gamma_{JBC}
+n\widehat{k}_{BC} \gamma_{JIC}+n\widehat{k}_{IC} \gamma_{CJB}\\
\notag&-nH\gamma_{IJB}+n\mathcal{O}(e^{-2Ht})\gamma_{BJI}
+n\mathcal{O}(e^{-2Ht})\gamma_{BIJ}
+n\mathcal{O}(e^{-2Ht})\gamma_{JBI}\\
\notag&+n\mathcal{O}(e^{-2Ht})\gamma_{JIB}
+n\mathcal{O}(e^{-2Ht})\gamma_{IJB}
+(e_B \widehat{n})k_{JI}-(e_J \widehat{n})k_{BI}\\
=&\,ne_B \widehat{k}_{IJ}-ne_J \widehat{k}_{BI}
-n\widehat{k}_{IC}\gamma_{BJC}-n\widehat{k}_{CJ} \gamma_{BIC}
+n\widehat{k}_{IC} \gamma_{JBC}
+n\widehat{k}_{BC} \gamma_{JIC}+n\widehat{k}_{IC} \gamma_{CJB}\\
\notag&-H\widehat{\gamma}_{IJB}-H\widetilde{\gamma}_{IJB}-\widehat{n}H\gamma_{IJB}+\widehat{n}\mathcal{O}(e^{-2Ht})\gamma_{BJI}
+\widehat{n}\mathcal{O}(e^{-2Ht})\gamma_{BIJ}
+\widehat{n}\mathcal{O}(e^{-2Ht})\gamma_{JBI}\\
\notag&+\widehat{n}\mathcal{O}(e^{-2Ht})\gamma_{JIB}
+\widehat{n}\mathcal{O}(e^{-2Ht})\gamma_{IJB}
+\mathcal{O}(e^{-2Ht})\widehat{\gamma}_{BJI}
+\mathcal{O}(e^{-2Ht})\widehat{\gamma}_{BIJ}
+\mathcal{O}(e^{-2Ht})\widehat{\gamma}_{JBI}\\
\notag&+\mathcal{O}(e^{-2Ht})\widehat{\gamma}_{JIB}
+\mathcal{O}(e^{-2Ht})\widehat{\gamma}_{IJB}
+\mathcal{O}(e^{-2Ht})\widetilde{\gamma}_{BJI}
+\mathcal{O}(e^{-2Ht})\widetilde{\gamma}_{BIJ}
+\mathcal{O}(e^{-2Ht})\widetilde{\gamma}_{JBI}\\
\notag&+\mathcal{O}(e^{-2Ht})\widetilde{\gamma}_{JIB}
+\mathcal{O}(e^{-2Ht})\widetilde{\gamma}_{IJB}
+(e_B \widehat{n})k_{JI}-(e_J \widehat{n})k_{BI}.
\end{align*}
Since $\widehat{\gamma}=\widehat{n}=\widehat{k}=0$ is a solution, the purely homogeneous terms cancel, leaving the asserted equation.
The computations for \eqref{eIi.hat.eq} are straightforward: 
\begin{align*}
\partial_t \widehat{e}^i_I\overset{\eqref{eIi.eq}}{=}nk_{IC}e^i_C-\partial_t\widetilde{e}^i_I
=&\,\widehat{n}k_{IC}e^i_C+\widehat{k}_{IC}e^i_C-\big[H+\mathcal{O}(e^{-2Ht})\big]e^i_I-\partial_t\widetilde{e}^i_I\\
=&\,\widehat{n}k_{IC}e^i_C+\widehat{k}_{IC}\widehat{e}^i_C-\big[H+\mathcal{O}(e^{-2Ht})\big]\widehat{e}^i_I,
\end{align*}
where in the last equality we used \eqref{eIi.eq} for the homogeneous variable $\widetilde{e}_I^i$.
For the equation \eqref{n.hat.eq}, we rewrite the RHS of \eqref{n.eq}:
\begin{align*}
&-\gamma_{CCD}e_Dn-nk_{CD}k_{CD}+n\Lambda
+(1+c_s^2)n\rho^{1-2r_s} v_Cv_C+\frac{3}{2}(1-c_s^2)n\rho+\partial_t\mathrm{tr}\widetilde{k}\\
=&-\gamma_{CCD}e_D\widehat{n}-n\widehat{k}_{CD}\widehat{k}_{CD}-2n\widetilde{k}_{CD}\widehat{k}_{CD}-\widetilde{k}_{CD}\widetilde{k}_{CD}-3H^2\widehat{n}+\mathcal{O}(e^{-2Ht})\widehat{n}+\Lambda+\Lambda\widehat{n}\\
&+(1+c_s^2)\big\{\widehat{n}\rho^{1-2r_s}v_Cv_C+\widehat{\rho^{1-2r_s}}v_Cv_C+\widetilde{\rho}^{1-2r_s}\widehat{v}_Cv_C
+\widetilde{\rho}^{1-2r_s}\widetilde{v}_1\widehat{v}_1\big\}+(1+c_s^2)\widetilde{\rho}^{1-2r_s}\widetilde{v}_1^2\\
&+\frac{3}{2}(1-c_s^2)\widehat{n}\rho+\frac{3}{2}(1-c_s^2)\widehat{\rho}+\frac{3}{2}(1-c_s^2)\widetilde{\rho}+\partial_t\mathrm{tr}\widetilde{k}\\
\tag{$H^2=\frac{\Lambda}{3}$}=&-\gamma_{CCD}e_D\widehat{n}-n\widehat{k}_{CD}\widehat{k}_{CD}-2Hn\widehat{n}+\mathcal{O}(e^{-2Ht})n\widehat{k}_{CC}+\mathcal{O}(e^{-2Ht})\widehat{n}\\
&+(1+c_s^2)\big\{\widehat{n}\rho^{1-2r_s}v_Cv_C+\widehat{\rho^{1-2r_s}}v_Cv_C+\widetilde{\rho}^{1-2r_s}\widehat{v}_Cv_C
+\widetilde{\rho}^{1-2r_s}\widetilde{v}_1\widehat{v}_1\big\}\\
&+\frac{3}{2}(1-c_s^2)\widehat{n}\rho+\frac{3}{2}(1-c_s^2)\widehat{\rho},
\end{align*}
where in the last equality we used the fact that the purely homogeneous terms cancel out (since the background is a solution), and we plugged in \eqref{parab.gauge} in the form $\widehat{k}_{CC}=-\widehat{n}$. Hence, we obtain the desired equation. 

The constraint equation \eqref{mom.hat.const} follows directly from \eqref{momconst} by plugging in \eqref{diff.var} and using the anti-symmetry of $\gamma_{CCD}$:
\begin{align*}
e_C\widehat{k}_{CI}=&\,e_Ck_{CI}=-e_I\widehat{n}+k_{ID}\gamma_{CCD}+k_{CD}\gamma_{CID}-(1+c_s^2)\rho^{1-2r_s} v_0v_I\\
=&-e_I\widehat{n}-\widehat{k}_{ID}\gamma_{CDC}+\widehat{k}_{CD}\gamma_{CID}-(\widetilde{k}_{ID}+\delta_{ID}H)\gamma_{CDC}+(\widetilde{k}_{CD}+\delta_{CD}H)\gamma_{CID}\\
&-(1+c_s^2)\big\{\widehat{\rho^{1-2r_s}}v_0v_I+\widetilde{\rho}^{1-2r_s}\widehat{v}_0v_I+\widetilde{\rho}^{1-2r_s}\widetilde{v}_0\widehat{v}_I\big\}-(1+c_s^2)\widetilde{\rho}^{1-2r_s}\widetilde{v}_0\widetilde{v}_I\\
=&-e_I\widehat{n}+\widehat{k}_{ID}\gamma_{CCD}+\widehat{k}_{CD}\gamma_{CID}
-(\widetilde{k}_{ID}+\delta_{ID}H)\widehat{\gamma}_{CDC}+(\widetilde{k}_{CD}+\delta_{CD}H)\widehat{\gamma}_{CID}\\
&-(1+c_s^2)\big\{\widehat{\rho^{1-2r_s}}v_0v_I+\widetilde{\rho}^{1-2r_s}\widehat{v}_0v_I+\widetilde{\rho}^{1-2r_s}\widetilde{v}_0\widehat{v}_I\big\}.
\end{align*}

Next, we employ \eqref{vI.eq} to compute:
\begin{align*}
\partial_t\widehat{v}_I=&\,\frac{n}{v_0}\big\{v_C e_C \widehat{v}_I+k_{CI}v_0v_C+\frac{1}{2}e_I(\widehat{\rho^{2r_s}})+n^{-1}(e_I\widehat{n})v_0^2+\gamma_{CDI}v_Cv_D\big\}-\partial_t\widetilde{v}_I\\
=&\,n\frac{v_C}{v_0}e_C \widehat{v}_I+\widehat{n}k_{CI}v_C+\widehat{k}_{CI}v_C
+\widetilde{k}_{CI}\widehat{v}_C+\widetilde{k}_{1I}\widetilde{v}_1\\
&+\frac{1}{2}\frac{n}{v_0}e_I(\widehat{\rho^{2r_s}})+(e_I\widehat{n})v_0+\frac{n}{v_0}\gamma_{CDI}\widehat{v}_Cv_D
+\frac{n}{v_0}\gamma_{1DI}\widetilde{v}_1\widehat{v}_D
+\frac{n}{v_0}\widehat{\gamma}_{11I}\widetilde{v}_1^2
-\partial_t\widetilde{v}_I,
\end{align*}
where we used the vanishing of $\widetilde{\gamma}_{11I}=0$, see \eqref{gamma.hom2}. 
We arrive at \eqref{vI.hat.eq} after dropping the homogeneous terms.

Finally, for the equation \eqref{rho.hat.eq}, we plug in \eqref{rho.eq2} and \eqref{trk}:
\begin{align*}
&\big(\frac{1-2r_s}{2r_s}+\frac{1}{2}\frac{\rho^{2r_s}}{v_0^2}\big)\partial_t(\widehat{\rho^{2r_s}})\\
=&-n\big(\frac{v_Cv_I}{v_0^2}k_{CI}-\mathrm{tr}k\big)\rho^{2r_s}
-n\frac{\rho^{2r_s}}{v_0^2}\frac{v_Iv_D}{v_0}e_D\widehat{v}_I
+n\frac{\rho^{2r_s}}{v_0}e_C\widehat{v}_C+n\big(\frac{1-2r_s}{2r_s}-\frac{1}{2}\frac{\rho^{2r_s}}{v_0^2}\big)\frac{v_C}{v_0}e_C(\widehat{\rho^{2r_s}})\\
&-n\frac{\rho^{2r_s}}{v_0^2}\frac{v_I}{v_0}
\big\{n^{-1}(e_In)v_0^2+\gamma_{CDI}v_Cv_D\big\}
+n\frac{\rho^{2r_s}}{v_0}\big\{n^{-1}(e_Cn)v_C+\gamma_{CDC}v_D\big\}
+\big(\frac{1-2r_s}{2r_s}+\frac{1}{2}\frac{\rho^{2r_s}}{v_0^2}\big)\partial_t(\widetilde{\rho}^{2r_s})\\
=&-\widehat{n}\big(\frac{v_Cv_I}{v_0^2}k_{CI}-\mathrm{tr}k\big)\rho^{2r_s}
-\frac{v_Cv_I}{v_0^2}\widehat{k}_{CI}\rho^{2r_s}+\big\{\frac{v_0^2-\rho^{2r_s}}{v_0^2}H-3H+\mathcal{O}(e^{-2Ht})\big\}\widehat{\rho^{2r_s}}\\
&-(\widetilde{k}_{CI}+H\delta_{CI})\big(\frac{v_Cv_I}{v_0^2}\widehat{\rho^{2r_s}}+\frac{\widehat{v}_Cv_I}{v_0^2}\widetilde{\rho}^{2r_s}+\frac{\widetilde{v}_C\widehat{v}_I}{v_0^2}\widetilde{\rho}^{2r_s}+\frac{\widetilde{v}_C\widetilde{v}_I}{v_0^2\widetilde{v}_0^2}(v_0+\widetilde{v}_0)\widehat{v}_0\widetilde{\rho}^{2r_s}+\frac{\widetilde{v}_1^2}{\widetilde{v}_0^2}\widetilde{\rho}^{2r_s}\big)\\
&-\big\{\frac{\widehat{\rho^{2r_s}}}{v_0^2}H+\frac{\widetilde{\rho}^{2r_s}}{v_0^2}H
-\frac{\widetilde{\rho}^{2r_s}}{v_0^2\widetilde{v}_0^2}(v_0+\widetilde{v}_0)
\widehat{v}_0H+\frac{\widetilde{\rho}^{2r_s}}{\widetilde{v}_0^2}H+2H+\mathcal{O}(e^{-2Ht})\big\}\widetilde{\rho}^{2r_s}
-\widehat{n}\rho^{2r_s}\\
&-n\frac{\rho^{2r_s}}{v_0^2}\frac{v_Iv_D}{v_0}e_D\widehat{v}_I
+n\frac{\rho^{2r_s}}{v_0}e_C\widehat{v}_C+n\big(\frac{1-2r_s}{2r_s}-\frac{1}{2}\frac{\rho^{2r_s}}{v_0^2}\big)\frac{v_C}{v_0}e_C(\widehat{\rho^{2r_s}})\\
&-\frac{\rho^{2r_s}}{v_0^2}\frac{v_I}{v_0}
\big\{(e_I\widehat{n})v_0^2++n\gamma_{CDI}\widehat{v}_Cv_D
+n\gamma_{1DI}\widetilde{v}_1\widehat{v}_D+n\widehat{\gamma}_{11I}\widetilde{v}_1^2\big\}\\
&+\frac{\rho^{2r_s}}{v_0}\big\{(e_C\widehat{n})v_C+\widehat{\gamma}_{CDC}v_D\big\}
+\big\{\frac{1-2r_s}{2r_s}+\frac{1}{2}\frac{\widehat{\rho^{2r_s}}}{v_0^2}
-\frac{1}{2}\frac{\widetilde{\rho}^{2r_s}}{v_0^2\widetilde{v}_0^2}(v_0+\widetilde{v}_0)\widehat{v}_0+\frac{1}{2}\frac{\widetilde{\rho}^{2r_s}}{\widetilde{v}_0^2}\big\}\partial_t(\widetilde{\rho}^{2r_s}),
\end{align*}
where we made use of the identities $\widetilde{\gamma}_{CDC}=\widetilde{\gamma}_{11I}=0$, see \eqref{gamma.hom2}.
The desired equation is obtained by dropping all purely homogeneous terms.
\end{proof}
\section{Norms, total energy, and bootstrap assumptions}\label{sec:Boots}
\subsection{Norms of the unknown variables expressed relative to the orthonormal frame}\label{SS:SOBOLEVNORMS}

We define the norm $\|f\|_{L^2(\Sigma_t)}$ for a scalar function by
\begin{align}\label{L2}
\|f\|_{L^2(\Sigma_t)}^2:=\int_{\Sigma_t}f^2(t,\omega)d\mathbb{S}^3,
\end{align}
where $d\mathbb{S}^3$ is the canonical volume form on $\mathbb{S}^3$. 
Define also the corresponding $H^M(\Sigma_t),W^{M,\infty}(\Sigma_t)$ norms:
\begin{align}\label{HM}
\|f\|_{H^M(\Sigma_t)}^2=\sum_{|\iota|\leq M}\|Y^\iota f\|_{L^2(\Sigma_t)},\qquad\|f\|_{W^{M,\infty}(\Sigma_t)}=\sum_{|\iota|\leq M}\|Y^\iota f\|_{L^\infty(\Sigma_t)},
\end{align}
where is $\iota$ is a spatial multi-index and $Y^\iota$ is the operator that acts on $v$ with repeated differentiation with respect to the vector fields $Y_1,Y_2,Y_3$. 

For the indexed variables that we study, their corresponding $H^M(\Sigma_t),W^{M,\infty}(\Sigma_t)$ norms are simply the sum of their components:
\begin{align}\label{HM.hat}
\begin{split}
\|\widehat{k}\|_{H^M(\Sigma_t)}^2=\sum_{I,J=1}^3\|\widehat{k}_{IJ}\|_{H^M(\Sigma_t)}^2,&\qquad\|\widehat{k}\|_{W^{M,\infty}(\Sigma_t)}^2=\sum_{I,J=1}^3\|\widehat{k}_{IJ}\|_{W^{M,\infty}(\Sigma_t)}^2,\\
\|\widehat{\gamma}\|_{H^M(\Sigma_t)}^2=\sum_{I,J,B=1}^3\|\widehat{\gamma}_{IJB}\|_{H^M(\Sigma_t)}^2,&\qquad\|\widehat{\gamma}\|_{W^{M,\infty}(\Sigma_t)}^2=\sum_{I,J,B=1}^3\|\widehat{\gamma}_{IJB}\|_{W^{M,\infty}(\Sigma_t)}^2,\\
\|\widehat{e}\|_{H^M(\Sigma_t)}^2=\sum_{i,I=1}^3\|\widehat{e}^i_I\|_{H^M(\Sigma_t)}^2,&\qquad\|\widehat{e}\|_{W^{M,\infty}(\Sigma_t)}^2=\sum_{i,I=1}^3\|\widehat{e}^i_I\|_{W^{M,\infty}(\Sigma_t)}^2,\\
\|\widehat{v}\|_{H^M(\Sigma_t)}^2=\sum_{\mu=0}^3\|\widehat{v}_\mu\|_{H^M(\Sigma_t)}^2,&\qquad\|\widehat{v}\|_{W^{M,\infty}(\Sigma_t)}^2=\sum_{\mu=0}^3\|\widehat{v}_\mu\|_{W^{M,\infty}(\Sigma_t)}^2.
\end{split}
\end{align}
Define the energy of the geometric variables:
\begin{align}\label{geom.en}
\mathcal{E}_{geom}(t)=&\,e^{2Ht}\big\{\|\widehat{\gamma}\|^2_{H^N(\Sigma_t)}+\|\widehat{e}\|^2_{H^N(\Sigma_t)}+\|\widehat{k}\|^2_{H^N(\Sigma_t)}\}
+e^{3Ht}\|\widehat{n}\|^2_{H^N(\Sigma_t)}
\end{align}
and the following energies of the fluid variables:
\begin{align}\label{fluid.en}
\notag\mathcal{E}_{fluid.low}(t)=&\,
e^{2Ht}\|\widehat{v}\|^2_{H^{N-1}(\Sigma_t)}+e^{\frac{8r_s}{1-2r_s}Ht}\|\widehat{\rho^{2r_s}}\|^2_{H^{N-2}(\Sigma_t)},\\
\mathcal{E}_{fluid.top}(t)=&\,e^{-4A_sHt}\big\{e^{2Ht}\|\widehat{v}\|^2_{\dot{H}^N(\Sigma_t)}+e^{\frac{8r_s}{1-2r_s}Ht}\big(\|\widehat{\rho^{2r_s}}\|^2_{\dot{H}^N(\Sigma_t)}+\|\widehat{\rho^{2r_s}}\|^2_{\dot{H}^{N-1}(\Sigma_t)}\big)\big\}\\
\notag=&\,e^{-4A_sHt}e^{2Ht}\|\widehat{v}\|^2_{\dot{H}^N(\Sigma_t)}+e^{4Ht}\big(\|\widehat{\rho^{2r_s}}\|^2_{\dot{H}^N(\Sigma_t)}+\|\widehat{\rho^{2r_s}}\|^2_{\dot{H}^{N-1}(\Sigma_t)}\big),
\end{align}
where we note the identity $\frac{8r_s}{1-2r_s}-4A_s=4$ and recall the definitions
$$A_s=\frac{3c_s^2-1}{1-c_s^2}\in(0,+\infty),\qquad r_s=\frac{c_s^2}{1+c_s^2}\in(\frac{1}{4},\frac{1}{2}),\qquad\text{for $c_s^2\in(\frac{1}{3},1)$}.$$
On the other hand, observe that 
$$ A_s\in(0,\frac{1}{2})\qquad
\text{and} \qquad r_s\in(\frac{1}{4},\frac{3}{10}) \qquad \text{for $c_s^2\in(\frac{1}{3},\frac{3}{7})$}.$$
Define also the total energy:
\begin{align}\label{tot.en}
\mathcal{E}_{tot}(t)=\mathcal{E}_{geom}(t)+\mathcal{E}_{fluid.low}(t)
+\mathcal{E}_{fluid.top}(t).
\end{align}
\begin{remark}
There is a $c_s$-dependent discrepancy in the weights (powers of $e^{Ht}$) used in the top order energy in \eqref{fluid.en} compared with the lower order energy in \eqref{fluid.en} for the fluid variables. The larger $c_s$ is, the larger the discrepancy. In fact, it becomes infinite as the sound speed $c_s$ tends to $1$.\footnote{Our method of proof can only tolerate the range of sound speeds $c_s^2\in(\frac{1}{3},\frac{3}{7})$. In Appendix \ref{app:Euler} all sound speeds in the range $c_s^2\in(\frac{1}{3},1)$ are allowed, albeit for the relativistic Euler equations in a fixed FLRW background.} 
This comes from a degeneracy in the top order energy estimates that we derive in Section \ref{S:top.est}, see Proposition \ref{prop:fluid.est}. We note that the weights in the $H^{N-1}(\Sigma_t),H^{N-2}(\Sigma_t)$ norms of $\widehat{v},\widehat{\rho^{2r_s}}$ respectively are optimal, that is to say the corresponding renormalized variables $e^{Ht}\widehat{v},e^{\frac{4r_s}{1-2r_s}Ht}\widehat{\rho^{2r_s}}$ have a limit as $t\rightarrow+\infty$. This will be proven in Section \ref{sec:refined.est}, once we have completed our overall continuity argument.

Normally, there should be a comparable (but smaller) discrepancy in the corresponding norms of the geometric variables. However, for the range of sound speeds $c_s^2\in(\frac{1}{3},\frac{3}{7})$ that we are working with, the discrepancy in the weights of the top order energy of the fluid variables is small enough such that it does not propagate to the top order energy of the geometric variables.
\end{remark}
\subsection{Bootstrap assumptions}\label{subsec:Boots}

Our bootstrap assumptions are that there exists a bootstrap time $T_{Boot}\in(T,+\infty)$ such that it holds
\begin{align}\label{Boots}
\mathcal{E}_{tot}(t)\leq \varepsilon , \qquad\forall t\in[T,T_{Boot}),
\end{align}
for a sufficiently small constant $\varepsilon>0$ to be determined below. Such a $T_{Boot}>0$ obviously exists by Cauchy stability and the closeness of the perturbed initial data on $\Sigma_T$ to the homogeneous background initial data.

\section{Future stability estimates}\label{sec:FSE}

In this section, we derive the exponential decay of the variables $\widehat{k}_{IJ},\widehat{\gamma}_{IJB},\widehat{e}^i_I,\widehat{n},\widehat{v}_0,\widehat{v}_I,\widehat{\rho^{2r_s}}$, in a manner that yields a strict 
improvement of our bootstrap assumptions \eqref{Boots}, which in turn by a standard continuation argument implies that $T_{Boot}=+\infty$, ie. the perturbed solution exists for all future time and it satisfies the estimate \eqref{Boots} for all $t\in[T,+\infty)$,
see Proposition \ref{prop:cont.arg} for the final statement.

\subsection{Basic estimates and identities}

We will frequently use the standard Sobolev inequality in $(\mathbb{S}^3,d\omega)$:
\begin{align}\label{Sob}
\|f\|_{L^{\infty}(\Sigma_t)}\leq C\|f\|_{H^2(\Sigma_t)},
\end{align}
where $C>0$ depends on $(\mathbb{S}^3,d\omega)$, $d\omega$ being the standard metric on $\mathbb{S}^3$.
Already, \eqref{Boots} and \eqref{Sob} imply the bounds 
\begin{align}\label{WN-3.est}
\begin{split}
\|\widehat{k}\|_{W^{N-2,\infty}(\Sigma_t)}
+\|\widehat{\gamma}\|_{W^{N-2,\infty}(\Sigma_t)}+\|\widehat{e}\|_{W^{N-2,\infty}(\Sigma_t)}\leq&\,C\varepsilon e^{-Ht},\\
\|\widehat{n}\|_{W^{N-2,\infty}(\Sigma_t)}\leq&\,C\varepsilon e^{-\frac{3}{2}Ht},\\
\|\widehat{v}\|_{W^{N-3,\infty}(\Sigma_t)}\leq&\,C\varepsilon e^{-Ht},\\
\|\widehat{\rho^{2r_s}}\|_{W^{N-4,\infty}(\Sigma_t)}\leq&\,C\varepsilon e^{-\frac{4r_s}{1-2r_s}Ht},
\end{split}
\end{align}
for all $t\in[T,T_{Boot})$, where we note that $\frac{4r_s}{1-2r_s}>2$ for $c_s^2>\frac{1}{3}$. Other power differences of $\rho$ satisfy the bounds: 
\begin{align}\label{rho.hat.boots}
\begin{split}
e^{\frac{2}{1-2r_s}Ht}\big(\|\widehat{\rho}\|_{W^{N-4}(\Sigma_t)}
+\|\widehat{\rho}\|_{H^{N-2}(\Sigma_t)}\big)
+e^{\frac{1}{r_s}Ht}\big(\|\widehat{\rho}\|_{\dot{H}^N(\Sigma_t)}+\|\widehat{\rho}\|_{\dot{H}^{N-1}(\Sigma_t)}\big)\leq&\,C\varepsilon,\\
e^{2Ht}\big(\|\widehat{\rho^{1-2r_s}}\|_{W^{N-4}(\Sigma_t)}+\|\widehat{\rho^{1-2r_s}}\|_{H^{N-2}(\Sigma_t)}\big)
+e^{\frac{1-2r_s}{r_s}Ht}\big(\|\widehat{\rho^{1-2r_s}}\|_{\dot{H}^N(\Sigma_t)}
+\|\widehat{\rho^{1-2r_s}}\|_{\dot{H}^{N-1}(\Sigma_t)}\big)\leq&\, C\varepsilon.
\end{split}
\end{align}

When we commute the equations \eqref{k.hat.eq}-\eqref{mom.hat.const} with $Y^\iota$, we will use the commutator relation:
\begin{align}\label{comm.eI}
[Y^\iota,e_I]f=\sum_{\iota_1\cup\iota_2=\iota,\,|\iota_2|<|\iota|}(Y^{\iota_1}e^a_I)Y^{\iota_2}Y_af.
\end{align}
Also, since $Y_a$ is $d\omega$-Killing, we have the following integration by parts formula relative to $e_I$:
\begin{align}\label{IBP}
\int_{\Sigma_t}f_1 (e_If_2) d\omega=-\int_{\Sigma_t}[(e_If_1)f_2+(Y_ie^i_I)f_1f_2]d\omega.
\end{align}

In the derivations of the energy estimates below, it will be useful to control the following variables in $L^\infty(\Sigma_t)$.
\begin{lemma}\label{lem:Linfty.est}
Assume that the bootstrap assumptions \eqref{Boots} are valid for $\varepsilon$ sufficiently small and $T$ sufficiently large. Then there exist (non-zero) constants $\widetilde{v}^\infty_0,\widetilde{\rho}^\infty$, depending on the asymptotic data of the background fluid variables, 
such that the following pointwise bounds hold:
\begin{align}\label{Linfty.est}
\begin{split}
\|v_0^{-1}-(\widetilde{v}_0^\infty)^{-1}e^{Ht}\|_{L^\infty(\Sigma_t)}\leq&\, C(\varepsilon+e^{-Ht})e^{Ht},\\
\|\frac{\rho^{2r_s}}{v_0^2}-(\widetilde{\rho}^\infty)^{2r_s}(\widetilde{v}_0^\infty)^{-2}e^{2Ht-\frac{4r_s}{1-2r_s}Ht}\|_{L^\infty(\Sigma_t)}\leq &\,C(\varepsilon+e^{-2Ht})e^{2Ht-\frac{4r_s}{1-2r_s}Ht},\\
\|\partial_t\log v_0 + H\|_{L^\infty(\Sigma_t)}\leq&\,Ce^{-Ht},\\
\|\partial_t\log \rho + \frac{2}{1-2r_s}H\|_{L^\infty(\Sigma_t)}\leq&\,Ce^{2Ht-\frac{4r_s}{1-2r_s}Ht},
\end{split}
\end{align}
for all $t\in[T,T_{Boot})$, where we recall that $\frac{4r_s}{1-2r_s}>2$.
\end{lemma}
\begin{proof}
By Lemmas \ref{lem:hom.est}, \ref{lem:app.hom}, there exist (non-zero) constants $\widetilde{v}_0^\infty,\widetilde{\rho}^\infty$ such that 
\begin{align}\label{v0.rho.tilde.exp}
\begin{split}
\|\widetilde{v}_0-\widetilde{v}_0^{\infty}e^{-Ht}\|_{L^\infty(\Sigma_t)}\leq Ce^{-3Ht},\qquad\|\widetilde{\rho}^{2r_s}-(\widetilde{\rho}^\infty)^{2r_s}e^{-\frac{4r_s}{1-2r_s}Ht}\|\leq Ce^{-\frac{8r_s}{1-2r_s}Ht}.
\end{split}
\end{align}
Then, the first two estimates in \eqref{Linfty.est} follow from the bootstrap assumptions \eqref{Boots} and the triangle inequality. For the third estimate in \eqref{Linfty.est}, we use the identities \eqref{e0v0}, \eqref{e0vI} to write 
\begin{align}
\label{dt.v0}\partial_t\log v_0=&\,\frac{n}{v_0}\frac{v_I}{v_0}\frac{1}{v_0}\big\{k_{CI}v_0v_C+\frac{1}{2}e_I(\rho^{2r_s})+n^{-1}(e_In)v_0^2+\gamma_{CDI}v_Cv_D\big\}\\
\notag&+r_snv_0\frac{\rho^{2r_s}}{v_0^2}\partial_t\log\rho.
\end{align}
Rewriting as well
\begin{align*}
\frac{n}{v_0}\frac{v_I}{v_0}\frac{1}{v_0}k_{CI}v_0v_C=&\,
\widehat{n}\frac{v_I}{v_0}\frac{1}{v_0}k_{CI}v_C+
\frac{v_I}{v_0}\frac{1}{v_0}\widehat{k}_{CI}v_C
+\frac{v_I}{v_0}\frac{1}{v_0}(\widetilde{k}_{CI}+\delta_{CI}H)v_C
-\frac{v_Iv_I}{v_0^2}H\\
=&\,\widehat{n}\frac{v_I}{v_0}\frac{1}{v_0}k_{CI}v_C+
\frac{v_I}{v_0}\frac{1}{v_0}\widehat{k}_{CI}v_C
+\frac{v_I}{v_0}\frac{1}{v_0}(\widetilde{k}_{CI}+\delta_{CI}H)v_C
+\frac{\rho^{2r_s}}{v_0^2}H-H,
\end{align*}
and plugging it in \eqref{dt.v0}, we solve for $\partial_t\log v_0+H$. The desired estimate follows by bounding the resulting RHS in $L^\infty(\Sigma_t)$ using the bootstrap assumptions \eqref{Boots} and the fourth estimate in \eqref{Linfty.est}. 

Thus, it remains to derive the fourth estimate in \eqref{Linfty.est}. For this purpose, we rewrite the equation \eqref{rho.eq2}:
\begin{align}\label{dt.rho}
\notag&\big(1-2r_s+r_s\frac{\rho^{2r_s}}{v_0^2}\big)\partial_t\log\rho
+n\big(\frac{v_Cv_I}{v_0^2}k_{CI}-\mathrm{tr}k\big)\\
=&-\frac{n}{v_0^2}\frac{v_Iv_D}{v_0}e_Dv_I
+\frac{n}{v_0}e_Cv_C+n\big(1-2r_s-r_s\frac{\rho^{2r_s}}{v_0^2}\big)\frac{v_C}{v_0}e_C\log \rho\\
\notag&-\frac{n}{v_0^2}\frac{v_I}{v_0}
\big\{n^{-1}(e_In)v_0^2+\gamma_{CDI}v_Cv_D\big\}
+\frac{n}{v_0}\big\{n^{-1}(e_Cn)v_C+\gamma_{CDC}v_D\big\}.
\end{align}
Rewrite also 
\begin{align*}
n\big(\frac{v_Cv_I}{v_0^2}k_{CI}-\mathrm{tr}k\big)=&\,\widehat{n}\big(\frac{v_Cv_I}{v_0^2}k_{CI}-\mathrm{tr}k\big)+\frac{v_Cv_I}{v_0^2}\widehat{k}_{CI}-\widehat{k}_{CC}\\
&+\frac{v_Cv_I}{v_0^2}(\widetilde{k}_{CI}+\delta_{CI}H)
+\frac{\rho^{2r_s}}{v_0^2}H-(\widetilde{k}_{CC}+3H)+2H
\end{align*}
and plug it into the last expression on the LHS of \eqref{dt.rho}. Then isolate $\big(1-2r_s+r_s\frac{\rho^{2r_s}}{v_0^2}\big)\partial_t( \rho^{2r_s})+2H$ and estimate the RHS using the bootstrap assumptions \eqref{Boots} and the already derived first two estimates in \eqref{Linfty.est} to infer that
\begin{align}\label{dt.rho.est}
\big\|\big(1-2r_s+r_s\frac{\rho^{2r_s}}{v_0^2}\big)\partial_t\log\rho+2H\big\|_{L^\infty(\Sigma_t)}\leq Ce^{-Ht}.
\end{align}
Dividing with the coefficient of $\partial_t\log \rho$ and applying the second estimate in \eqref{Linfty.est} once more gives the desired estimate, since $0>2-\frac{4r_s}{1-2r_s}>-1$, for $c_s^2\in(\frac{1}{3},\frac{3}{7})$, $r_s\in(\frac{1}{4},\frac{3}{10})$.
\end{proof}
\subsection{Energy estimates for the geometric variables}

In this subsection, we derive $H^N(\Sigma_t)$ energy estimates for the geometric variables. First, we commute the equations \eqref{k.hat.eq}-\eqref{n.hat.eq}, \eqref{mom.hat.const} with $\partial^\iota$, $|\iota|\leq N$, using \eqref{comm.eI}: 
\begin{align}
\label{k.hat.diff.eq}\partial_tY^\iota\widehat{k}_{IJ}+3HY^\iota\widehat{k}_{IJ}=&-e_IY^\iota e_J \widehat{n}+ne_CY^\iota \widehat{\gamma}_{IJC}-ne_IY^\iota\widehat{\gamma}_{CJC}+\mathfrak{K}^\iota_{IJ},\\
\label{gamma.hat.diff.eq}\partial_tY^\iota\widehat{\gamma}_{IJB}+HY^\iota\widehat{\gamma}_{IJB}=&\,ne_B Y^\iota\widehat{k}_{IJ}-ne_JY^\iota \widehat{k}_{BI}+\mathfrak{G}_{IJB}^\iota,\\
\label{eIi.hat.diff.eq}\partial_tY^\iota\widehat{e}_I^i+HY^\iota\widehat{e}_I^i=&\,\mathfrak{E}_{Ii}^\iota,\\
\label{n.hat.diff.eq}\partial_tY^\iota\widehat{n}+2HY^\iota\widehat{n}=&\,e_CY^\iota e_C \widehat{n}+\mathfrak{N}^\iota,\\
\label{mom.hat.diff.const}
e_C Y^\iota\widehat{k}_{CI}
=&\,\mathfrak{M}_I^\iota,
\end{align}
where
\begin{align}
\label{mathfrak.K}\mathfrak{K}^\iota_{IJ}=&\,
\sum_{\iota_1\cup\iota_2=\iota,\,|\iota_2|<|\iota|}\big\{Y^{\iota_1}(ne_C^a)Y_aY^{\iota_2} \widehat{\gamma}_{IJC}
-(Y^{\iota_1}e_I^a)Y_aY^{\iota_2} e_J \widehat{n}
-Y^{\iota_1}(ne_I^a)Y_aY^{\iota_2}\widehat{\gamma}_{CJC}\big\}\\
\notag&+Y^\iota\bigg[\gamma_{IJC} e_C \widehat{n}-n\widehat{\gamma}_{CID}\gamma_{DJC}-n\widehat{\gamma}_{IJD}\gamma_{CCD}
-n\widetilde{\gamma}_{CID}\widehat{\gamma}_{DJC}
-n\widetilde{\gamma}_{IJD}\widehat{\gamma}_{CCD}\\
\notag&-(1+c_s^2)n(\widehat{\rho^{1-2r_s}} v_I v_J+\widetilde{\rho}^{1-2r_s} \widehat{v}_I v_J+\widetilde{\rho}^{1-2r_s} \widetilde{v}_I \widehat{v}_J)
-\frac{1}{2}\delta_{IJ}(1-c_s^2)\widehat{n}\rho\\
\notag&-\frac{1}{2}\delta_{IJ}(1-c_s^2)\widehat{\rho}
+\delta_{IJ}H(\widehat{n}+\widehat{n}^2)+\mathcal{O}(e^{-2Ht})\widehat{n}+\delta_{IJ}\mathcal{O}(e^{-2Ht})\widehat{n}^2\\
\notag&-\big[3H+\widehat{n}+\mathcal{O}(e^{-2Ht})\big]\widehat{n}\widehat{k}_{IJ}
-\big[\widehat{n}+\mathcal{O}(e^{-2Ht})\big]\widehat{k}_{IJ}\bigg],\\
\label{mathfrak.G}\mathfrak{G}^\iota_{IJB}=&
\sum_{\iota_1\cup\iota_2=\iota,\,|\iota_2|<|\iota|}\big\{
Y^{\iota_1}(ne_B^a)Y_aY^{\iota_2}\widehat{k}_{IJ}-Y^{\iota_2}(ne_J^a)Y_aY^{\iota_2} \widehat{k}_{BI}
\big\}\\
\notag&+Y^\iota\bigg[n\widehat{k}_{IC} \gamma_{JBC}
+n\widehat{k}_{BC} \gamma_{JIC}+n\widehat{k}_{IC} \gamma_{CJB}
-n\widehat{k}_{IC}\gamma_{BJC}-n\widehat{k}_{CJ} \gamma_{BIC}\\
\notag&-H\widehat{n}\gamma_{IJB}+\widehat{n}\mathcal{O}(e^{-2Ht})\gamma_{BJI}
+\widehat{n}\mathcal{O}(e^{-2Ht})\gamma_{BIJ}
+\widehat{n}\mathcal{O}(e^{-2Ht})\gamma_{JBI}\\
\notag&+\widehat{n}\mathcal{O}(e^{-2Ht})\gamma_{JIB}
+\widehat{n}\mathcal{O}(e^{-2Ht})\gamma_{IJB}
+\mathcal{O}(e^{-2Ht})\widehat{\gamma}_{BJI}
+\mathcal{O}(e^{-2Ht})\widehat{\gamma}_{BIJ}\\
\notag&+\mathcal{O}(e^{-2Ht})\widehat{\gamma}_{JBI}
+\mathcal{O}(e^{-2Ht})\widehat{\gamma}_{JIB}
+\mathcal{O}(e^{-2Ht})\widehat{\gamma}_{IJB}
+(e_B \widehat{n})k_{JI}-(e_J \widehat{n})k_{BI}\bigg],\\
\label{mathfrak.E}\mathfrak{E}_{Ii}^\iota=&\,Y^\iota\big\{\widehat{n}k_{IC}e_C^i+\widehat{k}_{IC} \widehat{e}_C^i+\mathcal{O}(e^{-2Ht})\widehat{e}_I^i\big\},\\
\label{mathfrak.N}\mathfrak{N}^\iota
=&\sum_{\iota_1\cup\iota_2=\iota,\,|\iota_2|<|\iota|}(Y^{\iota_1}e_C^a)Y_aY^\iota e_C\widehat{n}
+Y^\iota\bigg[\mathcal{O}(e^{-2Ht})n\widehat{k}_{CC}-\gamma_{CCD}e_D\widehat{n}-n\widehat{k}_{CD}\widehat{k}_{CD}\\
\notag&+(1+c_s^2)\big\{\widehat{n}\rho^{1-2r_s}v_Cv_C+\widehat{\rho^{1-2r_s}}v_Cv_C+\widetilde{\rho}^{1-2r_s}\widehat{v}_Cv_C
+\widetilde{\rho}^{1-2r_s}\widetilde{v}_1\widehat{v}_1\big\}\\
\notag&+\frac{3}{2}(1-c_s^2)\widehat{n}\rho+\frac{3}{2}(1-c_s^2)\widehat{\rho}
-2H\widehat{n}^2+\mathcal{O}(e^{-2Ht})\widehat{n}\bigg],\\
\label{mathfrak.M}\mathfrak{M}_I^\iota=&
-\sum_{\iota_1\cup\iota_2=\iota,\,|\iota_2|<|\iota|}(Y^{\iota_1}e_C^a)Y_aY^{\iota_2}k_{CI}
-Y^\iota\bigg[e_I\widehat{n}+  
\widehat{k}_{ID} \gamma_{CDC}
-\widehat{k}_{CD} \gamma_{CID}+(\widetilde{k}_{ID}+\delta_{ID}H)\widehat{\gamma}_{CDC}\\
\notag&-(\widetilde{k}_{CD}+\delta_{CD}H)\widehat{\gamma}_{CID}
+(1+c_s^2)\big\{\widehat{\rho^{1-2r_s}}v_0v_I+\widetilde{\rho}^{1-2r_s}\widehat{v}_0v_I+\widetilde{\rho}^{1-2r_s}\widetilde{v}_0\widehat{v}_I\big\}\bigg].
\end{align}
The preceding terms can be viewed as error terms. 
\begin{lemma}\label{lem:frak.est}
Let $N\ge7$ and assume that the bootstrap assumptions \eqref{Boots} are valid. Then the following estimates holds: 
\begin{align}\label{frak.est}
&e^{Ht}\|\mathfrak{K}_{IJ}^\iota\|_{L^2(\Sigma_t)}
+e^{Ht}\|\mathfrak{G}_{IJB}^\iota\|_{L^2(\Sigma_t)}
+e^{Ht}\|\mathfrak{E}_{Ii}^\iota\|_{L^2(\Sigma_t)}
+e^{\frac{3}{2}Ht}\|\mathfrak{N}^\iota\|_{L^2(\Sigma_t)}
+e^{Ht}\|\mathfrak{M}^\iota_I\|_{L^2(\Sigma_t)}\\
\notag\leq&\,Ce^{-\frac{1}{2}Ht}\sum_{C=1}^3e^{\frac{3}{2}Ht}\|e_C\widehat{n}\|_{H^N(\Sigma_t)}+C e^{-\frac{1}{2}Ht}\sqrt{\mathcal{E}_{tot}(t)}
\end{align}
for every $|\iota|\leq N$ and $t\in[T,T_{Boot})$.
\end{lemma}
\begin{proof}
First, we observe that since $N\ge7$, each term in \eqref{mathfrak.K}-\eqref{mathfrak.M} has at most one factor with more than $N-4$ spatial derivatives. Hence, all terms in \eqref{mathfrak.K}-\eqref{mathfrak.M} can be controlled using the bootstrap assumptions \eqref{Boots} and the basic estimates \eqref{WN-3.est}, \eqref{rho.hat.boots}. We conclude \eqref{frak.est} by counting the powers of $e^{Ht}$ that correspond to each factor in every term, recalling the definition of the energies \eqref{geom.en}-\eqref{tot.en}. Indeed, we include here examples of terms containing geometric variables that decay the slowest: 
\begin{align}\label{frak.least.est}
\begin{split}
e^{Ht}\|\partial^\iota (n\widehat{\gamma}_{CID}\gamma_{DJC})\|_{L^2(\Sigma_t)}
\leq &\,Ce^{Ht}(e^{-Ht}\|\widehat{\gamma}\|_{H^N(\Sigma_t)})\leq Ce^{-Ht}\sqrt{\mathcal{E}_{tot}(t)},\\
e^{Ht}\|\partial^\iota ((e_B\widehat{n})k_{IJ})\|_{L^2(\Sigma_t)}
\leq &\,Ce^{Ht}\sum_{C=1}^3\|e_C\widehat{n}\|_{H^N(\Sigma_t)}\leq Ce^{-\frac{1}{2}Ht}\sum_{C=1}^3e^{\frac{3}{2}Ht}\|e_B\widehat{n}\|_{H^N(\Sigma_t)},\\
e^{\frac{3}{2}Ht}\|\partial^\iota (n\widehat{k}_{CD}\widehat{k}_{CD})\|_{L^2(\Sigma_t)}\leq&\,Ce^{\frac{1}{2}Ht}\|\widehat{k}\|_{H^N(\Sigma_t)}\leq Ce^{-\frac{1}{2}Ht}\sqrt{\mathcal{E}_{tot}(t)}.
\end{split}
\end{align}
The ones containing fluid variables are in fact more decaying decay quicker. We include the following examples of the slowest decaying terms: 
\begin{align}\label{frak.least.est2}
\notag
e^{\frac{3}{2}Ht}\|\partial^\iota\widehat{\rho}\|_{L^2(\Sigma_t)}\leq&\,C e^{\frac{3}{2}Ht}e^{-\frac{1}{r_s}Ht}\sqrt{\mathcal{E}_{tot}(t)}\leq Ce^{-\frac{3}{2}Ht}\sqrt{\mathcal{E}_{tot}(t)},\qquad (\text{$\frac{1}{r_s}>3$ for $\frac{1}{3}<c_s^2<\frac{3}{7}$})\\
e^{\frac{3}{2}Ht}\|\partial^\iota(\widetilde{\rho}^{1-2r_s}\widehat{v}_Cv_C)\|_{L^2(\Sigma_t)}\leq&\,C e^{\frac{3}{2}Ht}e^{-3Ht}\|\widehat{v}\|_{H^N(\Sigma_t)}\\
\notag\leq&\, Ce^{-\frac{3}{2}Ht}e^{2A_sHt-Ht}\sqrt{\mathcal{E}_{tot}(t)}\\
\notag\leq&\, Ce^{-\frac{3}{2}Ht}\sqrt{\mathcal{E}_{tot}(t)}.\qquad\qquad (\text{$2A_s<1$ for $\frac{1}{3}<c_s^2<\frac{3}{7}$})
\end{align}
One can tediously check that all other terms in \eqref{mathfrak.K}-\eqref{mathfrak.M} either decay faster or are at the same rate as above. 
\end{proof}
We will make use of the following energy identity.
\begin{lemma}\label{lem:geom.en.id}
The geometric variables $\widehat{k}_{IJ},\widehat{\gamma}_{IJB},\widehat{e}_I^i,\widehat{n}$ satisfy:
\begin{align}\label{geom.en.id}
\notag&\frac{1}{2}\partial_t(e^{2Ht}Y^\iota \widehat{k}_{IJ}Y^\iota \widehat{k}_{IJ})
+\frac{1}{4}\partial_t(e^{2Ht}Y^\iota \widehat{\gamma}_{IJB}Y^\iota \widehat{\gamma}_{IJB})
+\frac{1}{2}\partial_t(e^{2Ht}Y^\iota \widehat{e}_I^iY^\iota \widehat{e}_I^i)
+\frac{1}{2}\partial_t[e^{3Ht}(Y^\iota \widehat{n})^2]\\
\notag&+2He^{2Ht}Y^\iota \widehat{k}_{IJ}Y^\iota \widehat{k}_{IJ}
+\frac{1}{2}He^{2Ht}(Y^\iota\widehat{n})^2\\
=&-e^{2Ht}e_I(Y^\iota\widehat{k}_{IJ}Y^\iota e_J \widehat{n})+e^{2Ht}ne_C(Y^\iota\widehat{k}_{IJ}Y^\iota \widehat{\gamma}_{IJC})-e^{2Ht}ne_I(Y^\iota\widehat{k}_{IJ}Y^\iota\widehat{\gamma}_{CJC})\\
\notag&+e^{3Ht}e_C(Y^\iota\widehat{n}Y^\iota e_C \widehat{n})
-\sum_{C=1}^3e^{3Ht}(Y^\iota e_C \widehat{n})^2
+\sum_{\iota_1\cup\iota_2=\iota,\,|\iota_2|<|\iota|}e^{3Ht}(Y^{\iota_1}e_C^a)(Y_aY^{\iota_2}\widehat{n})Y^\iota e_C\widehat{n}\\
\notag&+e^{2Ht}\big\{\mathfrak{M}^\iota_J Y^\iota e_J\widehat{n}+n\mathfrak{M}^\iota_J Y^\iota\widehat{\gamma}_{CJC}+Y^\iota\widehat{k}_{IJ}\mathfrak{K}^\iota_{IJ}
+Y^\iota\widehat{\gamma}_{IJB}\mathfrak{G}_{IJB}^\iota
+Y^\iota\widehat{e}_I^i\mathfrak{E}^\iota_{Ii}\big\}
+e^{3Ht}(Y^\iota\widehat{n})\mathfrak{N}^\iota,
\end{align}
for every $|\iota|\leq N$.
\end{lemma}
\begin{proof}
It follows by multiplying the equations \eqref{k.hat.diff.eq}, \eqref{gamma.hat.diff.eq}, \eqref{eIi.hat.diff.eq}, \eqref{n.hat.diff.eq} with $e^{2Ht}\widehat{k}_{IJ},\frac{1}{2}e^{2Ht}\widehat{\gamma}_{IJB},e^{2Ht}\widehat{e}_I^i$, $e^{3Ht}\widehat{n}$ respectively, adding the resulting expressions, differentiating by parts, and plugging in the differentiated momentum constraint. 
\end{proof}
Now we can proceed to derive the main estimate in this subsection.
\begin{proposition}\label{prop:geom.est}
Let $N\ge7$ and assume that the bootstrap assumptions \eqref{Boots} are valid. Then the following integral estimate holds:
\begin{align}\label{geom.est}
\mathcal{E}_{geom}(t)+\int_T^t\sum_{C=1}^3e^{3H\tau}\|e_C\widehat{n}\|^2_{H^N(\Sigma_\tau)}d\tau\leq C\mathcal{E}_{geom}(T)+C\int^t_Te^{-\frac{1}{2}H\tau}\mathcal{E}_{tot}(\tau)d\tau,
\end{align}
for all $t\in[T,T_{Boot})$.  
\end{proposition}
\begin{proof}
Integrating the energy identity \eqref{geom.en.id} in $\Sigma_t$, summing in $|\iota|\leq N$, and integrating by parts using \eqref{IBP} gives
\begin{align}\label{geom.est2}
\notag&\frac{1}{2}\partial_t\big\{e^{2Ht}\|\widehat{k}\|^2_{H^N(\Sigma_t)}
+\frac{1}{2}e^{2Ht}\|\widehat{\gamma}\|^2_{H^N(\Sigma_t)}
+e^{2Ht}\|\widehat{e}\|^2_{H^N(\Sigma_t)}
+e^{3Ht}\|\widehat{n}\|^2_{H^N(\Sigma_t)}\big\}
+\sum_{C=1}^3e^{3Ht}\|e_C\widehat{n}\|^2_{H^N(\Sigma_t)}\\
\leq&\sum_{|\iota|\leq N}\int_{\Sigma_t}e^{2Ht}\big\{
(Y_ae_J^a)(Y^\iota\widehat{k}_{IJ}Y^\iota \widehat{n})-Y_a(ne_C^a)(Y^\iota\widehat{k}_{IJ}Y^\iota \widehat{\gamma}_{IJC})+Y_a(ne_I^a)(Y^\iota\widehat{k}_{IJ}Y^\iota\widehat{\gamma}_{CJC})
\big\}d\omega\\
\notag&-\sum_{|\iota|\leq N}\int_{\Sigma_t}e^{3Ht}(Y_ae_C^a)Y^\iota \widehat{n}Y^\iota e_C\widehat{n}
d\omega+\sum_{|\iota|\leq N}\int_{\Sigma_t}\bigg[\sum_{\iota_1\cup\iota_2=\iota,\,|\iota_2|<|\iota|}e^{3Ht}(Y^{\iota_1}e_C^a)(Y_aY^{\iota_2}\widehat{n})Y^\iota e_C\widehat{n}\\
\notag&+e^{2Ht}\big\{\mathfrak{M}^\iota_J Y^\iota e_J\widehat{n}+n\mathfrak{M}^\iota_J Y^\iota\widehat{\gamma}_{CJC}+Y^\iota\widehat{k}_{IJ}\mathfrak{K}^\iota_{IJ}
+Y^\iota\widehat{\gamma}_{IJB}\mathfrak{G}_{IJB}^\iota
+Y^\iota\widehat{e}_I^i\mathfrak{E}^\iota_{Ii}\big\}
+e^{3Ht}(Y^\iota\widehat{n})\mathfrak{N}^\iota \bigg]d\omega.
\end{align}
Next, we integrate in $[T,t]$, make use of the bootstrap assumptions \eqref{Boots}, the basic estimates \eqref{WN-3.est}, Lemma \ref{lem:frak.est} and Young's inequality to obtain the energy inequality:
\begin{align}\label{geom.est3}
&\mathcal{E}_{geom}(t)+\int_T^t\sum_{C=1}^3e^{3H\tau}\|e_C\widehat{n}\|^2_{H^N(\Sigma_\tau)}d\tau\\
\notag\leq&\,C\mathcal{E}_{geom}(T)
+\int^t_T\frac{C}{\eta}e^{-\frac{1}{2}H\tau}\mathcal{E}_{tot}(\tau)d\tau 
+\int_T^t\eta\sum_{C=1}^3e^{3H\tau}\|e_C\widehat{n}\|^2_{H^N(\Sigma_\tau)}d\tau,
\end{align}
for some $\eta<1$ of our choice, resulting from the application of Young's inequality to the terms having a factor $Y^\iota e_C\widehat{n}$ with $|\iota|=N$, as well as the ones coming from \eqref{frak.est}. Absorbing the last term in the corresponding one in the LHS and readjusting the generic constant $C$, we conclude the desired estimate. 
\end{proof}
\subsection{Energy estimates for the fluid variables}\label{S:top.est}

First, we derive lower order energy estimates for the fluid variables $\widehat{v}_I,\widehat{\rho^{2r_s}}$ by treating the evolution equations \eqref{vI.hat.eq}, \eqref{rho.hat.eq} as transport equations. For the top order energy estimates,  we use the additional equation satisfied by $v_IY^\iota \widehat{v}_I$, see \eqref{vI.hat.top.eq} below. Once we have obtained $H^N(\Sigma_t)$ estimates for the fluid variables $\widehat{v}_I,\widehat{\rho^{2r_s}}$, we control the $H^N(\Sigma_t)$ norm of $\widehat{v}_0$ using the identity:
\begin{align}\label{v0.hat.id.diff}
\widehat{v}_0=\frac{1}{v_0+\widetilde{v}_0}\big\{(\widetilde{v}_C+v_C)\widehat{v}_C
+\widehat{\rho^{2r_s}}\big\}.
\end{align}
To begin with, we commute the equations \eqref{vI.hat.eq}, \eqref{rho.hat.eq} with $Y^\iota$, $|\iota|\leq N$, using \eqref{comm.eI}:
\begin{align}
\label{vI.hat.eq.diff}\partial_tY^\iota\widehat{v}_I+H Y^\iota\widehat{v}_I
=&\,\frac{1}{2}\frac{n}{v_0}e_IY^\iota(\widehat{\rho^{2r_s}})+n\frac{v_C}{v_0}e_CY^\iota\widehat{v}_I+\mathfrak{V}_I^\iota,\\
\label{rho.hat.eq.diff}f_{\frac{1-2r_s}{2r_s}}\partial_tY^\iota\widehat{\rho^{2r_s}}+2 HY^\iota\widehat{\rho^{2r_s}}
=&-n\frac{\rho^{2r_s}}{v_0^2}\frac{v_Iv_D}{v_0}e_DY^\iota\widehat{v}_I
+n\frac{\rho^{2r_s}}{v_0}e_CY^\iota\widehat{v}_C\\
\notag&+n\big(\frac{1-2r_s}{2r_s}-\frac{1}{2}\frac{\rho^{2r_s}}{v_0^2}\big)\frac{v_C}{v_0}e_CY^\iota (\widehat{\rho^{2r_s}})+\mathfrak{P}^\iota,
\end{align}
where 
\begin{align}
\label{frak.VI}\mathfrak{V}_I^\iota=&
\sum_{\iota_1\cup\iota_2=\iota,\,|\iota_2|<|\iota|}\big\{\frac{1}{2}Y^{\iota_1}(\frac{n}{v_0}e_I^a)Y_aY^{\iota_2}(\widehat{\rho^{2r_s}})+Y^{\iota_1}(n\frac{v_C}{v_0}e_C^a)Y_aY^{\iota_2}\widehat{v}_I\big\}\\
\notag&+Y^\iota\big\{
(e_I\widehat{n})v_0+n\gamma_{CDI}\frac{v_D}{v_0}\widehat{v}_C
+n\gamma_{1DI}\frac{\widetilde{v}_1}{v_0}\widehat{v}_D
+\frac{n}{v_0}\widetilde{v}_1^2\widehat{\gamma}_{11I}\\
\notag&+\widehat{n}k_{CI}v_C+\widehat{k}_{CI}v_C+(\widetilde{k}_{CI}+\delta_{CI}H)\widehat{v}_C\big\},\\
\label{f}f_{\frac{1-2r_s}{2r_s}}=&\,\frac{1-2r_s}{2r_s}+\frac{1}{2}\frac{\rho^{2r_s}}{v_0^2},\\
\label{frak.P}\mathfrak{P}^\iota=&
\sum_{\iota_1\cup\iota_2=\iota,\,|\iota_2|<|\iota|}f_{\frac{1-2r_s}{2r_s}}\bigg[
Y^{\iota_1}(f^{-1}_{\frac{1-2r_s}{2r_s}}n\frac{\rho^{2r_s}}{v_0}e_C^a)Y_aY^{\iota_2}\widehat{v}_C-Y^{\iota_1}(f^{-1}_{\frac{1-2r_s}{2r_s}}n\frac{\rho^{2r_s}}{v_0^2}\frac{v_Iv_D}{v_0}e_D^a)Y_aY^{\iota_2}\widehat{v}_I\\
\notag&+Y^{\iota_1}\big\{f^{-1}_{\frac{1-2r_s}{2r_s}}n\big(\frac{1-2r_s}{2r_s}-\frac{1}{2}\frac{\rho^{2r_s}}{v_0^2}\big)\frac{v_C}{v_0}e_C^a\big\}Y_aY^{\iota_2} (\widehat{\rho^{2r_s}})
-2Y^{\iota_1}(f^{-1}_{\frac{1-2r_s}{2r_s}}) HY^{\iota_2}\widehat{\rho^{2r_s}}\bigg]\\
\notag&+\sum_{\iota_1\cup\iota_2=\iota}f_{\frac{1-2r_s}{2r_s}}Y^{\iota_1}(f^{-1}_{\frac{1-2r_s}{2r_s}})Y^{\iota_2}\bigg[
\frac{\rho^{2r_s}}{v_0}\big\{(e_C\widehat{n})v_C
+n\widehat{\gamma}_{CDC}v_D\big\}\\
\notag&-\frac{\rho^{2r_s}}{v_0^2}\frac{v_I}{v_0}
\big\{(e_I\widehat{n})v_0^2+n\gamma_{CDI}\widehat{v}_Cv_D
+n\gamma_{1DI}\widetilde{v}_1\widehat{v}_D+n\widehat{\gamma}_{11I}\widetilde{v}_1^2\big\}\\
\notag&-\frac{1}{2}\frac{\widehat{\rho^{2r_s}}}{v_0^2}\partial_t(\widetilde{\rho}^{2r_s})
+\frac{1}{2}\frac{\widetilde{\rho}^{2r_s}}{v_0^2\widetilde{v}^2_0}\partial_t(\widetilde{\rho}^{2r_s})(v_0+\widetilde{v}_0)\widehat{v}_0
-\widehat{n}\big(\frac{v_Cv_I}{v_0^2}k_{CI}-\mathrm{tr}k\big)\rho^{2r_s}
-\rho^{2r_s}\widehat{n}\\
\notag&+\mathcal{O}(e^{-2Ht})\widehat{\rho^{2r_s}}
-\frac{v_Cv_I}{v_0^2}\widehat{k}_{CI}\rho^{2r_s}
-H\frac{\widehat{\rho^{2r_s}}}{v_0^2}(\rho^{2r_s}+\widetilde{\rho}^{2r_s})
+H\frac{\widetilde{\rho}^{4r_s}}{v_0^2\widetilde{v}_0^2}(v_0+\widetilde{v}_0)\widehat{v}_0\\
\notag&-(\widetilde{k}_{CI}+\delta_{CI}H)
\big(\frac{v_Cv_I}{v_0^2}\widehat{\rho^{2r_s}}+\frac{\widehat{v}_Cv_I}{v_0^2}\widetilde{\rho}^{2r_s}+\frac{\widetilde{v}_C\widehat{v}_I}{v_0^2}\widetilde{\rho}^{2r_s}+\frac{\widetilde{v}_C\widetilde{v}_I}{v_0^2\widetilde{v}_0^2}(v_0+\widetilde{v}_0)\widehat{v}_0\widetilde{\rho}^{2r_s}\big)\bigg].
\end{align}
The equations \eqref{vI.hom.eq}, \eqref{vI.hat.eq}, \eqref{vI.hat.eq.diff} also imply:
\begin{align}\label{vI.hat.top.eq}
\partial_t(v_IY^\iota\widehat{v}_I)+2H v_IY^\iota\widehat{v}_I
=&\,\frac{1}{2}n\frac{v_I}{v_0}e_IY^\iota(\widehat{\rho^{2r_s}})+n\frac{v_Cv_I}{v_0}e_CY^\iota\widehat{v}_I+\mathfrak{S}^\iota
\end{align}
where
\begin{align}
\label{frak.S}\mathfrak{S}^\iota=&\,v_I\mathfrak{V}_I^\iota
+\big\{\frac{1}{2}\frac{n}{v_0}e_I(\widehat{\rho^{2r_s}})+n\frac{v_C}{v_0}e_C\widehat{v}_I
+(e_I\widehat{n})v_0+n\gamma_{CDI}\frac{v_D}{v_0}\widehat{v}_C
+n\gamma_{1DI}\frac{\widetilde{v}_1}{v_0}\widehat{v}_D
+\frac{n}{v_0}\widetilde{v}_1^2\widehat{\gamma}_{11I}\\
\notag&+\widehat{n}k_{CI}v_C+\widehat{k}_{CI}v_C+(\widetilde{k}_{CI}+\delta_{CI}H)v_C\big\}Y^\iota\widehat{v}_I.
\end{align}
At top order, $|\iota|=N$, the equations \eqref{vI.hat.eq.diff}, \eqref{rho.hat.eq.diff}, \eqref{vI.hat.top.eq} combined yield the following energy identity.
\begin{lemma}
The top order fluid variables satisfy:
\begin{align}\label{top.fluid.id}
\notag&\frac{1}{2}\partial_t\big\{e^{4Ht}\rho^{2r_s}Y^\iota \widehat{v}_I Y^\iota \widehat{v}_I
+\frac{1}{2}f_{\frac{1-2r_s}{2r_s}}e^{4Ht}(Y^\iota\widehat{\rho^{2r_s}})^2
-e^{4Ht}\frac{\rho^{2r_s}}{v_0^2}v_Iv_EY^\iota\widehat{v}_IY^\iota\widehat{v}_E \big\}\\
\notag&+\big[He^{4Ht}-\frac{1}{4}\partial_t(f_{\frac{1-2r_s}{2r_s}}e^{4Ht})\big](Y^\iota\widehat{\rho^{2r_s}})^2
-\frac{1}{2}e^{2Ht}\partial_t(e^{2Ht}\rho^{2r_s})
Y^\iota \widehat{v}_I Y^\iota \widehat{v}_I
+\frac{1}{2}\partial_t(\frac{\rho^{2r_s}}{v_0^2})e^{4Ht}
v_Iv_EY^\iota \widehat{v}_I Y^\iota \widehat{v}_E
\\
=&\,\frac{1}{2}ne^{4Ht}\frac{\rho^{2r_s}}{v_0}e_I[Y^\iota(\widehat{\rho^{2r_s}})Y^\iota\widehat{v}_I]
-\frac{1}{2}ne^{4Ht}\frac{\rho^{2r_s}}{v_0^2}\frac{v_Iv_D}{v_0}e_D(Y^\iota\widehat{v}_IY^\iota\widehat{\rho^{2r_s}})
+\frac{1}{2}ne^{4Ht}\rho^{2r_s}\frac{v_C}{v_0}e_C(Y^\iota\widehat{v}_IY^\iota\widehat{v}_I)\\
\notag&+\frac{1}{4}ne^{4Ht}\big(\frac{1-2r_s}{2r_s}-\frac{1}{2}\frac{\rho^{2r_s}}{v_0^2}\big)\frac{v_C}{v_0}e_C[(Y^\iota \widehat{\rho^{2r_s}})^2]
-\frac{1}{2}ne^{4Ht}\rho^{2r_s}\frac{v_C}{v_0}\frac{v_Iv_E}{v_0^2}e_C(Y^\iota\widehat{v}_IY^\iota\widehat{v}_E)\\
\notag&+e^{4Ht}\rho^{2r_s}Y^\iota\widehat{v}_I\mathfrak{V}_I^\iota
+\frac{1}{2}e^{4Ht}(Y^\iota\widehat{\rho^{2r_s}})\mathfrak{P}^\iota
-e^{4Ht}\frac{\rho^{2r_s}}{v_0^2}v_IY^\iota\widehat{v}_I\mathfrak{S}^\iota
\end{align}
for every $|\iota|=N$.
\end{lemma}
\begin{proof}
The energy identity \eqref{top.fluid.id} is obtained by considering the following algebraic combination of equations: 
\begin{align*}
e^{4Ht}\rho^{2r_s}Y^\iota\widehat{v}_I\times\eqref{vI.hat.eq.diff}+\frac{1}{2}e^{4Ht}Y^\iota\widehat{\rho^{2r_s}}\times\eqref{rho.hat.eq.diff}-e^{4Ht}\frac{\rho^{2r_s}}{v_0^2}v_EY^\iota\widehat{v}_E\times\eqref{vI.hat.top.eq}
\end{align*}
and applying Leibnitz rule.
\end{proof}
Next, we control the error terms \eqref{frak.VI}, \eqref{frak.P}, \eqref{frak.S}.
\begin{lemma}\label{lem:frak.fluid.est}
Let $N\ge7$ and assume the bootstrap assumptions \eqref{Boots} are valid.
Then the expressions $\mathfrak{V}_I,\mathfrak{P}^\iota,\mathfrak{S}^\iota$ satisfy the lower order estimate: 
\begin{align}\label{error.fluid.less.top.est}
\sum_{|\iota|\leq N-1}\sum_{I=1}^3e^{Ht}\|\mathfrak{V}^\iota_I\|_{L^2(\Sigma_t)}+\sum_{|\iota|\leq N-2}e^{\frac{4r_s}{1-2r_s}Ht}\|\mathfrak{P}^\iota\|_{L^2(\Sigma_t)}\leq&\, C(e^{-Ht}+e^{2Ht-\frac{4r_s}{1-2r_s}Ht})\sqrt{\mathcal{E}_{tot}(t)}
\end{align}
where we recall that $\frac{4r_s}{1-2r_s}>2$, $r_s>\frac{1}{4}$,
and the top order estimates:
\begin{align}\label{error.fluid.top.est}
\sum_{|\iota|=N}\bigg[\sum_{I=1}^3e^{2Ht}\rho^{r_s}\|\mathfrak{V}^\iota_I\|_{L^2(\Sigma_t)}
+e^{2Ht}\frac{\rho^{r_s}}{v_0}\|\mathfrak{S}^\iota\|_{L^2(\Sigma_t)}\bigg]
\leq &\,Ce^{-Ht}e^{A_sHt}\sqrt{\mathcal{E}_{tot}(t)}+C\sum_{I=1}^3\|e_I\widehat{n}\|_{H^N(\Sigma_t)},
\end{align}
\begin{align}\label{error.fluid.top.est2}
\sum_{|\iota|=N-1,N}e^{2Ht}\|\mathfrak{P}^\iota\|_{L^2(\Sigma_t)}\leq&\,C(e^{-Ht}+e^{2Ht-\frac{4r_s}{1-2r_s}Ht})\sqrt{\mathcal{E}_{tot}(t)}+C\sum_{I=1}^3\|e_I\widehat{n}\|_{H^N(\Sigma_t)},
\end{align}
for all $t\in[T,T_{Boot})$.
\end{lemma}
\begin{remark}\label{rem:top.fluid.error}
The exponent $A_sHt$ having a bad sign in the RHS of \eqref{error.fluid.top.est} comes from the terms in $\mathfrak{V}_I^\iota,\mathfrak{S}^\iota$ that do not contain a $\rho^{2r_s}$ factor. This leads to the weak top order estimate \eqref{fluid.top.est}, which only allows us to close our bootstrap argument in the range $c_s^2\in(\frac{1}{3},\frac{3}{7})$, see Remark \ref{rem:top.fluid.est}. On the other hand, these terms are not present when looking at the relativistic Euler equations in a fixed FLRW background, which allows to treat the whole range $c_s^2\in(\frac{1}{3},1)$, see Appendix \ref{app:Euler}.
\end{remark}
\begin{proof}
First, observe that since $N\ge7$, there is at most one factor in every term of $\mathfrak{V}_I^\iota,\mathfrak{P}^\iota,\mathfrak{S}^\iota$ that contains more than $N-4$ spatial derivatives, for every $|\iota|\leq N$. Hence, we control all error terms by counting the powers of $e^{Ht}$ corresponding to the lower order (optimal) behavior of each factor, leaving one factor with maximum number of spatial derivatives, using the bootstrap assumptions \eqref{Boots}, the basic estimates \eqref{WN-3.est}, \eqref{rho.hat.boots}, and Lemma \ref{lem:Linfty.est} (to handle the inverse powers of $v_0$).

{\it 1. Lower order error estimate \eqref{error.fluid.less.top.est}.} To control $\mathfrak{V}_I^\iota$, $|\iota|\leq N-1$, we recall the definitions \eqref{geom.en}, \eqref{fluid.en}, \eqref{tot.en} of the energies $\mathcal{E}_{geom}(t),\mathcal{E}_{fluid.low}(t),\mathcal{E}_{tot}(t)$:
\begin{align}\label{frak.fluid.est}
\notag&e^{Ht}\bigg\|\sum_{\iota_1\cup\iota_2=\iota,\,|\iota_2|<|\iota|}\big\{\frac{1}{2}Y^{\iota_1}(\frac{n}{v_0}e_I^a)Y_aY^{\iota_2}(\widehat{\rho^{2r_s}})+Y^{\iota_1}(n\frac{v_C}{v_0}e_C^a)Y_aY^{\iota_2}\widehat{v}_I\big\}\bigg\|_{L^2(\Sigma_t)}\\
\leq&\,Ce^{Ht}\big(\|\widehat{\rho^{2r_s}}\|_{H^{N-1}(\Sigma_t)}+e^{-Ht}\|\widehat{v}\|_{H^{N-1}(\Sigma_t)}\big)\\
\notag\leq&\,C(e^{Ht-\frac{4r_s}{1-2r_s}}+e^{-Ht})\sqrt{\mathcal{E}_{tot}(t)}\leq Ce^{-Ht}\sqrt{\mathcal{E}_{tot}(t)},\qquad\qquad(\frac{4r_s}{1-2r_s}>2,\;r_s>\frac{1}{4})
\end{align}
\begin{align}\label{frak.fluid.est2}
\notag&e^{Ht}\big\|Y^\iota\big\{
(e_I\widehat{n})v_0+n\gamma_{CDI}\frac{v_D}{v_0}\widehat{v}_C
+n\gamma_{1DI}\frac{\widetilde{v}_1}{v_0}\widehat{v}_D
+\frac{n}{v_0}\widetilde{v}_1^2\widehat{\gamma}_{11I}\\
&+\widehat{n}k_{CI}v_C+\widehat{k}_{CI}v_C+(\widetilde{k}_{CI}+\delta_{CI}H)\widehat{v}_C\big\}\big\|_{L^2(\Sigma_t)}\\
\notag\leq&\, Ce^{-Ht}\|\widehat{n}\|_{H^N(\Sigma_t)}+C\big(\|\widehat{v}\|_{H^{N-1}(\Sigma_t)}
+\|\widehat{\gamma}\|_{H^{N-1}(\Sigma_t)}
+\|\widehat{n}\|_{H^{N-1}(\Sigma_t)}
+\|\widehat{k}\|_{H^{N-1}(\Sigma_t)}\big)\\
\notag\leq&\,Ce^{-Ht}\sqrt{\mathcal{E}_{tot}(t)}.
\end{align}
This completes the estimate \eqref{error.fluid.less.top.est} for $\mathfrak{V}_I^\iota$. We proceed to estimate the $L^2(\Sigma_t)$ norm of $\mathfrak{P}^\iota$, $|\iota|\leq N-2$, in a similar manner:
\begin{align}\label{frak.fluid.est3}
\notag &e^{\frac{4r_s}{1-2r_s}Ht}\bigg\|\sum_{\iota_1\cup\iota_2=\iota,\,|\iota_2|<|\iota|}f_{\frac{1-2r_s}{2r_s}}\bigg[
Y^{\iota_1}(f^{-1}_{\frac{1-2r_s}{2r_s}}n\frac{\rho^{2r_s}}{v_0}e_C^a)Y_aY^{\iota_2}\widehat{v}_C-Y^{\iota_1}(f^{-1}_{\frac{1-2r_s}{2r_s}}n\frac{\rho^{2r_s}}{v_0^2}\frac{v_Iv_D}{v_0}e_D^a)Y_aY^{\iota_2}\widehat{v}_I\\
&+Y^{\iota_1}\big\{f^{-1}_{\frac{1-2r_s}{2r_s}}n\big(\frac{1-2r_s}{2r_s}-\frac{1}{2}\frac{\rho^{2r_s}}{v_0^2}\big)\frac{v_C}{v_0}e_C^a\big\}Y_aY^{\iota_2} (\widehat{\rho^{2r_s}})
-2Y^{\iota_1}(f^{-1}_{\frac{1-2r_s}{2r_s}}) HY^{\iota_2}\widehat{\rho^{2r_s}}
\bigg]\bigg\|_{L^2(\Sigma_t)}\\
\notag\leq&\,C\|\widehat{v}\|_{H^{N-2}(\Sigma_t)}+C(e^{-Ht}+e^{2Ht-\frac{4r_s}{1-2r_s}Ht})e^{\frac{4r_s}{1-2r_s}Ht}\|\widehat{\rho^{2r_s}}\|_{H^{N-2}(\Sigma_t)}\qquad\qquad(|\iota_1|\ge1)\\
\notag\leq&\, C(e^{-Ht}+e^{2Ht-\frac{4r_s}{1-2r_s}Ht})\sqrt{\mathcal{E}_{tot}(t)},
\end{align}
\begin{align}\label{frak.fluid.est4}
\notag&e^{\frac{4r_s}{1-2r_s}Ht}\bigg\|\sum_{\iota_1\cup\iota_2=\iota}f_{\frac{1-2r_s}{2r_s}}Y^{\iota_1}(f^{-1}_{\frac{1-2r_s}{2r_s}})Y^{\iota_2}\bigg[
\frac{\rho^{2r_s}}{v_0}\big\{(e_C\widehat{n})v_C
+n\widehat{\gamma}_{CDC}v_D\big\}\\
\notag&-\frac{\rho^{2r_s}}{v_0^2}\frac{v_I}{v_0}
\big\{(e_I\widehat{n})v_0^2+n\gamma_{CDI}\widehat{v}_Cv_D
+n\gamma_{1DI}\widetilde{v}_1\widehat{v}_D+n\widehat{\gamma}_{11I}\widetilde{v}_1^2\big\}\\
&-\frac{1}{2}\frac{\widehat{\rho^{2r_s}}}{v_0^2}\partial_t(\widetilde{\rho}^{2r_s})
+\frac{1}{2}\frac{\widetilde{\rho}^{2r_s}}{v_0^2\widetilde{v}^2_0}\partial_t(\widetilde{\rho}^{2r_s})(v_0+\widetilde{v}_0)\widehat{v}_0
-\widehat{n}\big(\frac{v_Cv_I}{v_0^2}k_{CI}-\mathrm{tr}k\big)\rho^{2r_s}
-\rho^{2r_s}\widehat{n}\\
\notag &+\mathcal{O}(e^{-2Ht})\widehat{\rho^{2r_s}}
-\frac{v_Cv_I}{v_0^2}\widehat{k}_{CI}\rho^{2r_s}
-H\frac{\widehat{\rho^{2r_s}}}{v_0^2}(\rho^{2r_s}+\widetilde{\rho}^{2r_s})
+H\frac{\widetilde{\rho}^{4r_s}}{v_0^2\widetilde{v}_0^2}(v_0+\widetilde{v}_0)\widehat{v}_0\\
\notag&-(\widetilde{k}_{CI}+\delta_{CI}H)
\big(\frac{v_Cv_I}{v_0^2}\widehat{\rho^{2r_s}}+\frac{\widehat{v}_Cv_I}{v_0^2}\widetilde{\rho}^{2r_s}+\frac{\widetilde{v}_C\widehat{v}_I}{v_0^2}\widetilde{\rho}^{2r_s}+\frac{\widetilde{v}_C\widetilde{v}_I}{v_0^2\widetilde{v}_0^2}(v_0+\widetilde{v}_0)\widehat{v}_0\widetilde{\rho}^{2r_s}\big)\bigg]\bigg\|_{L^2(\Sigma_t)}\\
\notag\leq&\,Ce^{-Ht}\|\widehat{n}\|_{H^{N-1}(\Sigma_t)}
+C\big\{\|\widehat{\gamma}\|_{H^{N-2}(\Sigma_t)}
+\|\widehat{v}\|_{H^{N-2}(\Sigma_t)}
+\|\widehat{k}\|_{H^{N-2}(\Sigma_t)}+\|\widehat{n}\|_{H^{N-2}(\Sigma_t)}\big\}\\
\notag&+Ce^{2Ht-\frac{4r_s}{1-2r_s}Ht}e^{Ht}\|\widehat{v}\|_{H^{N-2}(\Sigma_t)}
+e^{2Ht}\|\widehat{\rho^{2r_s}}\|_{H^{N-2}(\Sigma_t)}\\
\notag\leq&\,C(e^{-Ht}+e^{2Ht-\frac{4r_s}{1-2r_s}Ht})\sqrt{\mathcal{E}_{tot}(t)}.
\end{align}
The above estimates give the desired control of $\mathfrak{P}^\iota$, $|\iota|\leq N-2$, in $L^2(\Sigma_t)$, which completes the proof of \eqref{error.fluid.less.top.est}.  

{\it 2. Top order error estimates \eqref{error.fluid.top.est}-\eqref{error.fluid.top.est2}.} More careful power counting is needed due to the discrepancy in the $e^{Ht}$-weights in the energy $\mathcal{E}_{fluid.top}(t)$, compared to $\mathcal{E}_{fluid.low}(t)$ in \eqref{fluid.en}. For the convenience of the reader, we note here the following power relations:
\begin{align}\label{rho.As.rel}
\notag
e^{2Ht}\rho^{r_s}\leq Ce^{Ht}e^{Ht-\frac{2r_s}{1-2r_s}}=&\,Ce^{Ht}e^{-A_sHt},\qquad A_s=\frac{3c_s^2-1}{1-c_s^2}=\frac{2r_s}{1-r_s}-1,\qquad r_s=\frac{c_s^2}{1+c_s^2},\\ e^{2Ht}\rho^{r_s}\|\widehat{v}\|_{\dot{H}^N(\Sigma_t)}\leq&\, Ce^{2Ht-\frac{2r_s}{1-2r_s}}\|\widehat{v}\|_{\dot{H}^N(\Sigma_t)}=
Ce^{Ht-A_sHt}\|\widehat{v}\|_{\dot{H}^N(\Sigma_t)}=Ce^{A_sHt}\sqrt{\mathcal{E}_{tot}(t)},\\
e^{2Ht}\|\widehat{\rho^{2r_s}}\|_{\dot{H}^N(\Sigma_t)}\leq &\,\sqrt{\mathcal{E}_{tot}(t)}.
\notag
\end{align}
We begin first with the estimates for $\mathfrak{V}_I^\iota$, $|\iota|=N$:
\begin{align}\label{frak.fluid.est5}
\notag&e^{2Ht}\rho^{r_s}\bigg\|\sum_{\iota_1\cup\iota_2=\iota,\,|\iota_2|<|\iota|}\big\{\frac{1}{2}Y^{\iota_1}(\frac{n}{v_0}e_I^a)Y_aY^{\iota_2}(\widehat{\rho^{2r_s}})+Y^{\iota_1}(n\frac{v_C}{v_0}e_C^a)Y_aY^{\iota_2}\widehat{v}_I\big\}\bigg\|_{L^2(\Sigma_t)}\\
\leq&\,Ce^{2Ht}\rho^{r_s}\big(\|\widehat{\rho^{2r_s}}\|_{H^N(\Sigma_t)}+e^{-Ht}\|\widehat{v}\|_{H^N(\Sigma_t)}\big)\\
\notag\leq&\,C(e^{-\frac{2r_s}{1-2r_s}Ht}+e^{-Ht}e^{A_sHt})\sqrt{\mathcal{E}_{tot}(t)}
\leq Ce^{-Ht}e^{A_sHt}\sqrt{\mathcal{E}_{tot}(t)},\qquad\qquad(\frac{2r_s}{1-2r_s}>1)
\end{align}
\begin{align}\label{frak.fluid.est6}
\notag&e^{2Ht}\rho^{r_s}\big\|Y^\iota\big\{
(e_I\widehat{n})v_0+n\gamma_{CDI}\frac{v_D}{v_0}\widehat{v}_C
+n\gamma_{1DI}\frac{\widetilde{v}_1}{v_0}\widehat{v}_D
+\frac{n}{v_0}\widetilde{v}_1^2\widehat{\gamma}_{11I}\\
\notag&+\widehat{n}k_{CI}v_C+\widehat{k}_{CI}v_C+(\widetilde{k}_{CI}+\delta_{CI}H)\widehat{v}_C\big\}\big\|_{L^2(\Sigma_t)}\\
\leq&\, Ce^{Ht-\frac{2r_s}{1-2r_s}Ht}\sum_{I=1}^3\|e_I\widehat{n}\|_{H^N(\Sigma_t)}+Ce^{2Ht}\rho^{r_s}e^{-Ht}\|\widehat{v}\|_{H^N(\Sigma_t)}\\
\notag&+C\big(\|\widehat{\gamma}\|_{H^N(\Sigma_t)}
+\|\widehat{n}\|_{H^N(\Sigma_t)}
+\|\widehat{k}\|_{H^N(\Sigma_t)}\big)\\
\notag\leq&\,C\sum_{I=1}^3\|e_I\widehat{n}\|_{H^N(\Sigma_t)}+Ce^{-Ht}e^{A_sHt}\sqrt{\mathcal{E}_{tot}(t)},
\end{align}
and proceed to estimate $\mathfrak{S}^\iota$, $|\iota|=N$, having just derived \eqref{error.fluid.top.est} for $\mathfrak{V}_I^\iota$: 
\begin{align}\label{frak.fluid.est7}
\notag&e^{2Ht}\frac{\rho^{r_s}}{v_0}\big\|v_I\mathfrak{V}_I^\iota
+\big\{\frac{1}{2}\frac{n}{v_0}e_I(\widehat{\rho^{2r_s}})+n\frac{v_C}{v_0}e_C\widehat{v}_I
+(e_I\widehat{n})v_0+n\gamma_{CDI}\frac{v_D}{v_0}\widehat{v}_C
+n\gamma_{1DI}\frac{\widetilde{v}_1}{v_0}\widehat{v}_D
+\frac{n}{v_0}\widetilde{v}_1^2\widehat{\gamma}_{11I}\\
&+\widehat{n}k_{CI}v_C+\widehat{k}_{CI}v_C+(\widetilde{k}_{CI}+\delta_{CI}H)v_C\big\}Y^\iota\widehat{v}_I\big\|_{L^2(\Sigma_t)}\\
\notag\leq&\,C\sum_{I=1}^3\|e_I\widehat{n}\|_{H^N(\Sigma_t)}+Ce^{-Ht}e^{A_sHt}\sqrt{\mathcal{E}_{tot}(t)}+Ce^{2Ht}\frac{\rho^{r_s}}{v_0}e^{-2Ht}\|\widehat{v}\|_{\dot{H}^N(\Sigma_t)}\\
\notag\leq&\,C\sum_{I=1}^3\|e_I\widehat{n}\|_{H^N(\Sigma_t)}+Ce^{-Ht}e^{A_sHt}\sqrt{\mathcal{E}_{tot}(t)}.
\end{align}
For the error terms in $\mathfrak{P}^\iota$, $|\iota|=N-1,N$, we have:
\begin{align}\label{frak.fluid.est8}
\notag &e^{2Ht}\bigg\|\sum_{\iota_1\cup\iota_2=\iota,\,|\iota_2|<|\iota|}f_{\frac{1-2r_s}{2r_s}}\bigg[
Y^{\iota_1}(f^{-1}_{\frac{1-2r_s}{2r_s}}n\frac{\rho^{2r_s}}{v_0}e_C^a)Y_aY^{\iota_2}\widehat{v}_C-Y^{\iota_1}(f^{-1}_{\frac{1-2r_s}{2r_s}}n\frac{\rho^{2r_s}}{v_0^2}\frac{v_Iv_D}{v_0}e_D^a)Y_aY^{\iota_2}\widehat{v}_I\\
&+Y^{\iota_1}\big\{f^{-1}_{\frac{1-2r_s}{2r_s}}n\big(\frac{1-2r_s}{2r_s}-\frac{1}{2}\frac{\rho^{2r_s}}{v_0^2}\big)\frac{v_C}{v_0}e_C^a\big\}Y_aY^{\iota_2} (\widehat{\rho^{2r_s}})
-2Y^{\iota_1}(f^{-1}_{\frac{1-2r_s}{2r_s}}) HY^{\iota_2}\widehat{\rho^{2r_s}}
\bigg]\bigg\|_{L^2(\Sigma_t)}\\
\notag\leq&\,Ce^{2Ht}e^{-\frac{4r_s}{1-2r_s}Ht}\|\widehat{v}\|_{H^N(\Sigma_t)}
+Ce^{2Ht}(e^{-Ht}+e^{2Ht-\frac{4r_s}{1-2r_s}Ht})\|\widehat{\rho^{2r_s}}\|_{H^N(\Sigma_t)}\qquad\qquad(|\iota_1|\ge1)\\
\notag=&\,Ce^{-Ht}e^{-2A_sHt}e^{Ht}\|\widehat{v}\|_{H^N(\Sigma_t)}
+C(e^{-Ht}+e^{2Ht-\frac{4r_s}{1-2r_s}Ht})e^{2Ht}\|\widehat{\rho^{2r_s}}\|_{H^N(\Sigma_t)}\qquad(A_s=\frac{2r_s}{1-2r_s}-1)\\
\notag\leq&\, C(e^{-Ht}+e^{2Ht-\frac{4r_s}{1-2r_s}Ht})\sqrt{\mathcal{E}_{tot}(t)},
\end{align}
\begin{align}\label{frak.fluid.est9}
\notag&e^{2Ht}\bigg\|\sum_{\iota_1\cup\iota_2=\iota}f_{\frac{1-2r_s}{2r_s}}Y^{\iota_1}(f^{-1}_{\frac{1-2r_s}{2r_s}})Y^{\iota_2}\bigg[
\frac{\rho^{2r_s}}{v_0}\big\{(e_C\widehat{n})v_C
+n\widehat{\gamma}_{CDC}v_D\big\}\\
\notag&-\frac{\rho^{2r_s}}{v_0^2}\frac{v_I}{v_0}
\big\{(e_I\widehat{n})v_0^2+n\gamma_{CDI}\widehat{v}_Cv_D
+n\gamma_{1DI}\widetilde{v}_1\widehat{v}_D+n\widehat{\gamma}_{11I}\widetilde{v}_1^2\big\}\\
&-\frac{1}{2}\frac{\widehat{\rho^{2r_s}}}{v_0^2}\partial_t(\widetilde{\rho}^{2r_s})
+\frac{1}{2}\frac{\widetilde{\rho}^{2r_s}}{v_0^2\widetilde{v}^2_0}\partial_t(\widetilde{\rho}^{2r_s})(v_0+\widetilde{v}_0)\widehat{v}_0
-\widehat{n}\big(\frac{v_Cv_I}{v_0^2}k_{CI}-\mathrm{tr}k\big)\rho^{2r_s}
-\rho^{2r_s}\widehat{n}\\
\notag &+\mathcal{O}(e^{-2Ht})\widehat{\rho^{2r_s}}
-\frac{v_Cv_I}{v_0^2}\widehat{k}_{CI}\rho^{2r_s}
-H\frac{\widehat{\rho^{2r_s}}}{v_0^2}(\rho^{2r_s}+\widetilde{\rho}^{2r_s})
+H\frac{\widetilde{\rho}^{4r_s}}{v_0^2\widetilde{v}_0^2}(v_0+\widetilde{v}_0)\widehat{v}_0\\
\notag&-(\widetilde{k}_{CI}+\delta_{CI}H)
\big(\frac{v_Cv_I}{v_0^2}\widehat{\rho^{2r_s}}+\frac{\widehat{v}_Cv_I}{v_0^2}\widetilde{\rho}^{2r_s}+\frac{\widetilde{v}_C\widehat{v}_I}{v_0^2}\widetilde{\rho}^{2r_s}+\frac{\widetilde{v}_C\widetilde{v}_I}{v_0^2\widetilde{v}_0^2}(v_0+\widetilde{v}_0)\widehat{v}_0\widetilde{\rho}^{2r_s}\big)\bigg]\bigg\|_{L^2(\Sigma_t)}\\
\notag\leq&\,Ce^{2Ht-\frac{4r_s}{1-2r_s}Ht}\sum_{I=1}^3\|e_I\widehat{n}\|_{H^N(\Sigma_t)}
+Ce^{2Ht-\frac{4r_s}{1-2r_s}Ht}\big\{\|\widehat{\gamma}\|_{H^N(\Sigma_t)}
+\|\widehat{k}\|_{H^N(\Sigma_t)}+\|\widehat{n}\|_{H^N(\Sigma_t)}\big\}\\
\notag&+C(e^{2Ht-\frac{4r_s}{1-2r_s}Ht}+e^{4Ht-\frac{8r_s}{1-2r_s}Ht}e^{Ht})\|\widehat{v}\|_{H^{N-2}(\Sigma_t)}
+e^{Ht}\|\widehat{\rho^{2r_s}}\|_{H^N(\Sigma_t)}\\
\notag\leq&\,C\sum_{I=1}^3\|e_I\widehat{n}\|_{H^N(\Sigma_t)}
+Ce^{-Ht}\sqrt{\mathcal{E}_{tot}(t)}
+C(e^{-2A_sHt}+e^{2Ht-\frac{4r_s}{1-2r_s}Ht}e^{-2A_sHt}e^{Ht})\|\widehat{v}\|_{H^{N-2}(\Sigma_t)}\\
\notag\leq&\,C\sum_{I=1}^3\|e_I\widehat{n}\|_{H^N(\Sigma_t)}
+C(e^{-Ht}+e^{2Ht-\frac{4r_s}{1-2r_s}Ht})\sqrt{\mathcal{E}_{tot}(t)}.
\end{align}
Thus, we have controlled the $L^2(\Sigma_t)$ norms of all top order error terms in $\mathfrak{P}^\iota$, $|\iota|=N-1,N$. This completes the proof of the lemma.
\end{proof}
Our goal in this section is to prove the following proposition.
\begin{proposition}\label{prop:fluid.est}
Let $N\ge7$ be sufficiently large and assume the bootstrap assumptions \eqref{Boots} are valid. Then the following energy estimates hold:
\begin{align}
\label{fluid.low.est}\mathcal{E}_{fluid.low}(t)\leq&\, C\mathcal{E}_{fluid.low}(T)+C\int^t_T(e^{-H\tau}+e^{2H\tau-\frac{4r_s}{1-2r_s}H\tau})\mathcal{E}_{tot}(\tau)d\tau,\\
\label{fluid.top.est}\mathcal{E}_{fluid.top}(t)\leq&\, C\mathcal{E}_{fluid.top}(T)+C\int^t_Te^{-H\tau}e^{2A_sH\tau}\mathcal{E}_{tot}(\tau)d\tau
+C\int^t_T(e^{-H\tau}+e^{2H\tau-\frac{4r_s}{1-2r_s}H\tau})\mathcal{E}_{tot}(\tau)d\tau\\
\notag&+\frac{1}{2}\int^t_T\sum_{I=1}^3e^{3H\tau}\|e_I\widehat{n}\|^2_{L^2(\Sigma_\tau)}d\tau,
\end{align}
for all $t\in[T,T_{Boot})$.
\end{proposition}
\begin{remark}\label{rem:top.fluid.est}
Here, it is evident that we can only hope to close our estimates if $A_s=\frac{3c_s^2-1}{1-c_s^2}<\frac{1}{2}$. This is verified for $c_s^2\in(\frac{1}{3},\frac{3}{7})$. Indeed, combining the energy estimates in Propositions \ref{prop:geom.est}, \ref{prop:fluid.est}, we obtain an overall estimate for the total energy $\mathcal{E}_{tot}(t)$, see the proof of Proposition \ref{prop:cont.arg}. Since the coefficient $e^{-Ht}e^{2A_sHt}$ naturally propagates to the latter overall estimate, Gronwall's inequality only yields an estimate consistent with our bootstrap assumptions \eqref{Boots} in the case where $-1+2A_s<0$.   
\end{remark}
\begin{proof}
{\it Step 1. Derivation of \eqref{fluid.low.est}.} Multiplying the equation \eqref{vI.hat.eq.diff} with $e^{2Ht}Y^\iota \widehat{v}_I$ gives the identity:
\begin{align}\label{low.vI.en.id}
\frac{1}{2}\partial_t(e^{2Ht}Y^\iota\widehat{v}_IY^\iota\widehat{v}_I)
=\frac{1}{2}e^{2Ht}n\frac{v_C}{v_0}e_C(Y^\iota\widehat{v}_IY^\iota\widehat{v}_I)+e^{2Ht}Y^\iota\widehat{v}_I\big\{\frac{1}{2}\frac{n}{v_0}e_IY^\iota(\widehat{\rho^{2r_s}})+\mathfrak{V}_I^\iota\big\}.
\end{align}
Next, we integrate \eqref{low.vI.en.id} in $\Sigma_t$, integrate by parts in $e_C$, sum over $|\iota|\leq N-1$, and use the bootstrap assumptions \eqref{Boots}, the basic estimates \eqref{WN-3.est}, \eqref{rho.hat.boots}, and Lemmas \ref{lem:Linfty.est}, \ref{lem:frak.fluid.est} to deduce the energy inequality:
\begin{align}\label{low.vI.en.ineq}
\notag\frac{1}{2}\partial_t\sum_{I=1}^3e^{2Ht}\|\widehat{v}_I\|_{H^{N-1}(\Sigma_t)}^2
\leq&\, C e^{-Ht}e^{2Ht}\|\widehat{v}\|_{H^{N-1}(\Sigma_t)}^2+
Ce^{Ht}\|\widehat{v}\|_{H^{N-1}(\Sigma_t)}e^{Ht}\|\widehat{\rho^{2r_s}}\|_{H^N(\Sigma_t)}\\
&+Ce^{Ht}\|\widehat{v}\|_{H^{N-1}(\Sigma_t)}(e^{-Ht}+e^{2Ht-\frac{4r_s}{1-2r_s}Ht})\sqrt{\mathcal{E}_{tot}(t)}\\
\notag\leq&\, C(e^{-Ht}+e^{2Ht-\frac{4r_s}{1-2r_s}Ht})\mathcal{E}_{tot}(t).
\end{align}
Similarly, multiplying \eqref{rho.hat.eq.diff} with $e^{2PHt}Y^\iota\widehat{\rho^{2r_s}}$ gives the identity: 
\begin{align}\label{low.rho.en.id}
\notag&\frac{1}{2}\partial_t\big[e^{2PHt}(Y^\iota\widehat{\rho^{2r_s}})^2\big]+\big(2Hf^{-1}_{\frac{1-2r_s}{2r_s}}-P\big)e^{2PHt}(Y^\iota\widehat{\rho^{2r_s}})^2\\
=&-e^{2PHt}f^{-1}_{\frac{1-2r_s}{2r_s}}n\frac{\rho^{2r_s}}{v_0^2}\frac{v_Iv_D}{v_0}Y^\iota(\widehat{\rho^{2r_s}})e_DY^\iota\widehat{v}_I
+e^{2PHt}f^{-1}_{\frac{1-2r_s}{2r_s}}n\frac{\rho^{2r_s}}{v_0}Y^\iota(\widehat{\rho^{2r_s}})e_CY^\iota\widehat{v}_C\\
\notag&+\frac{1}{2}f^{-1}_{\frac{1-2r_s}{2r_s}}e^{2PHt}n\big(\frac{1-2r_s}{2r_s}-\frac{1}{2}\frac{\rho^{2r_s}}{v_0^2}\big)\frac{v_C}{v_0}e_C\big[(Y^\iota \widehat{\rho^{2r_s}})^2\big]+e^{2PHt}f^{-1}_{\frac{1-2r_s}{2r_s}}Y^\iota(\widehat{\rho^{2r_s}})\mathfrak{P}^\iota.
\end{align}
Setting $P=\frac{4r_s}{1-2r_s}$, integrating \eqref{low.rho.en.id} in $\Sigma_t$, summing over $|\iota|\leq N-2$, integrating by parts in $e_C$ in the first term in the third line \eqref{low.rho.en.id}, and arguing as above we derive:
\begin{align}\label{low.rho.en.ineq}
\notag\frac{1}{2}\partial_t\big[e^{\frac{8r_s}{1-2r_s}Ht}\|\widehat{\rho^{2r_s}}\|^2_{H^{N-2}(\Sigma_t)}\big]\leq&\, C(e^{-Ht}+e^{2Ht-\frac{4r_s}{1-2r_s}Ht})e^{\frac{8r_s}{1-2r_s}Ht}\|\widehat{\rho^{2r_s}}\|^2_{H^{N-2}(\Sigma_t)}\\
&+e^{\frac{4r_s}{1-2r_s}Ht}\|\widehat{\rho^{2r_s}}\|_{H^{N-2}(\Sigma_t)}\|\widehat{v}\|_{H^{N-1}(\Sigma_t)}\\
\notag\leq&\,C(e^{-Ht}+e^{2Ht-\frac{4r_s}{1-2r_s}Ht})\mathcal{E}_{tot}(t).
\end{align}
Similarly, for $P=2$ and summing over $|\iota|=N-1$ instead, we have
\begin{align}\label{low.rho.en.ineq2}
\notag\frac{1}{2}\partial_t\big[e^{4Ht}\|\widehat{\rho^{2r_s}}\|^2_{\dot{H}^{N-1}(\Sigma_t)}\big]\leq&\, C(e^{-Ht}+e^{2Ht-\frac{4r_s}{1-2r_s}Ht})e^{4Ht}\|\widehat{\rho^{2r_s}}\|^2_{\dot{H}^{N-1}(\Sigma_t)}\\
&+e^{4Ht-\frac{4r_s}{1-2r_s}Ht}\|\widehat{\rho^{2r_s}}\|_{\dot{H}^{N-1}(\Sigma_t)}\|\widehat{v}\|_{\dot{H}^N(\Sigma_t)}\\
\notag\leq&\,C(e^{-Ht}+e^{2Ht-\frac{4r_s}{1-2r_s}Ht})\mathcal{E}_{tot}(t).
\qquad (4-\frac{4r_s}{1-2r_s}=2-2A_s)
\end{align}
From \eqref{v0.hat.id.diff}, we also have 
\begin{align}\label{v0.en.ineq}
\|\widehat{v}_0\|^2_{H^{N-m}(\Sigma_t)}\leq C\varepsilon \|\widehat{v}_0\|^2_{H^{N-m}(\Sigma_t)}+C\sum_{I=1}^3\|\widehat{v}_I\|_{H^{N-m}(\Sigma_t)}^2+Ce^{4Ht}\|\widehat{\rho^{2r_s}}\|_{H^{N-m}(\Sigma_t)}^2,\qquad m=0,1.
\end{align}
For $\varepsilon>0$ sufficiently small, we obtain the lower and top order bounds
\begin{align}\label{v0.en.ineq}
\begin{split}
e^{2Ht}\|\widehat{v}_0\|^2_{H^{N-1}(\Sigma_t)}\leq&\, C\sum_{I=1}^3e^{2Ht}\|\widehat{v}_I\|_{H^{N-1}(\Sigma_t)}^2+Ce^{4Ht}\|\widehat{\rho^{2r_s}}\|_{H^{N-1}(\Sigma_t)}^2,\\
e^{-4A_sHt}e^{2Ht}\|\widehat{v}_0\|^2_{H^{N}(\Sigma_t)}\leq&\, C\sum_{I=1}^3e^{-4A_sHt}e^{2Ht}\|\widehat{v}_I\|_{H^N(\Sigma_t)}^2+Ce^{-4A_sHt}e^{4Ht}\|\widehat{\rho^{2r_s}}\|_{H^N(\Sigma_t)}^2.
\end{split}
\end{align}
The desired lower order energy inequality \eqref{fluid.low.est} now follows by
summing the energy inequalities \eqref{low.vI.en.ineq}, \eqref{low.rho.en.ineq}, \eqref{low.rho.en.ineq2}, integrating in $[T,t]$, and using the first bound in \eqref{v0.en.ineq}.

\vspace*{.05in}

{\it Step 2. Derivation of \eqref{fluid.top.est}.} Integrate the top order energy identity \eqref{top.fluid.id} in $\Sigma_t$, integrate by parts the terms in the third line of \eqref{top.fluid.id}, sum over $|\iota|=N$, and use the bootstrap assumptions \eqref{Boots}, the basic estimates \eqref{WN-3.est}, \eqref{rho.hat.boots}, Lemmas \ref{lem:Linfty.est}, \ref{lem:frak.fluid.est}, and Young's inequality to deduce the energy inequality:
\begin{align}\label{top.fluid.en.ineq}
\notag&\frac{1}{2}\partial_t\bigg[\sum_{|\iota|=N}\int_{\Sigma_t}e^{4Ht}\rho^{2r_s}\big\{Y^\iota\widehat{v}_IY^\iota\widehat{v}_I-\frac{v_Iv_E}{v_0^2}Y^\iota \widehat{v}_IY^\iota\widehat{v}_E\big\} d\omega+\frac{1}{2}f_{\frac{1-2r_s}{2r_s}}e^{4Ht}\|\widehat{\rho^{2r_s}}\|^2_{\dot{H}^N(\Sigma_t)}\bigg]\\
&+\frac{4r_s-1}{2r_s}He^{4Ht}\|\widehat{\rho^{2r_s}}\|_{\dot{H}^N(\Sigma_t)}^2+\frac{4r_s-1}{1-2r_s}H\sum_{|\iota|=N}\int_{\Sigma_t}e^{4Ht}\rho^{2r_s}\big\{Y^\iota\widehat{v}_IY^\iota\widehat{v}_I-\frac{v_Iv_E}{v_0^2}Y^\iota \widehat{v}_IY^\iota\widehat{v}_E\big\} d\omega\\
\notag\leq&\, C(e^{-Ht}+e^{2Ht-\frac{4r_s}{1-2r_s}Ht})e^{4Ht}\|\widehat{\rho^{2r_s}}\|_{\dot{H}^N(\Sigma_t)}^2+Ce^{2Ht}\|\widehat{\rho^{2r_s}}\|_{\dot{H}^N(\Sigma_t)}e^{2Ht-\frac{4r_s}{1-2r_s}Ht}\|v\|_{\dot{H}^N(\Sigma_t)}\\
\notag&+Ce^{-Ht}e^{4Ht-\frac{4r_s}{1-2r_s}Ht}\|\widehat{v}\|^2_{\dot{H}^N(\Sigma_t)}+Ce^{2Ht-\frac{2r_s}{1-2r_s}Ht}\|\widehat{v}\|_{\dot{H}^N(\Sigma_t)}\big\{e^{-Ht}e^{A_sHt}\sqrt{\mathcal{E}_{tot}(t)}+C\sum_{I=1}^3\|e_I\widehat{n}\|_{H^N(\Sigma_t)}\big\}\\
\notag&+Ce^{2Ht}\|\widehat{\rho^{2r_s}}\|_{\dot{H}^N(\Sigma_t)}\big\{(e^{-Ht}+e^{2Ht-\frac{4r_s}{1-2r_s}Ht})\sqrt{\mathcal{E}_{tot}(t)}+\sum_{I=1}^3\|e_I\widehat{n}\|_{H^N(\Sigma_t)}\big\}\\
\notag\leq&\,\frac{C}{\eta}\big(e^{-Ht}+e^{2Ht-\frac{4r_s}{1-2r_s}Ht}+e^{-Ht}e^{2A_sHt}\big)\mathcal{E}_{tot}(t)
+\eta\sum_{I=1}^3e^{3Ht}\|e_I\widehat{n}\|^2_{H^N(\Sigma_t)}\qquad\qquad (2-\frac{4r_s}{1-2r_s}=-2A_s)
\end{align}
for some constant $\eta>0$ of our choice. Notice that by Cauchy-Schwartz it holds
\begin{align}\label{C-S.vIvE}
Y^\iota\widehat{v}_IY^\iota\widehat{v}_I-\frac{v_Iv_E}{v_0^2}Y^\iota \widehat{v}_IY^\iota\widehat{v}_E\ge (1-\frac{v_Ev_E}{v_0^2})Y^\iota\widehat{v}_IY^\iota\widehat{v}_I=\frac{\rho^{2r_s}}{v_0^2}Y^\iota\widehat{v}_IY^\iota\widehat{v}_I.
\end{align}
Also, we recall that by Lemma \ref{lem:Linfty.est}, $\frac{\rho^{2r_s}}{v_0^2}\sim e^{2Ht-\frac{4r_s}{1-2r_s}Ht}=e^{-2A_sHt}$, $e^{2Ht}\rho^{2r_s}\sim e^{-2A_sHt}$. Hence, integrating \eqref{top.fluid.en.ineq} in $[T,t]$, it follows that 
\begin{align}\label{top.fluid.en.ineq2}
&\sum_{I=1}^3e^{-4A_sHt}e^{2Ht}\|\widehat{v}_I\|^2_{\dot{H}^N(\Sigma_t)}+e^{4Ht}\|\widehat{\rho^{2r_s}}\|^2_{\dot{H}^N(\Sigma_t)}\\
\notag\leq&\,C\mathcal{E}_{fluid.top}(T)+\frac{C}{\eta}\int^t_T\big(e^{-Ht}+e^{2Ht-\frac{4r_s}{1-2r_s}Ht}+e^{-Ht}e^{2A_sHt}\big)\mathcal{E}_{tot}(t)
+\eta\sum_{I=1}^3e^{3Ht}\|e_I\widehat{n}\|^2_{H^N(\Sigma_t)}.
\end{align}
Using in addition the second bound in \eqref{v0.en.ineq} and choosing $\eta$ sufficiently small, we conclude the desired top order energy estimate \eqref{fluid.top.est}.
\end{proof}

\subsection{Global existence in all of $[T,+\infty)\times\mathbb{S}^3$}

Having derived the main energy estimates in Propositions \ref{prop:geom.est}, \ref{prop:fluid.est}, we can now complete the overall bootstrap argument by improving our assumptions \eqref{Boots} and obtain a global stability result.  
\begin{proposition}\label{prop:cont.arg}
Let $N\ge7$ and assume the bootstrap assumptions \eqref{Boots} are valid. Then
the following energy estimate holds:
\begin{align}\label{tot.en.est}
\mathcal{E}_{tot}(t)\leq C\mathcal{E}_{tot}(T)+C\int_T^t(e^{-\frac{1}{2}H\tau}+e^{2H\tau-\frac{4r_s}{1-2r_s}H\tau})\mathcal{E}_{tot}(\tau)d\tau+
C\int_T^te^{-H\tau}e^{2A_sHt}\mathcal{E}_{tot}(\tau)d\tau,
\end{align}
for all $t\in[T,T_{Boot})$. Moreover, if $2A_s<1$, ie. $c_s^2\in(\frac{1}{3},\frac{3}{7})$, and $\mathcal{E}_{tot}(T)$ is sufficiently small such that 
\begin{align}\label{tot.en.init}
C\mathcal{E}_{tot}(T)\exp\{C\int^{+\infty}_Te^{-\frac{1}{2}H\tau}+e^{2H\tau-\frac{4r_s}{1-2r_s}H\tau}+e^{-H\tau}e^{2A_sH\tau}d\tau\}<\varepsilon,
\end{align}
then \eqref{tot.en.est} yields a strict improvement of the bootstrap assumptions \eqref{Boots}. In the latter case, $T_{Boot}=+\infty$ and the estimate
\begin{align}\label{tot.en.est2}
\mathcal{E}_{tot}(t)<\varepsilon
\end{align}
holds for all $t\in[T,+\infty)$.
\end{proposition}
\begin{proof}
Summing the energy estimates \eqref{geom.est}, \eqref{fluid.low.est}, \eqref{fluid.top.est} and readjusting the constant $C$ gives \eqref{tot.en.est}. By Gronwall's inequality, we obtain 
\begin{align}\label{tot.en.est3}
\mathcal{E}_{tot}(t)\leq C\mathcal{E}_{tot}(T)\exp\{C\int^t_Te^{-\frac{1}{2}H\tau}+e^{2H\tau-\frac{4r_s}{1-2r_s}H\tau}+e^{-H\tau}e^{2A_sH\tau}d\tau\},
\end{align}
for all $t\in[T,T_{boot})$. 

Standard local well-posedness implies that the perturbed solution has a maximal time of existence $T_{max}>T$. If $T_{max}<+\infty$, then for initial data sufficiently close to the background solution on $t=T$, consistent with \eqref{tot.en.init}, there must exist a time $t_0\in[T,T_{max}]$, such that $\sup_{t\in[T,t_0)}\mathcal{E}_{tot}(t)=\varepsilon$, otherwise the solution can be further extended past $t=T_{max}$ by standard continuation criteria. The latter equality cannot be true due to \eqref{tot.en.est3} and \eqref{tot.en.init}. We conclude that $T_{max}=+\infty$ and that the bootstrap assumptions \eqref{Boots} are never saturated, that is, \eqref{tot.en.est2} holds for all $t\in[T,+\infty)$.
\end{proof}
\section{The precise asymptotic behavior towards infinity}\label{sec:refined.est}

In this section, we derive the precise asymptotic behavior of each of the reduced variables $\widehat{k}_{IJ},\widehat{\gamma}_{IJB},\widehat{e}_I^i,\widehat{n}$, $\widehat{\rho^{2r_s}},\widehat{v}_0,\widehat{v}_I$ using the overall energy estimate \eqref{tot.en.est3} and the equations in Lemma \ref{lem:red.eq2}. In particular, we prove:
\begin{proposition}\label{prop:asym.beh}
Let $N\ge7$ and assume Proposition \ref{prop:cont.arg} holds. Then the reduced variables $\widehat{k}_{IJ},\widehat{n}$ satisfy the improved estimates: 
\begin{align}\label{k.n.hat.imp.est}
\|\widehat{k}_{IJ}\|_{W^{N-4,\infty}(\Sigma_t)}+\|\widehat{n}\|_{W^{N-4,\infty}(\Sigma_t)}\leq C\varepsilon e^{-2Ht}.
\end{align}
Moreover, there exist functions $(\widehat{e}^\infty)_I^i\in W^{N-4,\infty}(\mathbb{S}^3)$, $\widehat{\rho^{2r_s}}^\infty,\widehat{v}_0^\infty,\widehat{v}_I^\infty\in W^{N-5,\infty}(\mathbb{S}^3)$ such that 
\begin{align}
\label{e.hat.infty}\|\widehat{e}^i_I-(\widehat{e}^\infty)^i_I(\omega)e^{-Ht}\|_{W^{N-4,\infty}(\Sigma_t)}\leq &\, C\varepsilon e^{-3Ht},\\
\label{v.hat.infty}\|\widehat{v}_\mu-\widehat{v}_\mu^\infty(\omega)e^{-Ht}\|_{W^{N-5,\infty}(\Sigma_t)}\leq &\,C\varepsilon e^{-2Ht},\\
\label{rho.hat.infty}\|\widehat{\rho^{2r_s}}-\widehat{\rho^{2r_s}}^\infty(\omega)e^{-\frac{4r_s}{1-2r_s}Ht}\|_{W^{N-5,\infty}(\Sigma_t)}\leq &\,C\varepsilon(e^{-Ht}+e^{2Ht-\frac{4r_s}{1-2r_s}Ht})e^{-\frac{4r_s}{1-2r_s}Ht},
\end{align}
for all $t\in[T,+\infty)$, where we recall that $\frac{4r_s}{1-2r_s}>2$, $r_s>\frac{1}{4}$.
\end{proposition}
\begin{proof}
The overall energy estimate \eqref{tot.en.est3}, together with the basic estimates \eqref{WN-3.est}, \eqref{rho.hat.boots}, applied to control the RHS of \eqref{k.hat.eq} yield the bound
\begin{align}\label{k.hat.ode.est}
\|\partial_t\widehat{k}_{IJ}+3H\widehat{k}_{IJ}\|_{W^{N-4,\infty}(\Sigma_t)}\leq C\varepsilon e^{-2Ht},
\end{align}
which in turn implies 
\begin{align}\label{k.hat.ode.est2}
\notag&\|e^{3Ht}\widehat{k}_{IJ}-e^{3HT}\widehat{k}_{IJ}(T,\omega)\|_{W^{N-4,\infty}(\Sigma_t)}\\
=&\,
\bigg\|\int^t_T\partial_t(e^{3H\tau}\widehat{k}_{IJ})d\tau\bigg\|_{W^{N-4,\infty}(\Sigma_t)}
\leq \int^t_T\|\partial_t(e^{3H\tau}\widehat{k}_{IJ}\|_{W^{N-4,\infty}(\Sigma_\tau)}d\tau
\leq C\varepsilon\int^t_Te^{H\tau}d\tau\leq C\varepsilon e^{Ht}.
\end{align}
Multiplying \eqref{k.hat.ode.est2} with $e^{-3Ht}$ and using the triangle inequality, we obtain
\begin{align}\label{k.hat.ode.est3}
\|\widehat{k}_{IJ}\|_{W^{N-4,\infty}(\Sigma_t)}\leq \|\widehat{k}_{IJ}\|_{W^{N-4,\infty}(\Sigma_T)}e^{-3Ht}+ C\varepsilon e^{-2Ht}\leq C\varepsilon e^{-2Ht}.
\end{align}
The argument for $\widehat{n}$ is similar. This completes the proof of \eqref{k.n.hat.imp.est}. 

Having derived the improved estimates for $\widehat{k}_{IJ},\widehat{n}$, we proceed to derive the precise asymptotic behaviors of $\widehat{e}^i_I,\widehat{v}_\mu,\widehat{\rho^{2r_s}}$. Treating the equations \eqref{eIi.hat.eq}, \eqref{vI.hat.eq}, \eqref{rho.hat.eq} as ODEs in $t$, using the energy estimate \eqref{tot.en.est2} to control the corresponding RHSs, we deduce the bounds:
\begin{align}\label{e.v.rho.hat.ode.est}
\notag \|\partial_t(e^{Ht}\widehat{e}_I^i)\|_{W^{N-5}(\Sigma_t)}\leq&\, C\varepsilon e^{-2Ht},\\
\|\partial_t(e^{Ht}\widehat{v}_I)\|_{W^{N-5}(\Sigma_t)}\leq&\, C\varepsilon e^{-Ht},\\
\notag \|\partial_t(e^{\frac{4r_s}{1-2r_s}Ht}\widehat{\rho^{2r_s}})\|_{W^{N-5}(\Sigma_t)}\leq &\,C\varepsilon(e^{-Ht}+e^{2Ht-\frac{4r_s}{1-2r_s}Ht}),
\end{align}
for all $t\in[T,+\infty)$. Since the latter RHSs are integrable in $[T,+\infty)$, it follows that the renormalized quantities $e^{Ht}\widehat{e}_I^i,e^{Ht}\widehat{v}_I,e^{\frac{4r_s}{1-2r_s}Ht}\widehat{\rho^{2r_s}}$ have limits as $t\rightarrow+\infty$, denoted by $(\widehat{e}^\infty)^i_I(\omega)\in W^{N-4,\infty}(\mathbb{S}^3)$, $\widehat{v}^\infty_I(\omega),\widehat{\rho^{2r_s}}^\infty(\omega)\in W^{N-5,\infty}(\Sigma_t)$. Integrating \eqref{e.v.rho.hat.ode.est} in $[t,+\infty)$ gives the desired asymptotic behaviors \eqref{e.hat.infty}-\eqref{rho.hat.infty} for $\widehat{e}_I^i,\widehat{v}_I,\widehat{\rho^{2r_s}}$. Indeed, the argument for $\widehat{e}_I^i$ is as follows:
\begin{align}\label{e.hat.infty2}
\notag\|\widehat{e}_I^i-(\widehat{e}^\infty)_I^i(\omega)e^{-Ht}\|_{W^{N-4,\infty}(\Sigma_t)}=&\,e^{-Ht}\bigg\|\int^{+\infty}_t\partial_\tau(e^{H\tau}\widehat{e}_I^i)d\tau\bigg\|_{W^{N-4,\infty}(\Sigma_\tau)}\\
\leq&\, e^{-Ht}\int^{+\infty}_t\|\partial_\tau(e^{H\tau}\widehat{e}_I^i)\|_{W^{N-4,\infty}(\Sigma_\tau)}d\tau\\
\notag\leq&\, e^{-Ht}\int_t^{+\infty}C\varepsilon e^{-2H\tau}d\tau\leq C\varepsilon e^{-3Ht}.
\end{align}
The derivations for $\widehat{v}_I,\widehat{\rho^{2r_s}}$ are similar. To conclude the asymptotic behavior for $\widehat{v}_0$, we use the identity \eqref{v0.hat.id.diff} and the already derived behaviors of $\widehat{v}_I,\widehat{\rho^{2r_s}}$.
\end{proof}

\appendix

\section{Derivation of the equations of motion}\label{app:divT}

\begin{lemma}\label{lem:eq.motion}
The equations of motion ${\bf D}^\mu T_{\mu\nu}=0$, for $T_{\mu\nu}$ given by \eqref{Tmunu} with $p=c_s^2\rho$, are equivalent to the system \eqref{divT.v}-\eqref{divT.rho}.
\end{lemma}
\begin{proof}
Taking the divergence of \eqref{Tmunu}, for $p=c_s^2\rho$, we have
\begin{align}
\label{divT}(1+c_s^2)u^\mu (e_\mu\rho) u_\nu+(1+c_s^2)\rho({\bf D}^\mu u_\mu)u_\nu+(1+c_s^2)\rho ({\bf D}_uu)_\nu+c_s^2{\bf D}_\nu\rho=0.
\end{align}
Contracting \eqref{divT} with $u^\nu$ gives
\begin{align}
\notag
u^\mu (e_\mu\rho)+(1+c_s^2)\rho({\bf D}^\mu u_\mu)=0,&\\
\frac{1}{1+c_s^2}\rho^{-\frac{c_s^2}{1+c_s^2}}u^\mu (e_\mu\rho)+\rho^{\frac{1}{1+c_s^2}}({\bf D}^\mu u_\mu)=0,&\label{divT.rho2}\\
\notag{\bf D}^\mu(\rho^\frac{1}{1+c^2_s} u_\mu)=0.&
\end{align}
Rewrite the energy momentum tensor \eqref{Tmunu} using $p=c_s^2\rho$ and \eqref{v.rs}: 
\begin{align}\label{Tmunu.2}
T_{\mu\nu}=(1+c_s^2)\rho^\frac{1-c_s^2}{1+c_s^2}v_\mu v_\nu+c_s^2\rho {\bf g}_{\mu\nu}.
\end{align}
Then the divergence of $T_{\mu\nu}$ becomes
\begin{align}\label{divT.2}
\notag{\bf D}^\mu T_{\mu\nu}=&\,(1+c_s^2){\bf D}^\mu(\rho^{\frac{1-c_s^2}{1+c_s^2}}v_\mu)v_\nu+(1+c_s^2)\rho^\frac{1-c_s^2}{1+c_s^2}({\bf D}_v v)_\nu+c_s^2{\bf D}_\nu\rho \\
=&\,(1+c_s^2){\bf D}^\mu(\rho^{\frac{1}{1+c_s^2}}u_\mu)v_\nu+(1+c_s^2)\rho^\frac{1-c_s^2}{1+c_s^2}({\bf D}_v v)_\nu+c_s^2{\bf D}_\nu\rho\\
\tag{by \eqref{divT.rho2}}=&\,(1+c_s^2)\rho^\frac{1-c_s^2}{1+c_s^2}({\bf D}_v v)_\nu+c_s^2{\bf D}_\nu\rho.
\end{align}
We arrive at \eqref{divT.v} after setting the last equation to zero and multiplying with $(1+c_s^2)^{-1}\rho^{\frac{c_s^2-1}{1+c_s^2}}$.

Going back to \eqref{divT.rho2}, we replace $u _\mu$ in favor of $v_\mu$ and expand the divergence:
\begin{align}\label{divT.rho3}
0={\bf D}^\mu(\rho^\frac{1}{1+c^2_s} u_\mu)={\bf D}^\mu(\rho^{\frac{1-c_s^2}{1+c_s^2}}v_\mu)=\frac{1-c_s^2}{1+c_s^2}\rho^{-\frac{2c_s^2}{1+c_s^2}}{\bf D}_v\rho+\rho^{\frac{1-c_s^2}{1+c_s^2}}{\bf D}^\mu v_\mu 
\end{align}
Multiplying the last RHS with $\rho^{\frac{c_s^2-1}{1+c_s^2}}$ gives \eqref{divT.rho}.
\end{proof}

\section{ODE analysis for the homogeneous tilted backgrounds}\label{app:ODE}

Here, we prove the existence of homogeneous solutions to the Einstein-Euler system, of the form \eqref{metric.hom}-\eqref{rho.hom}, and derive their precise asymptotic behaviors. For consistency, we express the homogeneous solutions in the gauge presented in Section \ref{sss:framework}. 

Consider the $\widetilde{g}$-orthonormal frame \eqref{frame.hom}. We compute the associated spatial connection coefficients using the Koszul formula:
\begin{align}\label{gamma.hom}
\widetilde{\gamma}_{IJB}=\frac{1}{2}\big\{\widetilde{g}([\widetilde{e}_I,\widetilde{e}_J],\widetilde{e}_B)-\widetilde{g}([\widetilde{e}_J,\widetilde{e}_B],\widetilde{e}_I)+\widetilde{g}([\widetilde{e}_B,\widetilde{e}_I],\widetilde{e}_J)\big\},
\end{align}
the choice of frame \eqref{frame.hom}, the commutation relation \eqref{[Yi,Yj]}, and the form of the metric \eqref{metric.hom}. They read:
\begin{align}\label{gamma.hom2}
\notag
\widetilde{\gamma}_{112}=\widetilde{\gamma}_{113}=\widetilde{\gamma}_{223}=\widetilde{\gamma}_{332}=0,\qquad\qquad\qquad\qquad\qquad\qquad\qquad \\
\notag \widetilde{\gamma}_{221}=-\widetilde{g}([\widetilde{e}_2,\widetilde{e}_1],\widetilde{e}_2)=2G^{-2}_2(t)G_1^{-1}(t)G^2(t),\qquad\widetilde{\gamma}_{331}=-\widetilde{g}([\widetilde{e}_3,\widetilde{e}_1],\widetilde{e}_3)=-2G^{-2}_2(t)G_1^{-1}(t)G^2(t),\\
 \widetilde{\gamma}_{123}=\frac{G^{-1}_1(t)G_2^{-1}(t)}{\sqrt{G_3^2(t)-G^4(t)G_2^{-2}(t)}}[G_3^2(t)-G_1^2(t)+G_2^2(t)],\\
\notag\widetilde{\gamma}_{231}=\frac{G^{-1}_1(t)G_2^{-1}(t)}{\sqrt{G_3^2(t)-G^4(t)G_2^{-2}(t)}}[G_3^2(t)-G_2^2(t)+G_1^2(t)-2G^4(t)G_2^{-2}(t)],\\
\notag\widetilde{\gamma}_{312}=\frac{G^{-1}_1(t)G_2^{-1}(t)}{\sqrt{G_3^2(t)-G^4(t)G_2^{-2}(t)}}[G_1^2(t)-G_3^2(t)+G_2^2(t)+2G^4(t)G_2^{-2}(t)].
\end{align}
Here, the frame and metric components are related via: 
\begin{align}\label{tilde:e=G}
\widetilde{e}_1^1=G^{-1}_1(t),\qquad \widetilde{e}_2^2=G^{-1}_2(t),\qquad\widetilde{e}_3^3=\frac{1}{\sqrt{G_3^2(t)-G^4(t)G_2^{-2}(t)}},\qquad\widetilde{e}_3^2=-\frac{G^2(t)G_2^{-2}(t)}{\sqrt{G_3^2(t)-G^4(t)G_2^{-2}(t)}}.
\end{align}
Notice that $\widetilde{\gamma}_{221}=-\widetilde{\gamma}_{331}$. Hence, we need only solve for $\gamma_{221}$ below.
Expand also the fluid speed relative to $\widetilde{e}_0,\widetilde{e}_1$
\begin{align}\label{fluid.hom.app}
\widetilde{u}=-\widetilde{u}_0\widetilde{e}_0+\widetilde{u}_1\widetilde{e}_1
\end{align}
and consider the renormalized fluid components $\widetilde{v}_0=\widetilde{\rho}^{r_s}\widetilde{u}_0,\widetilde{v}_1=\widetilde{\rho}^{r_s}\widetilde{u}_1$.

Then the identities in Lemma \ref{lem:red.eq} become:
\begin{align}
\label{k11.hom.eq}\partial_t\widetilde{k}_{11}-\mathrm{tr}\widetilde{k}\widetilde{k}_{11}=&-2\widetilde{\gamma}_{221}^2
+2\widetilde{\gamma}_{231}\widetilde{\gamma}_{312}
-\Lambda-(1+c_s^2)\widetilde{\rho}^{1-2r_s}\widetilde{v}_1^2-\frac{1}{2}(1-c_s^2)\widetilde{\rho},\\
\label{k22.hom.eq}\partial_t\widetilde{k}_{22}-\mathrm{tr}\widetilde{k}\widetilde{k}_{22}=&\,2\widetilde{\gamma}_{123}\widetilde{\gamma}_{312}-\Lambda-\frac{1}{2}(1-c_s^2)\widetilde{\rho},\\
\label{k33.hom.eq}\partial_t\widetilde{k}_{33}-\mathrm{tr}\widetilde{k}\widetilde{k}_{33}=&\,2\widetilde{\gamma}_{231}\widetilde{\gamma}_{123}-\Lambda-\frac{1}{2}(1-c_s^2)\widetilde{\rho},\\
\label{k23.hom.eq}\partial_t\widetilde{k}_{23}-\mathrm{tr}\widetilde{k}\widetilde{k}_{23}=&\,2\widetilde{\gamma}_{123}\widetilde{\gamma}_{221},\\
\label{gamma221.hom.eq}\partial_t\widetilde{\gamma}_{221}-\widetilde{k}_{11}\widetilde{\gamma}_{221}=&-2\widetilde{k}_{23}\widetilde{\gamma}_{123}-\widetilde{k}_{23}\widetilde{\gamma}_{231}-\widetilde{k}_{23}\widetilde{\gamma}_{312}\\
\label{gamma123.hom.eq}\partial_t\widetilde{\gamma}_{123}-\widetilde{k}_{11}\widetilde{\gamma}_{123}=&\,(\widetilde{k}_{11}-\widetilde{k}_{22})\widetilde{\gamma}_{312}+(\widetilde{k}_{11}-\widetilde{k}_{33})\widetilde{\gamma}_{231}-2\widetilde{k}_{23}\widetilde{\gamma}_{221}\\
\label{gamma231.hom.eq}\partial_t\widetilde{\gamma}_{231}-\widetilde{k}_{22}\widetilde{\gamma}_{231}=&\,(\widetilde{k}_{22}-\widetilde{k}_{33})\widetilde{\gamma}_{123}+(\widetilde{k}_{22}-\widetilde{k}_{11})\widetilde{\gamma}_{312}\\
\label{gamma312.hom.eq}\partial_t\widetilde{\gamma}_{312}-\widetilde{k}_{33}\widetilde{\gamma}_{312}=&\,(\widetilde{k}_{33}-\widetilde{k}_{11})\widetilde{\gamma}_{231}+(\widetilde{k}_{33}-\widetilde{k}_{22})\widetilde{\gamma}_{123}\\
\label{e1.hom.eq}\partial_t\widetilde{e}_1^1=&\,\widetilde{k}_{11}\widetilde{e}_1^1,\\
\label{e2.hom.eq}\partial_t\widetilde{e}_2^2=&\,\widetilde{k}_{22}\widetilde{e}_2^2+\widetilde{k}_{23}\widetilde{e}_3^2,\\
\label{e3.hom.eq}\partial_t\widetilde{e}_3^3=&\,\widetilde{k}_{33}\widetilde{e}_3^3,\\
\label{e3.hom.eq}\partial_t\widetilde{e}_3^2=&\,\widetilde{k}_{33}\widetilde{e}_3^2+\widetilde{k}_{23}\widetilde{e}_2^2,\\
\label{vI.hom.eq}
\partial_t \widetilde{v}_1=&\,\widetilde{k}_{11}\widetilde{v}_1,\\
\label{rho.hom.eq}-(1-2r_s)\partial_t\log\widetilde{\rho}+\mathrm{tr}\widetilde{k}=&\,\partial_t\log \widetilde{v}_0,
\end{align}
where $\widetilde{v}_0^2=\widetilde{v}_1^2+\widetilde{\rho}^{2r_s}$. Also, the constraints reduce to
\begin{align}
\label{Hamconst.hom}
&2\widetilde{\gamma}_{231}\widetilde{\gamma}_{123}+2\widetilde{\gamma}_{123}\widetilde{\gamma}_{312}+2\widetilde{\gamma}_{312}\widetilde{\gamma}_{231}-2\widetilde{\gamma}_{221}^2\\
\notag=&\,\widetilde{k}_{11}^2+\widetilde{k}_{22}^2+\widetilde{k}_{33}^2+2k_{23}^2-(\mathrm{tr}\widetilde{k})^2+2\Lambda+2(1+c_s^2)\widetilde{\rho}^{1-2r_s}\widetilde{v}_0^2-2c_s^2\widetilde{\rho},\\
\label{momconst.hom}&(\widetilde{k}_{22}-\widetilde{k}_{33})\widetilde{\gamma}_{221}-\widetilde{k}_{23}(\widetilde{\gamma}_{312}-\widetilde{\gamma}_{231})\\
\notag=&-(1+c_s^2)\widetilde{\rho}^{1-2r_s} \widetilde{v}_0\widetilde{v}_1.
\end{align}
Existence for the above system of equations leads to a homogeneous solution of the Einstein-Euler system, as long as the constraints are asymptotically satisfied to a certain order. 

The solutions we will be looking at admit the expansions
\begin{align}\label{hom.exp}
\notag\widetilde{k}_{\underline{I}\underline{I}}=-H+\sum_{m=2}^{6} \widetilde{k}_{IJ}^{\infty,m}e^{-mHt}+\mathcal{O}(e^{-7Ht}),\qquad \widetilde{k}_{23}=\widetilde{k}_{23}^{\infty,3}e^{-3Ht}+\mathcal{O}(e^{-5Ht}),\\
\widetilde{\gamma}_{221}=\mathcal{O}(e^{-4Ht}),\qquad
\widetilde{\gamma}_{IJB}=\widetilde{\gamma}_{IJB}^{\infty,1}e^{-Ht}+\sum_{m=3}^5\widetilde{\gamma}_{IJB}^{\infty,m}e^{-mHt}+\mathcal{O}(e^{-6Ht}),\quad I\neq J\neq B\neq I,\\
\notag\widetilde{e}_{I}^i=(\widetilde{e}_{I}^i)^{\infty,1}e^{-Ht}+\mathcal{O}(e^{-3Ht}),\quad I=i,\qquad\widetilde{e}_3^2=\mathcal{O}(e^{-4Ht}),\\
\notag\widetilde{v}_\mu=\sum_{m=1,3}\widetilde{v}_\mu^{\infty,m}e^{-mHt}+\mathcal{O}(e^{-4Ht}),\qquad
\widetilde{\rho}^{1-2r_s}=\sum_{m=2,4}(\widetilde{\rho}^{1-2r_s})^{\infty,m}e^{-mHt}+\mathcal{O}(e^{-5Ht}).
\end{align}
The initial data of the above variables at infinity consist of the constants $\widetilde{k}_{IJ}^{\infty,3},\widetilde{\gamma}_{IJB}^{\infty,1},(\widetilde{e}_I^i)^{\infty,1},\widetilde{v}_\mu^{\infty,1},(\widetilde{\rho}^{1-2r_s})^{\infty,2}$, though the values $\widetilde{\gamma}_{IJB}^{\infty,1}$ are determined from $(\widetilde{e}_I^i)^{\infty,1}$ and vice versa. Also, the extreme tilt of the fluid translates to the condition $|\widetilde{v}_0^{\infty,1}|=|\widetilde{v}_1^{\infty,1}|$. Finally, the remaining constants are further constrained by the identities \eqref{Hamconst.hom}-\eqref{momconst.hom}. More precisely, we have the following lemma.
\begin{lemma}\label{lem:app.hom}
Given constants $\widetilde{k}_{\underline{I}\underline{I}}^{\infty,3},\widetilde{k}^{\infty,3}_{23},G_I^{\infty,-1},\widetilde{v}^{\infty,1}_\mu,(\widetilde{\rho}^{1-2r_s})^{\infty,2}$, there exists a unique solution to the evolution equations \eqref{k11.hom.eq}-\eqref{rho.hom.eq} with the asymptotic behavior \eqref{hom.exp}, such that 
\begin{align}\label{e.gamma.ind.const}
\begin{split}
(\widetilde{e}^i_I)^{\infty,1}=&\,(G^{\infty,1}_I)^{-1},\qquad i=I,\\
\widetilde{\gamma}_{123}^{\infty,1}=&\,(G^{\infty,1}_1)^{-1}(G^{\infty,1}_2)^{-1}(G^{\infty,1}_3)^{-1}[(G^{\infty,1}_3)^2-(G^{\infty,1}_1)^2+(G^{\infty,1}_2)^2],\\
\widetilde{\gamma}_{231}^{\infty,1}=&\,(G^{\infty,1}_1)^{-1}(G^{\infty,1}_2)^{-1}(G^{\infty,1}_3)^{-1}[(G^{\infty,1}_3)^2-(G^{\infty,1}_2)^2+(G^{\infty,1}_1)^2],\\
\widetilde{\gamma}_{312}^{\infty,1}=&\,(G^{\infty,1}_1)^{-1}(G^{\infty,1}_2)^{-1}(G^{\infty,1}_3)^{-1}[(G^{\infty,1}_1)^2-(G^{\infty,1}_3)^2+(G^{\infty,2}_2)^2],
\end{split}
\end{align}
defined over an interval $[T,+\infty)$, for some $T>0$ sufficiently large. If in addition the following identities hold:
\begin{align}\label{hom.init.const}
\notag
&\sum_{I,J}(\widetilde{k}^{\infty,3}_{IJ})^2-(\mathrm{tr}k^{\infty,3})^2
+\sum_{I,J}(\widetilde{k}^{\infty,2}_{IJ}\widetilde{k}^{\infty,4}_{IJ}-2H\widetilde{k}_{IJ}^{\infty,6})-2\mathrm{tr}k^{\infty,2}\mathrm{tr}k^{\infty,4}+6H\mathrm{tr}k^{\infty,6}\\
\notag&2(1+c_s^2)\big\{(\widetilde{\rho}^{1-2r_s})^{\infty,4}(\widetilde{v}_0^{\infty,1})^2+2(\widetilde{\rho}^{1-2r_s})^{\infty,2}\widetilde{v}_0^{\infty,1}\widetilde{v}_0^{\infty,3}\big\}-2c_s^2\big[\widetilde{\rho}^{1-2r_s})^{\infty,2}\big]^{\frac{1}{1-2r_s}}\\
=&\,2\sum_{I\neq J\neq B\neq I}(\widetilde{\gamma}^{\infty,1}_{IJB}\widetilde{\gamma}^{\infty,5}_{BIJ}+\frac{1}{2}\widetilde{\gamma}^{\infty,3}_{IJB}\widetilde{\gamma}^{\infty,3}_{BIJ}),\\
\notag&\widetilde{k}_{23}^{\infty,3}(\widetilde{\gamma}_{312}^{\infty,1}-\widetilde{\gamma}_{231}^{\infty,1})\\
\notag=&\,(1+c_s^2)(\rho^{1-2r_s})^{\infty,2}\widetilde{v}_0^{\infty,1}\widetilde{v}_1^{\infty,1},
\end{align}
the metric $\widetilde{g}$ of the form \eqref{metric.hom}, defined from $\widetilde{e}_I^i$ via \eqref{tilde:e=G}, together with $\widetilde{\rho},\widetilde{u}=\widetilde{\rho}^{-r_s}\widetilde{v}$, constitute a homogeneous solution to the Einstein-Euler system of the form \eqref{metric.hom}-\eqref{rho.hom}, satisfying the estimates in Lemma \ref{lem:hom.est}. 
\end{lemma}
\begin{proof}
Set
\begin{align}\label{RIJ}
\widetilde{k}_{II}=:-H\delta_{IJ}+R_{IJ}
\end{align}
and rewrite equations \eqref{k11.hom.eq}-\eqref{rho.hom.eq} schematically:
\begin{align}
\label{RII.eq}\partial_tR_{\underline{I}\underline{I}}+3HR_{\underline{I}\underline{I}}=&\,R_{\underline{I}\underline{I}}\text{tr}R-H\text{tr}R+\widetilde{\gamma}\star\widetilde{\gamma}-\delta_{1I}(1+c_s^2)\widetilde{\rho}^{1-2r_s}\widetilde{v}_1^2-\frac{1}{2}(1-c_s^2)\widetilde{\rho},\\
\label{R23.eq}\partial_tR_{23}+3HR_{23}=&\,R_{23}\text{tr}R+2\widetilde{\gamma}_{123}\widetilde{\gamma}_{221},\\
\label{gamma.hom.sch}\partial_t\widetilde{\gamma}_{221}+H\widetilde{\gamma}_{221}=&\,R_{11}\widetilde{\gamma}_{221}+R_{23}\star\widetilde{\gamma}\\
\label{gamma.hom.sch2}\partial_t\widetilde{\gamma}_{IJB}+H\widetilde{\gamma}_{IJB}=&\,R_{\underline{I}\underline{I}}\widetilde{\gamma}_{IJB}+R\star \widetilde{\gamma},\qquad I\neq J\neq B\neq I\\
\label{e.hom.sch}\partial_t\widetilde{e}^i_I+H\widetilde{e}_I^i=&\,R\star\widetilde{e},\\
\label{v.hom.sch}\partial_t\widetilde{v}_1+H\widetilde{v}_1=&-R_{11}\widetilde{v}_1,\\
\label{rho.hom.sch}\partial_t\log (e^{3Ht}\widetilde{\rho}^{1-2r_s}\widetilde{v}_0)=&\,\text{tr}R.
\end{align}
The constants $\widetilde{k}^{\infty,3}_{IJ},\widetilde{v}^{\infty,1}_,(\widetilde{\rho}^{1-2r_s})^{\infty,2}$ and $(\widetilde{e}_I^i)^{\infty,1},\widetilde{\gamma}_{IJB}^{\infty,1}$, as in \eqref{e.gamma.ind.const}, are the initial conditions for the variables $R_{IJ},\widetilde{\gamma}_{IJB},\widetilde{e}^i_I,\widetilde{v}_1,\widetilde{\rho}^{1-2r_s}$ at infinity, with trivial data for $\widetilde{\gamma}_{221},\widetilde{e}_3^2$.

A standard Picard iteration argument gives rise to a solution to \eqref{RII.eq}-\eqref{rho.hom.sch}, defined in an interval $[T,+\infty)$, for some $T$ sufficiently large, and satisfying
\begin{align}\label{hom.rough.exp}
R_{\underline{I}\underline{I}}=\mathcal{O}(e^{-2Ht}),\qquad R_{23}=\widetilde{k}^{\infty,3}e^{-3Ht}+\mathcal{O}(e^{-5Ht}),\qquad\widetilde{\gamma}_{221}=\mathcal{O}(e^{-5Ht}),\\
\widetilde{\gamma}_{IJB}=\widetilde{\gamma}^{\infty,1}_{IJB}e^{-Ht}+\mathcal{O}(e^{-3Ht}),\qquad I\neq J\neq B\neq I,\\
\widetilde{e}_{I}^i=(\widetilde{e}_{I}^i)^{\infty,1}e^{-Ht}+\mathcal{O}(e^{-3Ht}),\quad I=i,\qquad\widetilde{e}_3^2=\mathcal{O}(e^{-4Ht}),\\
\notag\widetilde{v}_\mu=\widetilde{v}_\mu^{\infty,1}e^{-Ht}+\mathcal{O}(e^{-3Ht}),\qquad
\widetilde{\rho}^{1-2r_s}=(\widetilde{\rho}^{1-2r_s})^{\infty,2}e^{-2Ht}+\mathcal{O}(e^{-4Ht}).
\end{align}
%
Computing the higher orders in the expansions of the previous variables yields the desired asymptotic behavior \eqref{hom.exp} in a straightforward manner. 

Let 
\begin{align}
\label{C0.cal}\mathcal{C}_0:
=&\,\widetilde{k}_{11}^2+\widetilde{k}_{22}^2+\widetilde{k}_{33}^2+2k_{23}^2-(\mathrm{tr}\widetilde{k})^2+2\Lambda+2(1+c_s^2)\widetilde{\rho}^{1-2r_s}\widetilde{v}_0^2-2c_s^2\widetilde{\rho}\\
\notag&-2\widetilde{\gamma}_{231}\widetilde{\gamma}_{123}-2\widetilde{\gamma}_{123}\widetilde{\gamma}_{312}-2\widetilde{\gamma}_{312}\widetilde{\gamma}_{231}+\widetilde{\gamma}_{221}^2+\widetilde{\gamma}_{331}^2,\\
\label{C1.cal}\mathcal{C}_1:=&\,(\widetilde{k}_{22}-\widetilde{k}_{11})\widetilde{\gamma}_{221}+(\widetilde{k}_{33}-\widetilde{k}_{11})\widetilde{\gamma}_{331}-\widetilde{k}_{23}(\widetilde{\gamma}_{312}-\widetilde{\gamma}_{231})
+(1+c_s^2)\widetilde{\rho}^{1-2r_s} \widetilde{v}_0\widetilde{v}_1.
\end{align}
Differentiating $\mathcal{C}_0,\mathcal{C}_1$ and plugging in \eqref{k11.hom.eq}-\eqref{rho.hom.eq}, after a tedious computation, we obtain the equations:
\begin{align}
\label{C0.cal.eq}\partial_t\mathcal{C}_0-2\text{tr}k\,\mathcal{C}_0=&\,0\qquad\Longleftrightarrow\qquad\partial_t\mathcal{C}_0+[6H+\mathcal{O}(e^{-2Ht})]\mathcal{C}_0=0,\\
\label{C1.cal.eq}\partial_t\mathcal{C}_1-(\text{tr}\widetilde{k}+\widetilde{k}_{11})\mathcal{C}_1=&\,0\qquad\Longleftrightarrow \qquad\partial_t\mathcal{C}_1+[4H+\mathcal{O}(e^{-2Ht})]\mathcal{C}_1=0.
\end{align}
If $\mathcal{C}_\mu^{\infty,m}e^{-mHt}$ are the terms in the expansions of $\mathcal{C}_\mu$, $\mu=0,1$, then it is evident from the homogeneous equations \eqref{C0.cal.eq}-\eqref{C1.cal.eq} that $\mathcal{C}_\mu=0$ everywhere, provided that $\mathcal{C}_0^{\infty,6}=\mathcal{C}_1^{\infty,4}=0$. The latter conditions are in fact equivalent to \eqref{hom.init.const}. This is indeed the case by expanding each term in the identities \eqref{C0.cal}-\eqref{C1.cal} using \eqref{hom.exp}. Hence, the constraints \eqref{Hamconst.hom}-\eqref{momconst.hom} are valid for all $t\in[T,+\infty)$.

The homogeneous solution $\widetilde{g},\widetilde{u},\widetilde{\rho}$ to the Einstein-Euler system of the form \eqref{metric.hom}-\eqref{rho.hom} is defined such that $G_I,G$ satisfy \eqref{tilde:e=G} and $\widetilde{u}=\widetilde{\rho}^{-r_s}\widetilde{v}$. Then one straightforwardly shows that the variables $\widetilde{k}_{IJ},\widetilde{\gamma}_{IJB}$ given by the solution to \eqref{k11.hom.eq}-\eqref{gamma312.hom.eq} are in fact the connection coefficients of $\widetilde{g}$ by observing that they satisfy the same equations with same initial conditions at infinity.
The expansions \eqref{hom.exp} imply the estimates in Lemma \ref{lem:hom.est}, since the remainders $\mathcal{O}(e^{Bt})$ satisfy $|\partial_t^N\mathcal{O}(e^{Bt})|\leq C_{N,B}e^{Bt}$, an immediate consequence of the equations \eqref{k11.hom.eq}-\eqref{rho.hom.eq}. 
\end{proof}

\section[The stability problem on a fixed $\mathbb{S}^3$--FLRW background]{Future stability of perfect fluids with extreme tilt and linear equation of state $p=c_s^2\rho$ for the relativistic Euler equations on a fixed $\mathbb{S}^3$--FLRW background: The full range $1/\sqrt{3}<c_s<1$}\label{app:Euler}

Here, we assume that the metric $g$ is non-dynamical, everywhere equal to the FLRW metric 
\begin{align}\label{FLRW.g}
g=-dt^2+a^2(t)g_{\mathbb{S}^3}=-dt^2+a^2(t)\sum_{i=1}^3\psi^i\otimes\psi^i,
\end{align}
where 
\begin{align}\label{FLRW.a}
a^2(t)=(a^\infty)^2e^{2Ht}+\mathcal{O}(1),\qquad\frac{a'}{a}=H+\mathcal{O}(e^{-2Ht}),\qquad\text{as $t\rightarrow+\infty$}.
\end{align}
The variables $\rho,v$ solve the relativistic Euler equations, ie. \eqref{divT.v}-\eqref{divT.rho} with $g$ as in \eqref{FLRW.g}. It is easy to see that a similar ODE analysis from infinity as in Lemma \ref{lem:app.hom}, only simpler since there is no coupling to Einstein, yields a homogeneous solution to the Euler equations with the same qualitative properties for the fluid. Indeed, the homogeneous Euler equations \eqref{vI.hom.eq}-\eqref{rho.hom.eq} reduce to
\begin{align}
\label{tilde.v1.rho.eq}\partial_t[a(t)\widetilde{v}_1]=0,\qquad
\partial_t[\widetilde{\rho}^{1-2r_s}a^3(t)\widetilde{v}_0]=0,
\end{align}
where $\widetilde{v}_0^2=\widetilde{v}_1^2+\widetilde{\rho}^{2r_s}$. Solving \eqref{tilde.v1.rho.eq} from $t=+\infty$ to some $t=T\gg1$, we construct solutions having the asymptotic behavior 
\begin{align}\label{tilde.v.rho.exp}
\begin{split}
\widetilde{v}_\mu=\widetilde{v}_\mu^{\infty}e^{-Ht}+\mathcal{O}(e^{-3Ht}),\;\mu=0,1,\qquad
\widetilde{\rho}^{1-2r_s}=(\widetilde{\rho}^\infty)^{1-2r_s}e^{-2Ht}+\mathcal{O}(e^{-4Ht}),\\
r_s=\frac{c_s^2}{1+c_s^2}\in(\frac{1}{2},\frac{1}{4}),\qquad c_s^2\in(\frac{1}{3},1),\qquad (\widetilde{v}_0^\infty)^2=(\widetilde{v}_1^\infty)^2,
\end{split}
\end{align}
where $\widetilde{v}_1^\infty,\widetilde{\rho}^\infty$ are constants that correspond to the initial data of $\widetilde{v}_1,\widetilde{\rho}$ at $t=+\infty$.

The perturbed problem now concerns only the fluid variables $\rho,v$ and the equations \eqref{v0.eq}-\eqref{rho.eq} take the following form.
\begin{lemma}\label{lem:red.Euler}
The fluid variables $v_0,v_I,\rho^{2r_s}$ satisfy the equations: 
\begin{align}
\label{v0.Euler.eq}\partial_t v_0+\frac{a'}{a}v_0-\frac{v_C}{v_0}e_Cv_0=&\,\frac{1}{2}\frac{1}{v_0}\partial_t(\rho^{2r_s})+\frac{a'}{a}\frac{\rho^{2r_s}}{v_0},\\
\label{vI.Euler.eq}\partial_t v_I+\frac{a'}{a}v_I-\frac{v_C}{v_0}e_Cv_I=&\,\frac{1}{2}\frac{1}{v_0}e_I(\rho^{2r_s}),\\
\label{rho.Euler.eq}(1-2r_s)\partial_t (\rho^{2r_s})+3\frac{a'}{a}\rho^{2r_s}
-\frac{v_C}{v_0}e_C(\rho^{2r_s})=&-\frac{\rho^{2r_s}}{v_0}\partial_tv_0+\frac{\rho^{2r_s}}{v_0}e_Cv_C,
\end{align}
where $e_0=\partial_t$, $e_I=a^{-1}Y_I$.
\end{lemma}
\begin{proof}
They follow from the equations \eqref{v0.eq}-\eqref{rho.eq} in Lemma \ref{lem:red.eq}, after setting $n=1$, $k_{IJ}=-\delta_{IJ}\frac{a'}{a}$, and realizing that 
\begin{align}\label{Euler.Killing.term}
\gamma_{CDI}v_Cv_D=-\gamma_{CID}v_Cv_D=-g(\nabla_{\overline{v}}e_I,\overline{v})=0,\qquad \overline{v}:=v_Ie_I,\qquad\gamma_{CDC}=0,
\end{align}
since $e_I=a^{-1}Y_I$, where $Y_I$ is $g_{\mathbb{S}^3}$-Killing. 
\end{proof}
By making the same bootstrap assumptions as in \eqref{Boots}, albeit only for $\widehat{v}_\mu,\widehat{\rho^{2r_s}}$ since the geometric variables are fixed to their FLRW values, we notice that all the estimates derived in Section \ref{sec:FSE} are also valid for the solution to \eqref{v0.Euler.eq}-\eqref{rho.Euler.eq}. In order to show that in the present case we can treat the full range of sound speeds beyond radiation, ie. $c_s^2\in(\frac{1}{3},1)$, we need only improve the top order energy estimate \eqref{fluid.top.est} as to eliminate the term with integrand coefficient $e^{-H\tau}e^{2A_sH\tau}$, which prevents us from applying Gronwall's inequality when $2A_s\ge1$, see Remark \ref{rem:top.fluid.est}. Indeed, we will show that this is possible due to the simplified form of the equations \eqref{v0.Euler.eq}-\eqref{rho.Euler.eq} relative to \eqref{v0.eq}-\eqref{rho.eq}, importantly, thanks to the fact that all terms in the above RHSs contain a factor of order $\rho^{2r_s}$. Notice that this is not the case for example in the equation \eqref{vI.eq}, e.g. due to the term $\gamma_{CDI}v_Cv_D$. No obvious cancellations in the top order energy estimates seem to be possible because of such terms, cf. \eqref{Euler.Killing.term}. 

We proceed by considering the differentiated analogues of \eqref{v0.Euler.eq}-\eqref{rho.Euler.eq}, commuting the equations with $Y^\iota$, $|\iota|\leq N$:
\begin{align}
\label{v0.Euler.diff.eq}\partial_t Y^\iota\widehat{v}_0+\frac{a'}{a}Y^\iota\widehat{v}_0-\frac{v_C}{v_0}e_CY^\iota\widehat{v}_0=&\,\frac{1}{2}\frac{1}{v_0}\partial_tY^\iota\widehat{\rho^{2r_s}}
-\frac{1}{2}\frac{\partial_t\rho^{2r_s}}{v_0^2}Y^\iota\widehat{v}_0+\frac{a'}{a}\frac{Y^\iota\widehat{\rho^{2r_s}}}{v_0}-\frac{a'}{a}\frac{\rho^{2r_s}}{v_0^2}Y^\iota\widehat{v}_0+\mathfrak{R}_0^\iota,\\
\label{vI.Euler.diff.eq}\partial_t Y^\iota\widehat{v}_I+\frac{a'}{a}Y^\iota\widehat{v}_I-\frac{v_C}{v_0}e_CY^\iota\widehat{v}_I=&\,\frac{1}{2}\frac{1}{v_0}e_IY^\iota\widehat{\rho^{2r_s}}+\mathfrak{R}_I^\iota,
\end{align}
and
\begin{align}
\label{rho.Euler.diff.eq}&(1-2r_s)\partial_t Y^\iota\widehat{\rho^{2r_s}}+(3\frac{a'}{a}+\frac{\partial_tv_0}{v_0})Y^\iota\widehat{\rho^{2r_s}}
-(1-2r_s)\frac{v_C}{v_0}e_CY^\iota\widehat{\rho^{2r_s}}\\
\notag=&\,\frac{\rho^{2r_s}}{v_0}e_CY^\iota\widehat{v}_C
-\frac{\rho^{2r_s}}{v_0}\partial_tY^\iota\widehat{v}_0
+\frac{\rho^{2r_s}}{v_0}\frac{\partial_tv_0}{v_0}Y^\iota\widehat{v}_0
+\mathfrak{R}^\iota.
\end{align}
All lower order terms from the previous commutations are included in $\mathfrak{R}_0^\iota,\mathfrak{R}_I^\iota,\mathfrak{R}$. Moreover, they do not contribute to the dangerous terms containing $e^{-Ht}e^{2A_sHt}$ coefficients in the final energy inequalities and can be treated as in the coupled to Einstein case. 

The analogous top order energy identity to \eqref{top.fluid.id}, for the system \eqref{v0.Euler.diff.eq}-\eqref{rho.Euler.diff.eq}, is the following: 
\begin{align}\label{top.Euler.id}
-e^{4Ht}\rho^{2r_s}Y^\iota \widehat{v}_0\times\eqref{v0.Euler.diff.eq}
+e^{4Ht}\rho^{2r_s}Y^\iota \widehat{v}_I\times\eqref{vI.Euler.diff.eq}
+\frac{1}{2}e^{4Ht}Y^\iota\widehat{\rho^{2r_s}}\times\eqref{rho.Euler.diff.eq},
\end{align}
which expands to
\begin{align}\label{top.Euler.id2}
\notag&\frac{1}{2}\partial_t\big\{
e^{4Ht}\rho^{2r_s}Y^\iota \widehat{v}_I Y^\iota \widehat{v}_I
-e^{4Ht}\rho^{2r_s}(Y^\iota \widehat{v}_0)^2
+\frac{1}{2}(1-2r_s)e^{4Ht}(Y^\iota\widehat{\rho^{2r_s}})^2\big\}\\
\notag&-\frac{1}{2}e^{2Ht}\partial_t(e^{2Ht}\rho^{2r_s})
Y^\iota \widehat{v}_I Y^\iota \widehat{v}_I
+\frac{1}{2}e^{2Ht}\partial_t(e^{2Ht}\rho^{2r_s})
(Y^\iota \widehat{v}_0)^2
+(\frac{a'}{a}-H)e^{4Ht}\rho^{2r_s}Y^\iota \widehat{v}_I Y^\iota \widehat{v}_I\\
\notag&-(\frac{a'}{a}-H)e^{4Ht}\rho^{2r_s}(Y^\iota \widehat{v}_0)^2
+[\frac{1}{2}(3\frac{a'}{a}+\frac{\partial_tv_0}{v_0})-(1-2r_s)H]e^{4Ht}(Y^\iota\widehat{\rho^{2r_s}})^2
\\
=&\,\frac{1}{2}e^{4Ht}\frac{\rho^{2r_s}}{v_0}e_I(Y^\iota\widehat{\rho^{2r_s}}Y^\iota\widehat{v}_I)
-\frac{1}{2}e^{4Ht}\frac{\rho^{2r_s}}{v_0}\partial_t(Y^\iota\widehat{\rho^{2r_s}}Y^\iota\widehat{v}_0)\\
\notag&+\frac{1}{2}e^{4Ht}\rho^{2r_s}\frac{v_C}{v_0}e_C(Y^\iota\widehat{v}_IY^\iota\widehat{v}_I)-\frac{1}{2}e^{4Ht}\rho^{2r_s}\frac{v_C}{v_0}e_C[(Y^\iota\widehat{v}_0)]^2
+\frac{1}{4}e^{4Ht}(1-2r_s)\frac{v_C}{v_0}e_C[(Y^\iota \widehat{\rho^{2r_s}})^2]\\
\notag&+e^{4Ht}\rho^{2r_s}Y^\iota \widehat{v}_0\big\{\frac{1}{2}\frac{\partial_t\rho^{2r_s}}{v_0^2}Y^\iota\widehat{v}_0-\frac{a'}{a}\frac{Y^\iota\widehat{\rho^{2r_s}}}{v_0}+\frac{a'}{a}\frac{\rho^{2r_s}}{v_0^2}Y^\iota\widehat{v}_0\big\}+\frac{1}{2}e^{4Ht}\rho^{2r_s}Y^\iota\widehat{\rho^{2r_s}}\frac{\rho^{2r_s}}{v_0}\frac{\partial_tv_0}{v_0}Y^\iota\widehat{v}_0\\
\notag&+e^{4Ht}\rho^{2r_s}Y^\iota\widehat{v}_I\mathfrak{R}_I^\iota
-e^{4Ht}\rho^{2r_s}Y^\iota\widehat{v}_0\mathfrak{R}_0^\iota
+\frac{1}{2}e^{4Ht}(Y^\iota\widehat{\rho^{2r_s}})\mathfrak{R}^\iota.
\end{align}
Differentiating by parts in $\partial_t$ in the term $-\frac{1}{2}e^{4Ht}\frac{\rho^{2r_s}}{v_0}\partial_t(Y^\iota\widehat{\rho^{2r_s}}Y^\iota\widehat{v}_0)$ and rearranging some terms in the RHS gives
\begin{align}\label{top.Euler.id3}
\notag&\frac{1}{2}\partial_t\big\{
e^{4Ht}\rho^{2r_s}\big[Y^\iota \widehat{v}_I Y^\iota \widehat{v}_I
+\frac{1}{v_0}Y^\iota\widehat{\rho^{2r_s}}Y^\iota\widehat{v}_0
-(Y^\iota \widehat{v}_0)^2\big]\big\}
+\frac{1}{4}(1-2r_s)\partial_t\big[e^{4Ht}(Y^\iota\widehat{\rho^{2r_s}})^2\big]\\
\notag&-\frac{1}{2}e^{2Ht}\partial_t(e^{2Ht}\rho^{2r_s})\big\{
Y^\iota \widehat{v}_I Y^\iota \widehat{v}_I
+\frac{1}{v_0}Y^\iota\widehat{\rho^{2r_s}}Y^\iota\widehat{v}_0
-(Y^\iota \widehat{v}_0)^2\big\}\\
\notag&+(\frac{a'}{a}-H)e^{4Ht}\rho^{2r_s}Y^\iota \widehat{v}_I Y^\iota \widehat{v}_I
-(\frac{a'}{a}-H)e^{4Ht}\rho^{2r_s}(Y^\iota \widehat{v}_0)^2\\
\notag&+[\frac{1}{2}(3\frac{a'}{a}+\frac{\partial_tv_0}{v_0})-(1-2r_s)H]e^{4Ht}(Y^\iota\widehat{\rho^{2r_s}})^2
-(\frac{1}{2}\frac{\partial_t\rho^{2r_s}}{v_0^2}+\frac{a'}{a}\frac{\rho^{2r_s}}{v_0^2})e^{4Ht}\rho^{2r_s}(Y^\iota \widehat{v}_0)^2
\\
=&\,\frac{1}{2}e^{4Ht}\frac{\rho^{2r_s}}{v_0}e_I(Y^\iota\widehat{\rho^{2r_s}}Y^\iota\widehat{v}_I)
+(H-\frac{a'}{a})e^{4Ht}\frac{\rho^{2r_s}}{v_0}Y^\iota\widehat{\rho^{2r_s}}Y^\iota\widehat{v}_0\\
\notag&+\frac{1}{2}e^{4Ht}\rho^{2r_s}\frac{v_C}{v_0}e_C(Y^\iota\widehat{v}_IY^\iota\widehat{v}_I)-\frac{1}{2}e^{4Ht}\rho^{2r_s}\frac{v_C}{v_0}e_C[(Y^\iota\widehat{v}_0)]^2
+\frac{1}{4}e^{4Ht}(1-2r_s)\frac{v_C}{v_0}e_C[(Y^\iota \widehat{\rho^{2r_s}})^2]\\
\notag&+e^{4Ht}\rho^{2r_s}Y^\iota\widehat{v}_I\mathfrak{R}_I^\iota
-e^{4Ht}\rho^{2r_s}Y^\iota\widehat{v}_0\mathfrak{R}_0^\iota
+\frac{1}{2}e^{4Ht}(Y^\iota\widehat{\rho^{2r_s}})\mathfrak{R}^\iota.
\end{align}
After the previous manipulations, it turns out that all terms in the RHS of \eqref{top.Euler.id3} can be bounded in $L^1(\Sigma_t)$ by the total energy times a uniformly (in $c_s^2$) decaying coefficient, after integrating by parts in $e_I,e_C$, and all terms in the LHS which are exactly at the level of the total energy have a favorable sign. Indeed, recall the asymptotic behaviors (Lemma \ref{lem:Linfty.est})
\begin{align*}
\|\frac{\partial_tv_0}{v_0}+H\|_{L^\infty(\Sigma_t)}\leq Ce^{-Ht},\qquad \|\frac{\partial_t(\rho^{2r_s})}{\rho^{2r_s}}+\frac{4r_s}{1-2r_s}H\|_{L^\infty(\Sigma_t)}\leq Ce^{2Ht-\frac{4r_s}{1-2r_s}Ht},
\end{align*}
integrate \eqref{top.Euler.id3} in $\Sigma_t$ and sum in $|\iota|\leq N$ to deduce the energy inequality 
\begin{align}\label{top.Euler.ineq}
\notag&\frac{1}{2}\partial_t\sum_{|\iota|\leq N}\bigg[\int_{\Sigma_t}
e^{4Ht}\rho^{2r_s}\big[Y^\iota \widehat{v}_I Y^\iota \widehat{v}_I
+\frac{1}{v_0}Y^\iota\widehat{\rho^{2r_s}}Y^\iota\widehat{v}_0
-(Y^\iota \widehat{v}_0)^2\big]d\omega
+\frac{1}{2}(1-2r_s)e^{4Ht}\|Y^\iota\widehat{\rho^{2r_s}}\|^2_{\dot{H}^N(\Sigma_t)}\bigg]\\
\notag&+\sum_{|\iota|\leq N}\int_{\Sigma_t}(\frac{2r_s}{1-2r_s}-1)He^{4Ht}\rho^{2r_s}\big\{
Y^\iota \widehat{v}_I Y^\iota \widehat{v}_I
+\frac{1}{v_0}Y^\iota\widehat{\rho^{2r_s}}Y^\iota\widehat{v}_0
-(Y^\iota \widehat{v}_0)^2\big\}d\omega\\
\notag&+2r_sHe^{4Ht}\|\widehat{\rho^{2r_s}}\|_{H^N(\Sigma_t)}^2
+\int_{\Sigma_t}(\frac{2r_s}{1-2r_s}-1)He^{4Ht}\frac{\rho^{4r_s}}{v_0^2}(Y^\iota \widehat{v}_0)^2d\omega
\\
\leq&\,Ce^{4Ht}e^{-\frac{4r_s}{1-2r_s}Ht}\|\widehat{\rho^{2r_s}}\|_{H^N(\Sigma_t)}\|\widehat{v}\|_{H^N(\Sigma_t)}\\
\notag&+\sum_{|\iota|\leq N}\int_{\Sigma_t}Ce^{-Ht}e^{4Ht}\rho^{2r_s}\big|Y^\iota\widehat{v}_IY^\iota\widehat{v}_I-(Y^\iota\widehat{v}_0)^2\big|d\omega
+Ce^{-Ht}e^{4Ht}\|\widehat{\rho^{2r_s}}\|_{H^N(\Sigma_t)}^2\\
\notag&+\sum_{|\iota|\leq N}\int_{\Sigma_t}\big\{e^{4Ht}\rho^{2r_s}Y^\iota\widehat{v}_I\mathfrak{R}_I^\iota
-e^{4Ht}\rho^{2r_s}Y^\iota\widehat{v}_0\mathfrak{R}_0^\iota
+\frac{1}{2}e^{4Ht}(Y^\iota\widehat{\rho^{2r_s}})\mathfrak{R}^\iota\big\} d\omega.
\end{align}
The terms in the last line, containing error terms from the commutation of equations \eqref{v0.Euler.eq}-\eqref{rho.Euler.eq} with $Y^\iota$, can be treated similarly to \eqref{error.fluid.top.est2}, since they all contain factors of order $\rho^{2r_s}$ in terms of decay in powers of $e^{Ht}$.

Moreover, we recall the algebraic relation
\begin{align}\label{Euler.mixed.error}
4-\frac{4r_s}{1-2r_s}=2+(1-2A_s)-1\quad\Rightarrow\quad e^{4Ht}e^{-\frac{4r_s}{1-2r_s}Ht}\|\widehat{\rho^{2r_s}}\|_{H^N(\Sigma_t)}\|\widehat{v}\|_{H^N(\Sigma_t)}\leq e^{-Ht}\mathcal{E}_{tot}(t)
\end{align}
to infer from \eqref{top.Euler.ineq}
\begin{align}\label{top.Euler.ineq2}
\notag&\frac{1}{2}\partial_t\sum_{|\iota|\leq N}\bigg[\int_{\Sigma_t}
e^{4Ht}\rho^{2r_s}\big[Y^\iota \widehat{v}_I Y^\iota \widehat{v}_I
+\frac{1}{v_0}Y^\iota\widehat{\rho^{2r_s}}Y^\iota\widehat{v}_0
-(Y^\iota \widehat{v}_0)^2\big]d\omega
+\frac{1}{2}(1-2r_s)e^{4Ht}\|Y^\iota\widehat{\rho^{2r_s}}\|^2_{\dot{H}^N(\Sigma_t)}\bigg]\\
&+\sum_{|\iota|\leq N}\int_{\Sigma_t}(\frac{2r_s}{1-2r_s}-1)He^{4Ht}\rho^{2r_s}\big\{
Y^\iota \widehat{v}_I Y^\iota \widehat{v}_I
+\frac{1}{v_0}Y^\iota\widehat{\rho^{2r_s}}Y^\iota\widehat{v}_0
-(Y^\iota \widehat{v}_0)^2\big\}d\omega\\
\notag&+2r_sHe^{4Ht}\|\widehat{\rho^{2r_s}}\|^2_{H^N(\Sigma_t)}
+(\frac{2r_s}{1-2r_s}-1)He^{4Ht}\frac{\rho^{4r_s}}{v_0^2}\|Y^\iota \widehat{v}_0\|^2_{H^N(\Sigma_t)}
\\
\notag\leq&\,C(e^{-Ht}+e^{2Ht-\frac{4r_s}{1-2r_s}Ht})\mathcal{E}_{tot}(t)+\sum_{|\iota|\leq N}\int_{\Sigma_t}Ce^{-Ht}e^{4Ht}\rho^{2r_s}\big|Y^\iota\widehat{v}_IY^\iota\widehat{v}_I-(Y^\iota\widehat{v}_0)^2\big|d\omega.
\end{align}
\begin{lemma}\label{lem:Euler.v.en}
The following inequality holds:
\begin{align}\label{Euler.v.en}
\begin{split}
&\int_{\Sigma_t}e^{4Ht}\rho^{2r_s}\big[Y^\iota \widehat{v}_I Y^\iota \widehat{v}_I
+\frac{1}{v_0}Y^\iota\widehat{\rho^{2r_s}}Y^\iota\widehat{v}_0
-(Y^\iota \widehat{v}_0)^2\big]d\omega\\
\ge&\, \int_{\Sigma_t}e^{4Ht}\frac{\rho^{4r_s}}{v_0^2}Y^\iota \widehat{v}_I Y^\iota \widehat{v}_Id\omega-C(\varepsilon+e^{2Ht-\frac{4r_s}{1-2r_s}Ht})\mathcal{E}_{tot}(t)
\end{split}
\end{align}
for all $t\in[T,T_{Boot})$.
\end{lemma}
\begin{proof}
We differentiate the identity $v_0^2=v_Iv_I+\rho^{2r_s}$ in $Y^\iota$, $|\iota|\leq N$:
\begin{align}\label{v02.diff.id}
v_0Y^\iota\widehat{v}_0=v_IY^\iota\widehat{v}_I+\frac{1}{2}Y^\iota\widehat{\rho^{2r_s}}+\mathfrak{I}^\iota
\end{align}
where $\mathfrak{I}^\iota$ satisfies the bound
\begin{align}\label{frak.I.est}
e^{2Ht}\|\mathfrak{I}^\iota\|_{L^2(\Sigma_t)}\leq C\varepsilon\sqrt{\mathcal{E}_{tot}(t)}.
\end{align}
Taking $\frac{1}{2}Y^\iota\widehat{\rho^{2r_s}}$ to the LHS in \eqref{v02.diff.id}, diving by $v_0$ and squaring yields
\begin{align}\label{v02.diff.id2}
(Y^\iota\widehat{v}_0)^2-\frac{1}{v_0}Y^\iota\widehat{\rho^{2r_s}}Y^\iota\widehat{v}_0+\frac{1}{4}\frac{1}{v_0^2}(Y^\iota\widehat{\rho^{2r_s}})^2=\frac{v_Iv_C}{v_0^2}Y^\iota\widehat{v}_IY^\iota\widehat{v}_C+\frac{2}{v_0^2}v_IY^\iota v_I\mathfrak{I}^\iota+\frac{1}{v_0^2}(\mathfrak{I}^\iota)^2.
\end{align}
Hence, we have 
\begin{align}\label{Euler.v.en2}
\notag&\int_{\Sigma_t}e^{4Ht}\rho^{2r_s}\big[Y^\iota \widehat{v}_I Y^\iota \widehat{v}_I
+\frac{1}{v_0}Y^\iota\widehat{\rho^{2r_s}}Y^\iota\widehat{v}_0
-(Y^\iota \widehat{v}_0)^2\big]d\omega\\
=&\int_{\Sigma_t}\big\{e^{4Ht}\rho^{2r_s}\big[Y^\iota \widehat{v}_I Y^\iota \widehat{v}_I
-\frac{v_Iv_C}{v_0^2}Y^\iota\widehat{v}_IY^\iota\widehat{v}_C\big]
+\frac{1}{4}\frac{\rho^{2r_s}}{v_0^2}e^{4Ht}(Y^\iota\widehat{\rho^{2r_s}})^2\big\}d\omega\\
\notag&-\int_{\Sigma_t}e^{4Ht}\rho^{2r_s}\big[\frac{2}{v_0^2}v_IY^\iota v_I\mathfrak{I}^\iota+\frac{1}{v_0^2}(\mathfrak{I}^\iota)^2\big]d\omega\\
\tag{by Cauchy-Schwarz and \eqref{frak.I.est}}\geq&\,\int_{\Sigma_t}e^{4Ht}\frac{\rho^{4r_s}}{v_0^2}Y^\iota \widehat{v}_I Y^\iota \widehat{v}_Id\omega-C(\varepsilon+e^{2Ht-\frac{4r_s}{1-2r_s}Ht})\sqrt{\mathcal{E}_{tot}(t)}
\end{align}
as desired.
\end{proof}
Next, we apply Lemma \ref{lem:Euler.v.en} to \eqref{top.Euler.ineq2} and recall that $e^{4Ht}\frac{\rho^{4r_s}}{v_0^2}\sim e^{-4A_sHt}e^{2Ht}$ to deduce the energy inequality
\begin{align}\label{top.Euler.ineq3}
\notag&e^{-4A_sHt}e^{2Ht}\|\widehat{v}\|^2_{H^N(\Sigma_t)}+e^{4Ht}\|Y^\iota\widehat{\rho^{2r_s}}\|^2_{H^N(\Sigma_t)}\\
&+\int^t_Te^{-4A_sH\tau}e^{2H\tau}\|\widehat{v}\|^2_{H^N(\Sigma_\tau)}+e^{4H\tau}\|Y^\iota\widehat{\rho^{2r_s}}\|^2_{H^N(\Sigma_\tau)}d\tau
\\
\notag\leq&\,C\mathcal{E}_{tot}(T)+C(\varepsilon+e^{2Ht-\frac{4r_s}{1-2r_s}Ht})\mathcal{E}_{tot}(t)+\int^t_TC(\varepsilon+e^{-H\tau}+e^{2H\tau-\frac{4r_s}{1-2r_s}H\tau})\mathcal{E}_{tot}(\tau)d\tau
\end{align}
for all $t\in[T,T_{Boot})$. Taking $\varepsilon$ sufficiently small and $T$ sufficiently large allows us to absorb the last two terms in the LHS. Using as well the lower order estimate \eqref{fluid.low.est}, we obtain the inequality
\begin{align}\label{Euler.tot.en.est}
\mathcal{E}_{tot}(t)\leq C\mathcal{E}_{tot}(T)+C\int^t_T(e^{-H\tau}+e^{2H\tau-\frac{4r_s}{1-2r_s}H\tau})\mathcal{E}_{tot}(\tau)d\tau.
\end{align}
Thus, Gronwall's inequality applies for all $c_s^2\in(\frac{1}{3},1)$, giving a uniform bound of the total energy, $\mathcal{E}_{tot}(t)\leq C\mathcal{E}_{tot}(T)$, which can be used to close the bootstrap argument and conclude the future stability of the background homogeneous solutions to Euler's equations with extreme tilt, for all sound speeds beyond radiation.


\begin{thebibliography}{99}

\bibitem{And} M. T. Anderson,
\textit{Existence and stability of even-dimensional asymptotically de
{S}itter spaces}, Ann. Henri Poincar\'e {\bf 6} (2005), no. 5, 801--820.

\bibitem{BMO} F. Beyer, E. Marshall and T. A. Oliynyk,
{\it Future instability of FLRW fluid solutions for linear equations of state $p=K\rho$ with $1/3<K<1$}, Phys. Rev. D {\bf 107}, 104030.

\bibitem{BRR} U. Brauer, A. D. Rendall and O. Reula,
\textit{The cosmic no-hair theorem and the non-linear stability of
homogeneous Newtonian cosmological models}, Class. Quantum Grav. {\bf 11} (1994), no. 9, 2283--2296.

\bibitem{Christ} D. Christodoulou, 
\textit{The Formation of Shocks in 3-Dimensional Fluids}, 
EMS Monographs in Mathematics, Z\"urich, 2007.

\bibitem{CW} A. A. Coley and J. Wainwright, 
{\it Qualitative analysis of two-fluid Bianchi cosmologies}, Class. Quantum Grav. {\bf 9} (1992), 651-665.

\bibitem{FK} D. Fajman and K. Kr\"oncke,
\textit{Stable fixed points of the {E}instein flow with positive
cosmological constant}, Comm. Anal. Geom. {\bf 28} (2020), no. 7,
1533--1576.

\bibitem{F} G. Fournodavlos,
{\it Future dynamics of FLRW for the massless-scalar field system with positive cosmological constant}, J. Math. Phys. {\bf 63} (2022), no. 3, 032502, 23pp.

\bibitem{Fr} H. Friedrich, 
\textit{On the existence of n-geodesically complete or future complete solutions of Einstein's field equations with smooth asymptotic structure},
Comm. Math. Phys. {\bf 107} (1986), no. 4, 587-609.

\bibitem{GN} M. Goliath and U. S. Nilsson, 
{\it Isotropization of two-component fluids}, 
J. Math. Phys. {\bf 41} (2000), 6906-6917.

\bibitem{GE} M. Goliath and G. F. R. Ellis, 
{\it Homogeneous cosmologies with a cosmological constant}, 
Phys. Rev. D {\bf 60} (1999), 023502. 




\bibitem{HS}
M. Had\v zi\'c and Jared Speck, 
\textit{The global future stability of the FLRW solutions
to the dust-Einstein system with a positive cosmological constant}, 
J. Hyperbolic Differ. Equ. {\bf 12} (2015), no. 1, 87-188.


\bibitem{LEUW} W. C. Lim, H. van Elst, C. Uggla and J. Wainwright,
{\it Asymptotic isotropization in inhomogeneous cosmology},
Phys. Rev. D (3) {\bf 69} (2004), no. 10, 103507, 22 pp.

\bibitem{LK} C. L\"ubbe, J. A. Valiente Kroon,
\textit{A conformal approach for the analysis of the non-linear stability of radiation cosmologies}, Ann. Physics {\bf 328} (2013), 1-25. 

\bibitem{MMBE} M. S. Madsen, J. P. Mimoso, J. A. Butcher and G. F. R. Ellis, 
{\it Evolution of the density parameter in inflationary cosmology re-examined},
Phys. Rev. D. {\bf 46} (1992), 1399-1415. 

\bibitem{ME} M. S. Madsen and G. F. R. Ellis,
{\it The evolution of Omega in inflationary universes}, 
Mon. Not. R. Astron. Soc. {\bf 234} (1988), 67-77.

\bibitem{MO} E. Marshall and T. A. Oliynyk,
\textit{On the stability of relativistic perfect fluids with linear equations of state
$p=K\rho$ where $1/3<K<1$}, Lett. Math. Phys. {\bf 113} (2023), 102.

\bibitem{MV} M. Minucci and J. A. Valiente Kroon,
\textit{A conformal approach to the stability of {E}instein spaces
with spatial sections of negative scalar curvature}, Class. Quant. Grav. {\bf 38} (2021), no. 14, 40 pp.

\bibitem{M} P. Mondal,
{\it The nonlinear stability of n+1 dimensional FLRW spacetimes}, arXiv:2203.04785.

\bibitem{Ol1} T. A. Oliynyk, 
\textit{Future stability of the FLRW fluid solutions in the presence of a positive cosmological constant}, Comm. Math. Phys. {\bf 346} (2016), no. 1, 293-312.

\bibitem{Ol2} T. A. Oliynyk,
\textit{Future global stability for relativistic perfect fluids with linear equations of state $p=K\rho$ where $1/3<K<1/2$}, SIAM J. Math. Anal. {\bf 53} (2021), no. 4, 4118-4141.

\bibitem{Ol3} T. A. Oliynyk
{\it On the fractional density gradient blow-up conjecture of Rendall}, 
arXiv:2310.19184.

\bibitem{RM} A. K. Raychaudhuri and B. Modak, 
{\it Cosmological inflation with arbitrary initial conditions},
Class. Quantum Grav. {\bf 5} (1988), 225. 

\bibitem{Ren} A. D. Rendall, 
\textit{Asymptotics of solutions of the Einstein equations with positive cosmological constant}, Ann. Henri Poincar\'e {\bf 5} (2004), 1041-1064.

\bibitem{Rin} H. Ringstr\"om, 
\textit{Future stability of the Einstein-non-linear scalar field system}, Invent. Math. {\bf 173} (2008), no. 1, 123-208.

\bibitem{RS}
I. Rodnianski and J. Speck, 
\textit{The nonlinear future stability of the FLRW
family of solutions to the irrotational Euler-Einstein system with a positive
cosmological constant}, J. Eur. Math. Soc. {\bf 15} (2013), no. 6, 2369-2462.

\bibitem{Sp1} J. Speck, 
\textit{The nonlinear future stability of the FLRW family of solutions
to the Euler-Einstein system with a positive cosmological constant}, 
Selecta Math. (N.S.) {\bf 18} (2012), no. 3, 633-715.

\bibitem{SR} R. Stabell and S. Refsdal, 
{\it Classification of General Relativistic World Models},
Mon. Not. R. Astron. Soc. {\bf 132} (1996), no. 2, 379-388.

\bibitem{W} R. M. Wald, 
{\it Asymptotic behaviour of homogeneous cosmological
models in the presence of a positive cosmological constant}, 
Phys. Rev. D {\bf 28} (1983), 2118-2120. 


\end{thebibliography}
\end{document}